\documentclass[a4paper, reqno]{amsart}
\newif\ifprivate
\privatefalse

\usepackage[utf8]{inputenc}
\usepackage[T1]{fontenc}
\usepackage{lmodern}
\usepackage[draft=false,kerning=true]{microtype}

\usepackage{enumerate}
\usepackage{mathtools}
\usepackage{textcomp}
\usepackage{tikz-cd}

\usepackage{hyperref}

\ifprivate
\usepackage[mark]{gitinfo2}
\renewcommand{\gitMark}{\jobname\,\textbullet{}\,\gitFirstTagDescribe\,\textbullet{}\,\gitAuthorName,\,\gitAuthorIsoDate}
\newcommand{\TODO}[1]%
{\par\fbox{\begin{minipage}{0.9\linewidth}\textbf{TODO:} #1\end{minipage}}\par}
\fi

\newcommand{\proofparagraph}[1]{\medskip\par\noindent{\itshape #1.} }

\DeclarePairedDelimiter{\abs}{\lvert}{\rvert}

\newcommand{\bfx}{\mathbf{x}}

\newcommand{\bibinom}[3]{\binom{#1}{#2, #3}}
\newcommand{\C}{\mathbb{C}}
\newcommand{\calC}{\mathcal{C}}
\newcommand{\calD}{\mathcal{D}}
\newcommand{\calF}{\mathcal{F}}
\newcommand{\calG}{\mathcal{G}}
\newcommand{\calH}{\mathcal{H}}
\newcommand{\calI}{\mathcal{I}}
\newcommand{\calJ}{\mathcal{J}}
\newcommand{\calK}{\mathcal{K}}
\newcommand{\calL}{\mathcal{L}}
\newcommand{\calM}{\mathcal{M}}
\newcommand{\calS}{\mathcal{S}}
\newcommand{\calT}{\mathcal{T}}
\newcommand{\calV}{\mathcal{V}}
\newcommand{\calX}{\mathcal{X}}
\newcommand{\calY}{\mathcal{Y}}
\newcommand{\calZ}{\mathcal{Z}}
\DeclarePairedDelimiter{\ceil}{\lceil}{\rceil}
\newcommand{\dd}{\mathrm{d}}
\newcommand{\DLMF}[2]{\cite[\href{http://dlmf.nist.gov/#1.E#2}{#1.#2}]{NIST:DLMF:v1.0.16}}

\DeclarePairedDelimiterXPP{\f}[2]{\foperator{#1}}(){}{#2}
\DeclarePairedDelimiterXPP{\fexp}[1]{\exp}(){}{#1}
\DeclarePairedDelimiter{\floor}{\lfloor}{\rfloor}
\newcommand{\foperator}[1]{\mathop{{#1}\empty{}}}
\DeclarePairedDelimiter{\fractional}{\{}{\}}
\DeclarePairedDelimiterXPP{\inftynorm}[1]{}{\lVert}{\rVert}{_\infty}{#1}
\newcommand{\itemref}[1]{\eqref{#1}}
\DeclarePairedDelimiter{\iverson}{[}{]}
\newcommand{\kerneldim}{D}
\renewcommand{\MR}[1]{}
\newcommand{\N}{\mathbb{N}}
\DeclarePairedDelimiter{\norm}{\lVert}{\rVert}
\DeclarePairedDelimiterXPP{\Oh}[1]{\foperator{O}}(){}{#1}
\DeclarePairedDelimiterXPP{\oh}[1]{\foperator{o}}(){}{#1}

\newcommand{\R}{\mathbb{R}}
\newcommand{\repr}{\mathsf{reprq}}
\DeclarePairedDelimiterXPP{\Res}[2]{\operatorname{Res}}(){}{#1, #2}
\DeclarePairedDelimiter{\set}{\{}{\}}
\DeclarePairedDelimiterX{\setm}[2]{\{}{\}}{#1 \colon \mathopen{}#2}
\newcommand{\tildea}{\widetilde{a}}
\newcommand{\tildeA}{\widetilde{A}}
\newcommand{\tildecalC}{\widetilde{\calC}}
\newcommand{\tildeM}{\widetilde{M}}
\newcommand{\tildev}{\widetilde{v}}
\newcommand{\trinom}[4]{\binom{#1}{#2, #3, #4}}
\newcommand{\val}{\mathsf{lval}}
\newcommand{\Z}{\mathbb{Z}}

\newtheorem{theorem}{Theorem}
\newtheorem{corollary}[theorem]{Corollary}

\newtheorem{lemma}{Lemma}[section]
\newtheorem{proposition}[lemma]{Proposition}

\theoremstyle{remark}
\newtheorem{example}[lemma]{Example}
\newtheorem{remark}[lemma]{Remark}

\numberwithin{equation}{section}
\numberwithin{figure}{section}
\numberwithin{table}{section}

\begin{document}
\title[Analysis of Summatory Functions of Regular Sequences]{Analysis of Summatory Functions of Regular Sequences: Transducer and
  Pascal's Rhombus}
\author[C.~Heuberger]{Clemens Heuberger}
\address{Clemens Heuberger,
  Institut f\"ur Mathematik, Alpen-Adria-Universit\"at Klagenfurt,
  Universit\"atsstra\ss e 65--67, 9020 Klagenfurt am W\"orthersee, Austria}
\email{\href{mailto:clemens.heuberger@aau.at}{clemens.heuberger@aau.at}}

\author[D.~Krenn]{Daniel Krenn}
\address{Daniel Krenn,
  Institut f\"ur Mathematik, Alpen-Adria-Universit\"at Klagenfurt,
  Universit\"atsstra\ss e 65--67, 9020 Klagenfurt am W\"orthersee, Austria}
\email{\href{mailto:math@danielkrenn.at}{math@danielkrenn.at} \textit{or}
  \href{mailto:daniel.krenn@aau.at}{daniel.krenn@aau.at}}

\author[H.~Prodinger]{Helmut Prodinger}
\address{Helmut Prodinger, Department of Mathematical Sciences, Stellenbosch University, 7602 Stellenbosch,
 South Africa}
\email{\href{mailto:hproding@sun.ac.za}{hproding@sun.ac.za}}

\thanks{C.~Heuberger and D.~Krenn are supported by the
   Austrian Science Fund (FWF): P\,28466-N35.}

\begin{abstract}
  The summatory function of a $q$-regular sequence in the sense of Allouche and
  Shallit is analysed asymptotically. The result is a sum of periodic
  fluctuations for eigenvalues of absolute value larger than the joint spectral
  radius of the matrices of a linear representation of the sequence.
  The Fourier coefficients of the fluctuations are expressed in terms of
  residues of the corresponding Dirichlet generating function. A known pseudo
  Tauberian argument is extended in order to overcome convergence problems in
  Mellin--Perron summation.

  Two examples are discussed in more detail: The case of sequences defined as
  the sum of outputs written by a transducer when reading a $q$ary expansion of
  the input and the number of odd entries in the rows of Pascal's rhombus.
\end{abstract}

\maketitle

\section{Introduction}
In this paper, we study the asymptotic behaviour of the summatory function of
$q$-regular sequences.\footnote{
  In the standard literature
  \cite{Allouche-Shallit:1992, Allouche-Shallit:2003:autom} these sequences
  are called $k$-regular sequences (instead of $q$-regular sequences).}
Regular sequences have been introduced by Allouche and
Shallit \cite{Allouche-Shallit:1992} (see also \cite[Chapter~16]{Allouche-Shallit:2003:autom}); these are sequences which are
intimately related to the $q$-ary expansion of their arguments. Many special
cases have been investigated in the literature; our goal is to provide a single
result decomposing the summatory function into periodic fluctuations multiplied
by some scaling functions and to provide the Fourier coefficients of these
periodic fluctuations.

Note that it is well-known that the summatory function
of a $q$-regular sequence is itself $q$-regular. (This is an immediate
consequence of \cite[Theorem~3.1]{Allouche-Shallit:1992}.) Similarly, the
sequence of differences of a $q$-regular sequence is $q$-regular. Therefore, we
might also start to analyse a regular sequence by considering it to be the
summatory function of its sequence of differences.

In the remaining paper, we first recall the definition of $q$-regular sequences
in Section~\ref{section:introduction:regular-sequences}, then formulate a
somewhat simplified version of our main result in
Section~\ref{section:introduction:main-result}. In Section~\ref{sec:heuristic},
we give a heuristic non-rigorous argument to explain why the result is
expected. We outline the relation to previous work in
Section~\ref{introduction:relation-to-previous-work}.
We give two examples in Sections~\ref{sec:transducer} and \ref{sec:pascal}.
In principle, these examples
are straight-forward applications of the results, but still, we have to
reformulate the relevant questions in terms of a $q$-regular sequence and will
then provide shortcuts for the computation of the Fourier series. The first
example is generic and deals with sequences defined as the sum of outputs of
transducer automata; the second example---which motivated us to conduct this
study at this point---is a concrete problem counting the number of odd
entries in Pascal's rhombus.

The full formulation of our results and their proofs
are given in the appendix.

\subsection{\texorpdfstring{$q$}{q}-Regular Sequences}\label{section:introduction:regular-sequences}

We start by giving a definition of $q$-regular sequences, see Allouche and
Shallit~\cite{Allouche-Shallit:1992}. Let $q\ge 2$ be a fixed
integer and $(x(n))_{n\ge 0}$ be a sequence.

Then $(x(n))_{n\ge 0}$ is said to
be \emph{$(\C, q)$-regular} (briefly: \emph{$q$-regular} or simply \emph{regular}) if the $\C$-vector space
generated by its \emph{$q$-kernel}
\begin{equation*}
  \setm[\big]{\bigl(x(q^j n+r)\bigr)_{n\ge 0}}%
  {\text{integers $j\ge 0$, $0\le r<q^j$}}
\end{equation*}
has finite dimension.
In other words, $(x(n))_{n\ge 0}$ is $q$-regular if
there is an integer $\kerneldim$ and sequences $(x_1(n))_{n\ge 0}$,
\dots, $(x_\kerneldim(n))_{n\ge 0}$ such that for every $j\ge 0$ and $0\le r<q^j$
there exist integers $c_1$, \ldots, $c_\kerneldim$ such that
\begin{equation*}
  x(q^j n+r) = c_1 x_1(n) + \dotsb + c_\kerneldim x_\kerneldim(n)\qquad{\text{for all $n\ge 0$.}}
\end{equation*}

By Allouche and Shallit~\cite[Theorem~2.2]{Allouche-Shallit:1992},
$(x(n))_{n\ge 0}$ is $q$-regular if and only if there exists a vector valued
sequence $(v(n))_{n\ge 0}$ whose first component coincides with
$(x(n))_{n\ge 0}$ and there exist square matrices $A_0$, \ldots, $A_{q-1}\in\C^{d\times d}$ such that
\begin{equation}\label{eq:linear-representation}
  v(qn+r) = A_r v(n)\qquad\text{for $0\le r<q$, $n\ge 0$.}
\end{equation}
This is called a \emph{$q$-linear representation} of $x(n)$.

The best-known example for a $2$-regular function is the binary sum-of-digits
function.

\begin{example}\label{example:binary-sum-of-digits}
  For $n\ge 0$, let $x(n)=s(n)$ be the binary sum-of-digits function. We clearly
  have
  \begin{equation}\label{eq:recursion-binary-sum-of-digits}
    \begin{aligned}
      x(2n)&=x(n),\\
      x(2n+1)&=x(n)+1
    \end{aligned}
  \end{equation}
  for $n\ge 0$.

  Indeed, we have
  \begin{equation*}
    x(2^j n+ r) = x(n) + x(r)\cdot 1
  \end{equation*}
  for integers $j\ge 0$, $0\le r <2^j$ and $n\ge 0$; i.e., the complex vector space
  generated by the $2$-kernel is generated by $(x(n))_{n\ge 0}$ and the
  constant sequence $(1)_{n\ge 0}$.

  Alternatively, we set $v(n)=(x(n), 1)^\top$ and have
  \begin{align*}
    v(2n)&=
    \begin{pmatrix}
      x(n)\\1
    \end{pmatrix}=
    \begin{pmatrix}
      1&0\\
      0&1
    \end{pmatrix}v(n),\\
    v(2n+1)&=
             \begin{pmatrix}
               x(n)+1\\
               1
             \end{pmatrix}=
    \begin{pmatrix}
      1 & 1\\
      0 & 1
    \end{pmatrix}v(n)
  \end{align*}
  for $n\ge 0$. Thus \eqref{eq:linear-representation} holds with
  \begin{equation*}
    A_0 =
    \begin{pmatrix}
      1&0\\
      0&1
    \end{pmatrix},\qquad
    A_1 =
    \begin{pmatrix}
      1&1\\
      0&1
    \end{pmatrix}.
  \end{equation*}
\end{example}
We defer the discussion of other examples, both generic such as
sequences defined by transducer automata
as well as a specific example involving
the number of odd entries in Pascal's rhombus to Sections~\ref{sec:transducer}
and~\ref{sec:pascal}.

At this point, we note that a linear
representation~\eqref{eq:linear-representation} immediately leads to an
explicit expression for $x(n)$ by induction.

\begin{remark}\label{remark:regular-sequence-as-a-matrix-product}
  Let $r_{\ell-1}\ldots r_0$ be the $q$-ary digit
  expansion\footnote{
    Whenever we write that $r_{\ell-1}\ldots r_0$ is the  $q$-ary digit
    expansion of $n$, we
    mean that $r_j\in\set{0,\ldots, q-1}$ for $0\le j<\ell$, $r_{\ell-1}\neq 0$ and
    $n=\sum_{j=0}^{\ell-1} r_j q^j$. In particular, the $q$-ary expansion of
    zero is the empty word.}
  of $n$. Then
  \begin{equation*}
    x(n) = e_1 A_{r_0}\dotsm A_{r_{\ell-1}}v(0)
  \end{equation*}
  where $e_1=\begin{pmatrix}1& 0& \dotsc& 0\end{pmatrix}$.
\end{remark}

\subsection{Main Result}\label{section:introduction:main-result}

We are interested in the asymptotic behaviour of the summatory function
$X(N)=\sum_{0\le n<N}x(n)$.

At this point, we give a simplified version of our results. We choose any
vector norm $\norm{\,\cdot\,}$ on $\C^d$ and its induced matrix norm. We set $C\coloneqq
\sum_{r=0}^{q-1}A_r$. We choose $R>0$ such that $\norm{A_{r_1}\dotsm
  A_{r_\ell}}=\Oh{R^\ell}$ holds for all $\ell\ge 0$ and $0\le r_1, \dotsc,
r_{\ell}<q$. In other words, $R$ is an upper bound for the joint spectral
radius of $A_1$, \ldots, $A_{q-1}$.
The spectrum of $C$, i.e., the set of eigenvalues of $C$, is denoted by
$\sigma(C)$. For $\lambda\in\sigma(C)$, let $m(\lambda)$ denote the size of the
largest Jordan block of $C$ associated with $\lambda$.
Finally, we consider the Dirichlet series\footnote{
Note that the summatory function $X(N)$ contains the summand $x(0)$ but the Dirichlet series cannot.
This is because the choice of including $x(0)$ into $X(N)$ will lead to more consistent results.}
\begin{equation*}
  \calX(s) = \sum_{n\ge 1} n^{-s}x(n), \qquad
  \calV(s) = \sum_{n\ge 1} n^{-s}v(n).
\end{equation*}
Of course, $\calX(s)$ is the first component of $\calV(s)$.
The principal value of the complex logarithm is denoted by $\log$. The
fractional part of a real number $z$ is denoted by $\fractional{z}\coloneqq z-\floor{z}$.

\begin{theorem}\label{theorem:simple}
  With the notations above, we have
  \begin{multline}\label{eq:formula-X-n}
    X(N) = \sum_{\substack{\lambda\in\sigma(C)\\\abs{\lambda}>R}}N^{\log_q\lambda}
    \sum_{0\le k<m(\lambda)}(\log_q N)^k
    \Phi_{\lambda k}(\fractional{\log_q N}) \\
    + \Oh[\big]{N^{\log_q R}(\log N)^{\max\setm{m(\lambda)}{\abs{\lambda}=R}}}
  \end{multline}
  for suitable $1$-periodic continuous functions $\Phi_{\lambda k}$. If there
  are no eigenvalues $\lambda\in\sigma(C)$ with $\abs{\lambda}\le R$, the
  $O$-term can be omitted.

  For $\abs{\lambda}>R$ and $0\le k<m(\lambda)$, the function $\Phi_{\lambda k}$ is Hölder continuous with any exponent
  smaller than $\log_q(\abs{\lambda}/R)$.

  The Dirichlet series $\calV(s)$ converges absolutely and uniformly on compact
  subsets of the half plane $\Re
  s>\log_q R +1$ and can be continued to a meromorphic function on the half plane $\Re
  s>\log_q R$.
  It satisfies the functional equation
  \begin{equation}\label{eq:functional-equation-V}
    (I-q^{-s}C)\calV(s)= \sum_{n=1}^{q-1}n^{-s}v(n) +
    q^{-s}\sum_{r=0}^{q-1}A_r \sum_{k\ge
      1}\binom{-s}{k}\Bigl(\frac{r}{q}\Bigr)^k \calV(s+k)
  \end{equation}
  for $\Re s>\log_q R$. The right side converges absolutely and uniformly on
  compact subsets of $\Re s>\log_q R$. In particular, $\calV(s)$ can only have
  poles where $q^s\in\sigma(C)$.

  For $\lambda\in\sigma(C)$ with
  $\abs{\lambda}>\max\set{R, 1/q}$, the Fourier series
  \begin{equation*}
    \Phi_{\lambda k}(u) = \sum_{\ell\in \Z}\varphi_{\lambda k\ell}\exp(2\ell\pi i u)
  \end{equation*}
  converges pointwise for $u\in\R$ where
  \begin{equation}\label{eq:Fourier-coefficient:simple}
    \varphi_{\lambda k\ell} = \frac{(\log q)^k}{k!}
    \Res[\bigg]{\frac{\bigl(x(0)+\calX(s)\bigr)
      \bigl(s-\log_q \lambda-\frac{2\ell\pi i}{\log q}\bigr)^k}{s}}%
    {s=\log_q \lambda+\frac{2\ell\pi i}{\log q}}
  \end{equation}
  for $\ell\in\Z$, $0\le k<m(\lambda)$.
\end{theorem}
This theorem is proved in Appendix~\ref{section:proof-theorem-simple}.
Note that we write $\Phi_{\lambda k}(\fractional{\log_q N})$ to optically
emphasise the $1$-periodicity; technically, we have $\Phi_{\lambda
  k}(\fractional{\log_q N})=\Phi_{\lambda k}(\log_q N)$.

We come back to the binary sum of digits.

\begin{example}[Continuation of Example~\ref{example:binary-sum-of-digits}]We
  have $C=A_0+A_1=\bigl(
  \begin{smallmatrix}
    2&1\\0&2
  \end{smallmatrix}\bigr)
$. As $A_0$ is the identity matrix, any product $A_{r_1}\dotsm A_{r_\ell}$ has
the shape $A_1^k=\bigl(
\begin{smallmatrix}
  1&k\\0&1
\end{smallmatrix}\bigr)
$ where $k$ is the number of factors $A_1$ in the product. This implies that
$R$ with $\norm{A_{r_1}\dotsm A_{r_\ell}}=\Oh{R^\ell}$ may be chosen to be any number greater than $1$. As $C$ is a Jordan block
itself, we simply read off that the only eigenvalue of $C$
is $\lambda=2$ with $m(2)=2$.

Thus Theorem~\ref{theorem:simple} yields
\begin{equation*}
  X(N) = N(\log_2 N) \f{\Phi_{21}}{\fractional{\log_2 N}}
  + N \f{\Phi_{20}}{\fractional{\log_2 N}}
\end{equation*}
for suitable $1$-periodic continuous functions $\Phi_{21}$ and $\Phi_{20}$.

In principle, we can now use the functional equation
\eqref{eq:functional-equation-V}. Due to the fact that one component of $v$ is the constant
sequence where everything is known, it is more efficient to use an ad-hoc
calculation for $\calX$ by splitting the sum according to the parity of the index
and using the recurrence relation~\eqref{eq:recursion-binary-sum-of-digits} for $x(n)$. We obtain
\begin{align*}
  \calX(s)&=\sum_{n\ge 1}\frac{x(2n)}{(2n)^s} + \sum_{n\ge
            0}\frac{x(2n+1)}{(2n+1)^s}\\
  &=2^{-s}\sum_{n\ge 1}\frac{x(n)}{n^s} + \sum_{n\ge 0}\frac{x(n)}{(2n+1)^s} +
    \sum_{n\ge 0}\frac{1}{(2n+1)^s}\\
  &=2^{-s}\calX(s) + \frac{x(0)}{1^s} + \sum_{n\ge 1}\frac{x(n)}{(2n)^s} +
  \sum_{n\ge 1} x(n)\Bigl(\frac{1}{(2n+1)^s} - \frac{1}{(2n)^s}\Bigr) \\
  &\hspace{4.985em}+
  2^{-s}\sum_{n\ge 0}\frac1{\bigl(n+\frac12\bigr)^s}\\
  &= 2^{1-s}\calX(s) + 2^{-s}\f[\big]{\zeta}{s, \tfrac12} + \sum_{n\ge 1} x(n)\Bigl(\frac{1}{(2n+1)^s} - \frac{1}{(2n)^s}\Bigr)
\end{align*}
where the Hurwitz zeta function $\f{\zeta}{s, \alpha}\coloneqq\sum_{n+\alpha>0}(n+\alpha)^{-s}$ has been used. We get
\begin{equation}\label{eq:sum-of-digits-functional-equation}
  (1-2^{1-s})\calX(s)=2^{-s} \f[\big]{\zeta}{s, \tfrac12} + \sum_{n\ge 1} x(n)\Bigl(\frac{1}{(2n+1)^s} - \frac{1}{(2n)^s}\Bigr).
\end{equation}
As the sum of digits is bounded by the length of the expansion, we have
$x(n)=\Oh{\log n}$. By combining this estimate with
\begin{equation*}
  (2n+1)^{-s}-(2n)^{-s}
  = \int_{2n}^{2n+1} \Bigl(\frac{\dd}{\dd t}t^{-s}\Bigr)\,\dd t
  = \int_{2n}^{2n+1}(-s)t^{-s-1}\,\dd t
  = \Oh{\abs{s}n^{-\Re s-1}},
\end{equation*}
we see that the sum in \eqref{eq:sum-of-digits-functional-equation}
converges absolutely for $\Re s>0$ and is therefore analytic for $\Re s>0$.

Therefore, the right side of
\eqref{eq:sum-of-digits-functional-equation} is a meromorphic function for $\Re
s>0$ whose only pole is simple and at $s=1$ which originates from
$\f[\big]{\zeta}{s, \tfrac12}$.
Therefore, $\calX(s)$ is a meromorphic function for $\Re s>0$ with a double
pole at $s=1$ and simple poles at $1+\frac{2\ell \pi i}{\log 2}$ for
$\ell\in\Z\setminus\set{0}$.

Thus
\begin{equation}\label{eq:fluctuation-binary-sum-of-digit}
  \begin{aligned}
    \Phi_{21}(u) = \varphi_{210}
    &= (\log 2)\Res[\Big]{\frac{\calX(s)(s-1)}{s}}{s=1} \\
    &= (\log 2)\Res[\Big]{\frac{2^{-s}(s-1)}{1-2^{1-s}}
         \f[\big]{\zeta}{s, \tfrac12}}{s=1}
    = \frac12
  \end{aligned}
\end{equation}
by \eqref{eq:Fourier-coefficient:simple} and
\eqref{eq:sum-of-digits-functional-equation}.

We conclude that
\begin{equation*}
  X(N)=\frac12 N \log_2 N + N \f{\Phi_{20}}{\fractional{\log_2 N}}.
\end{equation*}
We refrain from computing the Fourier coefficients of $\Phi_{20}(u)$ explicitly
at this point: Numerically, they could be computed from
\eqref{eq:sum-of-digits-functional-equation}. However, an explicit expression
can be obtained by rewriting the residues of $\calX(s)$ in terms of shifted
residues of $\sum_{n\ge 1}\bigl(x(n)-x(n-1)\bigr)n^{-s}$ and computing the latter
explicitly; see \cite[Proof of
Corollary~2.5]{Heuberger-Kropf-Prodinger:2015:output}. This yields the
well-known result by Delange~\cite{Delange:1975:chiffres}.

It will also turn out that \eqref{eq:fluctuation-binary-sum-of-digit} being a constant function is an
immediate consequence of the fact that $
\begin{pmatrix}
  0& 1
\end{pmatrix}
$ is a left eigenvector of both $A_0$ and $A_1$ associated with the eigenvalue
$1$.
\end{example}

\subsection{Heuristic Approach: Mellin--Perron Summation}\label{sec:heuristic}
The purpose of this section is to explain why the formula
\eqref{eq:Fourier-coefficient:simple} for the Fourier coefficients is
expected. The approach here is heuristic and non-rigorous because we do not
have the required growth estimates.

By the Mellin--Perron summation formula of order $0$ (see, for example,
\cite[Theorem~2.1]{Flajolet-Grabner-Kirschenhofer-Prodinger:1994:mellin}),
we have
\begin{equation*}
  \sum_{1\le n<N}x(n) + \frac{x(N)}{2} = \frac1{2\pi i}\int_{\max\set{\log_q R + 2,1}
    -i\infty}^{\max\set{\log_q R + 2,1} +i\infty} \calX(s)\frac{N^s\,\dd s}{s}.
\end{equation*}
By Remark~\ref{remark:regular-sequence-as-a-matrix-product} and the definition
of $R$, we have
$x(N)=\Oh{R^{\log_q N}}=\Oh{N^{\log_q R}}$. Adding the summand $x(0)$ to match our definition of
$X(N)$ amounts to adding $\Oh{1}$.
Shifting the line of integration to the left---we have \emph{no analytic justification}
that this is allowed---and using the location of the poles of $\calX(s)$ claimed in
Theorem~\ref{theorem:simple} yield
\begin{multline*}
  X(N) = \sum_{\substack{\lambda\in\sigma(C)\\\abs{\lambda}>R}}\sum_{\ell\in\Z}
  \Res[\Big]{\frac{\calX(s)N^s}{s}}%
  {s=\log_q \lambda + \frac{2\ell\pi i}{\log q}} \\
  + \frac1{2\pi i}\int_{\log_q R+\varepsilon
    -i\infty}^{\log_q R+\varepsilon +i\infty} \calX(s)\frac{N^s\,\dd s}{s} + \Oh{N^{\log_q R} + 1}
\end{multline*}
for some $\varepsilon>0$.
Expanding $N^s$ as
\begin{equation*}
  N^s = \sum_{k\ge 0} \frac{(\log N)^k}{k!} N^{\log_q \lambda + \frac{2\ell\pi
      i}{\log q}} \Bigl(s-\log_q \lambda
  -\frac{2\ell\pi i}{\log q}\Bigr)^k
\end{equation*}
and assuming that the remainder integral converges absolutely yields
\begin{multline*}
  X(N) = \sum_{\substack{\lambda\in\sigma(C)\\\abs{\lambda}>R}} N^{\log_q
    \lambda}\sum_{0\le k<m_{\lambda\ell}}
  (\log_q N)^k \sum_{\ell\in\Z}\varphi_{\lambda k\ell}\exp\bigl(2\ell\pi i \log_q
  N\bigr)\\
  + \Oh{N^{\log_q R+\varepsilon}+1}
\end{multline*}
where $m_{\lambda \ell}$ denotes the order of the pole of $\calX(s)/s$ at
$\log_q\lambda + \frac{2\ell\pi i}{\log q}$ and $\varphi_{\lambda k \ell}$ is as
in \eqref{eq:Fourier-coefficient:simple}.

Summarising, this heuristic approach explains most of the formul\ae{} in
Theorem~\ref{theorem:simple}. Some details (exact error term and order of the
poles) are not explained by this approach.
A result ``repairing'' the zeroth order Mellin--Perron formula is known as
Landau's theorem, see \cite[\S~9]{Berthe-Lhote-Vallee:2016:probab}. It is not
applicable to our situation due to multiple poles along vertical lines which
then yield the periodic fluctuations. Instead, we prove
a theorem which provides the required
justification (not by estimating the relevant quantities, but by reducing the
problem to higher order Mellin--Perron summation). The essential assumption is
that the summatory function can be decomposed into fluctuations multiplied by
some growth factors such as in \eqref{eq:formula-X-n}.

\subsection{Relation to Previous Work}
\label{introduction:relation-to-previous-work}

Sequences defined as the output sum of transducer automata in the sense of
\cite{Heuberger-Kropf-Prodinger:2015:output} are a special case of regular
sequences; these are a generalisation of many previously studied concepts.
In that case, much more is known (variance, limiting distribution, higher
dimensional input). See \cite{Heuberger-Kropf-Prodinger:2015:output} for
references and results. A more detailed comparison can be found in Section~\ref{sec:transducer}.
Divide and Conquer recurrences (see \cite{Hwang-Janson-Tsai:2017:divide-conquer-half} and
\cite{Drmota-Szpankowski:2013:divide-and-conquer}) can also be seen as special cases of regular sequences.

The asymptotics of the summatory function of specific examples of regular
sequences has been studied in \cite{Grabner-Heuberger:2006:Number-Optimal},
\cite{Grabner-Heuberger-Prodinger:2005:counting-optimal-joint},
\cite{Dumas-Lipmaa-Wallen:2007:asymp}.

Dumas~\cite{Dumas:2013:joint, Dumas:2014:asymp} finally proved the first part
of Theorem~\ref{theorem:simple}. We re-prove it here in a self-contained way
because we need more explicit results than obtained by Dumas
(e.g., we need explicit expressions for the
fluctuations) for proving Hölder continuity and to explicitly get the precise
structure depending on the eigenspaces. Unfortunately,
Dumas' paper introduces linear representations as we do in
\eqref{eq:linear-representation}, but then the order of factors is reversed in his
equivalent of Remark~\ref{remark:regular-sequence-as-a-matrix-product}, which
means that some transpositions have to be silently introduced.

The first version of our pseudo-Tauberian argument
was provided in \cite{Flajolet-Grabner-Kirschenhofer-Prodinger:1994:mellin}:
there, no logarithmic factors were allowed and the growth conditions on the
Dirichlet series were stronger.

\def\moveplan{
We formulate the full version of our results here in
Appendix~\ref{section:results}. Formulating them will need quite a number of
definitions provided in Appendix~\ref{sec:definitions-notations}. In order to
cut straight to the results themselves, we will refrain from motivations and
comments on the definitions in Appendix~\ref{sec:definitions-notations} and
postpone those to Appendix~\ref{sec:motivation-definitions}.
Finally, the
proofs of our results will be provided in
Appendices~\ref{section:proof-contribution-of-eigenspace} to
\ref{section:proof-theorem-simple} after the very short
Appendix~\ref{additional-notation} where a few notations used throughout the proofs
are fixed.
}

\subsection*{Acknowledgement}
We thank Sara Kropf for her comments on an early version of this paper.

\def\movesectionresults{
\section{Results}\label{section:results}
\moveplan

As announced in the
introduction, we study matrix products instead of regular sequences. We will
come back to regular sequences in Appendix~\ref{section:proof-theorem-simple}.

\subsection{Problem Statement}
Let $q\ge 2$, $d\ge 1$ be fixed integers and $A_0$, \ldots,
$A_{q-1}\in\C^{d\times d}$.
We investigate the sequence $(f(n))_{n\ge 0}$ of $d\times d$ matrices such that
\begin{equation}\label{eq:regular-matrix-sequence}
  f(qn+r)=A_r f(n) \quad\text{ for $0\le r<q$, $0\le n$ with  $qn+r\neq 0$}
\end{equation}
and $f(0)=I$.

Let $n$ be an integer with $q$-ary expansion
$r_{\ell-1}\ldots r_0$. Then it is easily seen that \eqref{eq:regular-matrix-sequence} implies that
\begin{equation}\label{eq:f-as-product}
  f(n)=A_{r_0}\ldots A_{r_{\ell-1}}.
\end{equation}

We are interested in the asymptotic behaviour of $F(N)\coloneqq\sum_{0\le n<N} f(n)$.

\subsection{Definitions and Notations}\label{sec:definitions-notations}
In this section, we give all definitions and notations which are required in
order to state the results. For the sake of conciseness, we do not give any
motivations for our definitions here; those are deferred to Appendix~\ref{sec:motivation-definitions}.

The following notations are essential:
\begin{itemize}
\item
Let $\norm{\,\cdot\,}$ denote a fixed norm on $\C^d$ and its induced matrix
norm on $\C^{d\times d}$.

\item We set $B_r \coloneqq \sum_{0\le r'<r} A_{r'}$ for $0\le
r<q$ and $C\coloneqq\sum_{0\le r<q} A_r$.

\item
The joint spectral radius of $A_0$, \ldots, $A_{q-1}$ is denoted by
\begin{equation*}
  \rho\coloneqq\inf_{\ell}\sup
  \setm[\big]{ \norm{A_{r_1}\ldots A_{r_\ell}}^{1/\ell}}{r_1, \ldots, r_\ell\in\set{0, \ldots, q-1}}.
\end{equation*}
If the set of matrices $A_0$, \dots, $A_{q-1}$ has the finiteness property,
i.e., there is an $\ell>0$ such that
\begin{equation*}
  \rho = \sup
  \setm[\big]{\norm{A_{r_1}\ldots A_{r_\ell}}^{1/\ell}}{r_1, \ldots, r_\ell\in\set{0, \ldots, q-1}},
\end{equation*}
then we set $R=\rho$. Otherwise, we choose $R>\rho$ in such a way that there is
no eigenvalue $\lambda$ of $C$ with $\rho<\abs{\lambda}\le R$.

\item
The spectrum of $C$, i.e., the set of eigenvalues of $C$, is denoted by
$\sigma(C)$.

\item For a positive integer $n_0$, set
  \begin{equation*}
    \calF_{n_0}(s) \coloneqq \sum_{n\ge n_0} n^{-s}f(n)
  \end{equation*}
  for a complex variable $s$.

\item Set $\chi_k\coloneqq \frac{2\pi i k}{\log q}$ for $k\in\Z$.
\end{itemize}

In the formulation of Theorem~\ref{theorem:contribution-of-eigenspace} and
Corollary~\ref{corollary:main}, the following constants are needed
additionally:

\begin{itemize}
\item
Choose a regular matrix $T$ such that $T C T^{-1}\eqqcolon J$ is in Jordan form.

\item  Let $D$ be
the diagonal matrix whose $j$th diagonal element is $1$ if the $j$th diagonal
element of $J$ is not equal to $1$; otherwise the $j$th diagonal element of $D$
is $0$.

\item
Set $C'\coloneqq T^{-1}DJT$.

\item
Set $K\coloneqq T^{-1}DT(I-C')^{-1}(I-A_0)$.

\item
For a $\lambda\in\C$, let $m(\lambda)$ be the size of the largest
Jordan block associated with $\lambda$. In particular, $m(\lambda)=0$ if $\lambda\not\in\sigma(C)$.

\item For $m\ge 0$, set
\begin{equation*}
  \vartheta_m \coloneqq \frac1{m!}T^{-1}(I-D)T(C-I)^{m-1}(I-A_0);
\end{equation*}
here, $\vartheta_0$ remains undefined if $1\in\sigma(C)$.\footnote{
If $1\in\sigma(C)$, then the matrix $C-I$ is singular. In that case, $\vartheta_0$ will never be used.}

\item Define $\vartheta \coloneqq \vartheta_{m(1)}$.
\end{itemize}

All implicit $O$-constants depend on $q$, $d$, the matrices $A_0$, \ldots, $A_{q-1}$ (and therefore on $\rho$)
as well as on $R$.

\subsection{Decomposition into Periodic Fluctuations}
Instead of considering $F(N)$, it is certainly enough to consider $wF(N)$ for
all generalised left eigenvectors $w$ of $C$, e.g., the rows of $T$. The
result for $F(N)$ then follows by taking appropriate linear combinations.

\begin{theorem}\label{theorem:contribution-of-eigenspace}
  Let $w$ be a generalised left eigenvector of rank $m$ of $C$ corresponding to the eigenvalue $\lambda$.
  \begin{enumerate}
  \item\label{item:small-eigenvalue} If $\abs{\lambda}<R$, then
    \begin{equation*}
      wF(N)=wK + (\log_q N)^m w\vartheta_m  + \Oh{N^{\log_q R}}.
    \end{equation*}
  \item\label{item:R-eigenvalue} If $\abs{\lambda}=R$, then
    \begin{equation*}
      wF(N)=wK + (\log_q N)^m w\vartheta_m + \Oh{N^{\log_q R} (\log N)^{m}}.
    \end{equation*}
  \item\label{item:large-eigenvalue} If $\abs{\lambda}>R$, then there are $1$-periodic continuous functions
    $\Phi_k\colon \R\to\C^d$, $0\le k<m$, such that
    \begin{equation*}
      wF(N)=wK + (\log_q N)^mw\vartheta_m + N^{\log_q\lambda} \sum_{0\le k<m}(\log_q N)^k\Phi_k(\fractional{\log_q N})
    \end{equation*}
    for $N\ge q^{m-1}$. The function $\Phi_k$ is Hölder-continuous with any
    exponent smaller than  $\log_q\abs{\lambda}/R$.

    If, additionally, the left eigenvector $w(C-\lambda I)^{m-1}$ of $C$ happens to be a left eigenvector to each matrix
    $A_0$, \ldots, $A_{q-1}$ associated with the eigenvalue~$1$, then
    \begin{equation*}
      \Phi_{m-1}(u)=\frac1{q^{m-1}(m-1)!}w(C-q I)^{m-1}
    \end{equation*}
    is constant.
  \end{enumerate}
  Here, $wK=0$  for $\lambda=1$ and $w\vartheta_m=0$ for $\lambda\neq 1$.
\end{theorem}
This theorem is proved in Appendix~\ref{section:proof-contribution-of-eigenspace}.
Note that in general, the three summands in the theorem have different growths:
a constant, a logarithmic term and a term whose growth depends essentially
on the joint spectral radius and the eigenvalues larger than the
joint spectral radius, respectively. The vector $w$ is not directly visible in front of
the third summand; instead, the vectors of its Jordan chain are part of the function~$\Phi_k$.

Expressing the identity matrix as linear combinations of generalised left
eigenvalues and summing up the contributions of
Theorem~\ref{theorem:contribution-of-eigenspace} essentially yields the following corollary.

\begin{corollary}\label{corollary:main}
  With the notations above, we have
  \begin{multline*}
    F(N) = \sum_{\substack{\lambda\in\sigma(C)\\\abs{\lambda}>\rho}} N^{\log_q
      \lambda}\sum_{0\le k<m(\lambda)}(\log_q N)^k\Phi_{\lambda
      k}(\fractional{\log_q N})  + (\log_q N)^{m(1)} \vartheta + K\\
    + \Oh[\big]{N^{\log_q R}(\log N)^{\max\setm{m(\lambda)}{\abs{\lambda}=R}}}
  \end{multline*}
  for suitable $1$-periodic continuous functions $\Phi_{\lambda k}$.
  If $1$ is not an eigenvalue of $C$, then $\vartheta=0$.  If
  there are no eigenvalues $\lambda\in\sigma(C)$ with $\abs{\lambda}\le \rho$,
  then the $O$-term can be omitted.

  For $\abs{\lambda}>R$, the function $\Phi_{\lambda k}$ is Hölder continuous with any exponent
  smaller than $\log_q(\abs{\lambda}/R)$.
\end{corollary}
This corollary is proved in Appendix~\ref{section:proof:corollary-main}.

\subsection{Dirichlet Series}
This section gives the required result on the Dirichlet series~$\calF_{n_0}$. For
theoretical purposes, it is enough to study $\calF\coloneqq\calF_1$; for numerical purposes,
however, convergence improves for larger values of $n_0$.

\begin{theorem}\label{theorem:Dirichlet-series}Let $n_0$ be a positive
  integer. Then the Dirichlet series $\calF_{n_0}(s)$
  converges absolutely and uniformly on compact subsets of the half plane $\Re s > \log_q \rho + 1$, thus is analytic there.

  We have
  \begin{equation}\label{eq:analytic-continuation}
    (I-q^{-s}C)\calF_{n_0}(s) = \calG_{n_0}(s)
  \end{equation}
  for $\Re s>\log_q \rho +1$ with
  \begin{equation}\label{eq:Dirichlet-recursion}
    \calG_{n_0}(s) = \sum_{n=n_0}^{qn_0-1}n^{-s}f(n) + q^{-s}\sum_{r=0}^{q-1} A_r \sum_{k\ge
      1} \binom{-s}{k}\Bigl(\frac{r}{q}\Bigr)^k \calF_{n_0}(s+k).
  \end{equation}
  The series in \eqref{eq:Dirichlet-recursion} converge
  absolutely and uniformly on compact sets for $\Re s>\log_q \rho$. Thus \eqref{eq:analytic-continuation} gives a meromorphic
  continuation of $\calF_{n_0}$ to the half plane $\Re s>\log_q \rho$ with
  possible poles at $s=\log_q \lambda + \chi_\ell$ for each
  $\lambda\in \sigma(C)$ with $\abs{\lambda}>\rho$ and $\ell\in\Z$
  whose pole order is at most $m(\lambda)$.

  Let $\delta>0$. For real $z$, we set
  \begin{equation*}
    \mu_\delta(z)= \max\set{ 1 - (z-\log_q \rho -\delta), 0},
  \end{equation*}
  i.e., the linear function on the interval
  $[\log_q\rho+\delta, \log_q\rho+\delta+1]$
  with~$\mu_\delta(\log_q\rho+\delta)=1$ and~$\mu_\delta(\log_q\rho+\delta+1)=0$.
  Then
  \begin{equation}\label{eq:order-F}
    \calF_{n_0}(s) = \Oh[\big]{\abs{\Im s}^{\mu_\delta(\Re s)}}
  \end{equation}
  holds uniformly for $\log_q \rho+\delta\le \Re s$ and $\abs{q^s-\lambda} \ge \delta$
  for all eigenvalues $\lambda\in\sigma(C)$. Here, the implicit $O$-constant
  also depends on $\delta$.
\end{theorem}

\begin{remark}\label{remark:Dirichlet-series:bound}
  By the identity theorem for analytic functions, the meromorphic
  continuation of $\calF_{n_0}$ is unique on the domain given in the
  theorem. Therefore, the bound~\eqref{eq:order-F} does not depend on
  the particular expression for the meromorphic continuation given
  in~\eqref{eq:analytic-continuation}
  and~\eqref{eq:Dirichlet-recursion}.
\end{remark}

Theorem~\ref{theorem:Dirichlet-series} is proved in
Appendix~\ref{section:proof:Dirichlet-series}. In the proof we need
Dirichlet series like $\sum_{n\ge n_0} d(n)\, (n+\beta)^{-s}$ and their
differences to standard Dirichlet series $\sum_{n\ge n_0} d(n)\, n^{-s}$.
The following lemma provides some insights. It will turn out
to be useful to have it as result listed in this section and not
buried in the proofs sections.

\begin{lemma}\label{lemma:shifted-Dirichlet}
  Let $\calD(s) = \sum_{n \ge n_0} d(n)/n^s$ be a Dirichlet series with
  coefficients $d(n)=\Oh{n^{-\log_q R'}}$ for all $R'>\rho$.
  Let $\beta\in\C$ with $\abs{\beta}<n_0$ and $\delta>0$. Set
  \begin{equation*}
    \f{\Sigma}{s, \beta, \calD} \coloneqq
    \biggl(\, \sum_{n\ge n_0} \frac{d(n)}{(n+\beta)^s} \biggr) - \calD(s).
  \end{equation*}
  Then
  \begin{equation*}
    \f{\Sigma}{s, \beta, \calD} = \sum_{k\ge 1}
    \binom{-s}{k} \beta^k \calD(s+k),
  \end{equation*}
  where the series converges
  absolutely and uniformly on compact sets for $\Re s>\log_q \rho$,
  thus is analytic there.
  Moreover, with $\mu_\delta$ as in Theorem~\ref{theorem:Dirichlet-series},
  \begin{equation*}
    \f{\Sigma}{s, \beta, \calD}=\Oh[\big]{\abs{\Im s}^{\mu_\delta(\Re s)}}
  \end{equation*}
  as $\abs{\Im s}\to\infty$
  holds uniformly for $\log_q \rho + \delta\le \Re s\le \log_q \rho +\delta+1$.
\end{lemma}

\subsection{Fourier Coefficients}
As discussed in Section~\ref{sec:heuristic}, we would like to apply the zeroth
order Mellin--Perron summation formula but need analytic justification. In the
following theorem we prove that whenever it is known that the result is a
periodic fluctuation, the use of zeroth order Mellin--Perron summation can be
justified. In contrast to the remaining paper, this theorem does \emph{not}
assume that $f(n)$ is a matrix product.

\begin{theorem}\label{theorem:use-Mellin--Perron}
  Let $f(n)$ be a sequence, let $\kappa_0\in\R\setminus\set{0}$ and $\kappa\in\C$ with $\Re \kappa>
  \kappa_0 > -1$, $\delta>0$,
  $q>1$ be real numbers with $\delta \le \pi/(\log q)$
  and $\delta < \Re \kappa-\kappa_0$,
  and let $m$ be a positive integer. Moreover, let $\Phi_k$ be
  Hölder-continuous (with exponent $\alpha$ with
  $\Re\kappa-\kappa_0<\alpha\le 1$) $1$-periodic functions for $0\le k<m$ such that
  \begin{equation}\label{eq:F-N-periodic}
    F(N)\coloneqq \sum_{1\le n< N} f(n) = \sum_{0\le k<m}N^\kappa (\log_q N)^k
    \Phi_{k}(\fractional{\log_q N}) + \Oh{N^{\kappa_0}}
  \end{equation}
  for integers $N\to\infty$.

  For the Dirichlet series $\calF(s)\coloneqq \sum_{n\ge 1}n^{-s}f(n)$
  assume that
  \begin{itemize}
  \item there is some real number $\sigma_a\ge \Re \kappa$ such that $\calF(s)$ converges absolutely for $\Re s>\sigma_a$;
  \item  the Dirichlet series $\calF(s)$
    can be continued to a meromorphic function for $\Re s > \kappa_0-\delta$
    such that poles can only occur at $\kappa+\chi_\ell$ for $\ell\in\Z$ and such that these poles have order at
    most $m$;
  \item there is some real number~$\eta>0$ such that
    for $\kappa_0 \le \Re s  \le \sigma_a$ and
    $\abs{s-\kappa-\chi_\ell}\ge \delta$ for all $\ell\in\Z$, we have
    \begin{equation}\label{eq:Dirichlet-order}
      \calF(s) = \Oh[\big]{\abs{\Im s}^{\eta}}
    \end{equation}
    for $\abs{\Im s}\to\infty$.
  \end{itemize}
  All implicit $O$-constants may depend on $f$, $q$, $m$, $\kappa$, $\kappa_0$,
  $\alpha$, $\delta$, $\sigma_a$ and $\eta$.

  Then
  \begin{equation*}
    \Phi_k(u) = \sum_{\ell\in \Z}\varphi_{k\ell}\exp(2\ell\pi i u)
  \end{equation*}
  for $u\in\R$ where
  \begin{equation}\label{eq:Fourier-coefficient}
    \varphi_{k\ell} = \frac{(\log q)^k}{k!}
    \Res[\Big]{\frac{\calF(s)(s-\kappa-\chi_\ell)^k}{s}}%
    {s=\kappa+\chi_\ell}
  \end{equation}
  for $\ell\in\Z$ and $0\le k<m$.

  If $-1<\kappa_0<0$ and $\kappa\notin \frac{2\pi i}{\log q}\Z$, then $\calF(0)=0$.
\end{theorem}
This theorem is proved in Appendix~\ref{section:proof:use-Mellin--Perron}.
The theorem is more general than necessary for $q$-regular sequences because
Theorem~\ref{theorem:Dirichlet-series} shows that we could use some $0<\eta<1$.
However, it might be applicable in other cases, so we prefer to state it in
this more general form.

\section{Remarks on the Definitions}\label{sec:motivation-definitions}
In this section, we give some motivation for and comments on the definitions listed in
Appendix~\ref{sec:definitions-notations}.

\subsection{\texorpdfstring{$q$}{q}-Regular Sequences vs.\ Matrix Products}\label{section:q-regular-matrix-product}
We note one significant difference between the study of $q$-regular sequences
as in \eqref{eq:linear-representation} and the study of matrix
products~\eqref{eq:f-as-product}.
The recurrence \eqref{eq:linear-representation} is supposed to hold for
$qn+r=0$, too; i.e. $v(0)=A_0v(0)$. This implies that $v(0)$ is either the zero
vector (which is not interesting at all) or that $v(0)$ is a right eigenvector of
$A_0$ associated with the eigenvalue~$1$.

We do not want to impose this condition in the study of the matrix
product~\eqref{eq:f-as-product}. Therefore, we exclude the case $qn+r=0$ in
\eqref{eq:regular-matrix-sequence}.
This comes at the price of the terms $K$,
$\vartheta_m$, $\vartheta$ in Theorem~\ref{theorem:contribution-of-eigenspace} which
vanish if multiplied by a right eigenvector to the eigenvalue $1$ of $A_0$ from the
right. This is the reason why Theorem~\ref{theorem:simple} has simpler
expressions than those encountered in Theorem~\ref{theorem:contribution-of-eigenspace}.

\subsection{Joint Spectral Radius}
Let
\begin{equation*}
\rho_\ell\coloneqq \sup
\setm[\big]{\norm{A_{r_1}\ldots A_{r_\ell}}^{1/\ell}}{r_1, \ldots, r_\ell\in\set{0, \ldots, q-1}}.
\end{equation*}
Then the
submultiplicativity of the norm and Fekete's subadditivity lemma~\cite{Fekete:1923:ueber-verteil} imply that
$\lim_{\ell\to\infty}\rho_\ell=\inf_{\ell>0}\rho_{\ell}=\rho$,
cf.~\cite{Rota-Strang:1960}. In view of equivalence of norms, this shows that
the joint spectral radius does not depend on the chosen norm. For our purposes,
the important point is that the choice of $R$ ensures that there is an
$\ell_0>0$ such that $\rho_{\ell_0}\le R$, i.e., $\norm{A_{r_1}\ldots
A_{r_{\ell_0}}}\le R^{\ell_0}$ for all $r_j\in\set{0,\ldots, q-1}$. For any $\ell>0$, we use long division to write
$\ell=s\ell_0+r$ and by submultiplicativity of the norm, we get $\norm{A_{r_1}\ldots
    A_{r_\ell}}\le R^{s\ell_0} \rho_{r}^r$ and thus
\begin{equation}\label{eq:bound-prod}
  \norm{A_{r_1}\ldots
    A_{r_\ell}}=\Oh{R^{\ell}}
\end{equation}
for all $r_j\in\set{0,\ldots,q-1}$ and $\ell\to\infty$. We will only use
\eqref{eq:bound-prod} and no further properties of the joint spectral radius.
Note that~\eqref{eq:f-as-product} and \eqref{eq:bound-prod} imply that
\[f(n)=\Oh{R^{\log_q n}}=\Oh{n^{\log_q R}}\]
for $n\to\infty$.

As mentioned, we say that the set of matrices $A_0$, \dots, $A_{q-1}$,
has the \emph{finiteness property} if there is an $\ell>0$ with
$\rho_\ell=\rho$; see~\cite{Jungers:2009:joint-spectral-radius,
  Lagarias-Wang:1995:finiteness-conjecture-jsr}.

\subsection{Constants for Theorem~\ref{theorem:contribution-of-eigenspace}}\label{section:constants-for-theorem}
In contrast to usual conventions, we write matrix representations of
endomorphisms as multiplications $x\mapsto xM$ where $x$ is a (row) vector in
$\C^d$ and $M$ is a matrix. Note that we usually denote this endomorphism
by the corresponding calligraphic letter, for example, the endomorphism represented
by the matrix~$M$ is denoted by $\calM$.

Consider the endomorphism $\calC$ which maps a row vector $x\in\C^d$ to $xC$
and its generalised eigenspaces $W_\lambda$ for $\lambda\in\C$.  (These are the
generalised left eigenspaces of $C$. If $\lambda\notin\sigma(C)$, then
$W_\lambda=\set{0}$.) Then it is well-known that $\calC\rvert_{W_\lambda}$ is an
endomorphism of $W_\lambda$ and that
$\C^d=\bigoplus_{\lambda\in\sigma(C)}W_\lambda$. Let $\calT$ be the basis formed by
the rows of $T$. Then the matrix representation of $\calC$
with respect to~$\calT$ is $J$.

Let now $\calD$ be the endomorphism of $\C^d$ which acts as identity on
$W_\lambda$ for $\lambda\neq 1$ and as zero on $W_1$. Its matrix representation
with respect to the basis $\calT$ is $D$; its matrix representation with
respect to the standard basis is $T^{-1}DT$.

Finally, let $\calC'$ be the endomorphism $\calC'=\calC \circ \calD$. As
$\calC$ and $\calD$ decompose along
$\C^d=\bigoplus_{\lambda\in\sigma(C)}W_\lambda$ and $\calD$ commutes with every
other endomorphism on $W_\lambda$ for all $\lambda$, we clearly also have
$\calC'=\calD\circ\calC$. Thus the matrix representation of $\calC'$ with
respect to $\calT$ is $DJ=JD$; its matrix representation with respect to the
standard basis is $T^{-1}DJT=C'$.

Now consider a generalised left eigenvector $w$ of $C$.  If it is associated to the eigenvalue $1$, then $w
T^{-1}DT=\calD(w)=0$, $wK=0$ and $wC'=\calC'(w)=0$. Otherwise, that is, if $w$ is associated to an eigenvalue not equal to $1$, we have $wT^{-1}DT=\calD(w)=w$,
$wC'=\calC'(w)=\calC(w)=wC$,
$w{C'}^j={\calC'}^j(w)=\calC^j(w)=wC^j$ for $j\ge 0$ and $w\vartheta_m=0$. Also note that
$1$ is not an eigenvalue of $C'$, thus $I-C'$ is indeed regular.
If $1$ is not an eigenvalue of $C$, then everything is simpler:
$D$ is the identity matrix, $C'=C$, $K=(I-C)^{-1}(I-A_0)$ and $\vartheta=0$.

} 

\section{Sequences Defined by Transducer Automata}
\label{sec:transducer}
Let $q\ge 2$ be a positive integer. We consider a complete deterministic
subsequential transducer  $\calT$ with input alphabet $\set{0, \ldots, q-1}$
and output alphabet $\C$, see \cite[Chapter~1]{Berthe-Rigo:2010:combin}.
Recall that a transducer is said to be \emph{deterministic} and \emph{complete}
if for every state and every digit of the input alphabet, there is exactly one
transition starting in this state with this input label. A \emph{subsequential}
transducer has a final output label for every state.

For a non-negative integer $n$, let $\calT(n)$ be the sum of the output labels
(including the final output label) encountered when the transducer reads the
$q$ary expansion of $n$. This concept has been thoroughly studied in
\cite{Heuberger-Kropf-Prodinger:2015:output}: there, $\calT(n)$ is considered
as a random variable defined on the probability space $\set{0, \ldots, N-1}$
equipped with uniform distribution. The expectation in this model corresponds
(up to a factor of~$N$) to our summatory function $\sum_{0\le n<N}\calT(n)$.
We remark that in \cite{Heuberger-Kropf-Prodinger:2015:output}, the variance
and limiting distribution of the random variable $\calT(n)$ have also been
investigated. Most of the results there are also valid for higher dimensional input.

The purpose of this section is to show that $\calT(n)$ is a $q$-regular
sequence and to see that our results here coincide with the corresponding
results in \cite{Heuberger-Kropf-Prodinger:2015:output}. We note that the
binary sum of digits considered in Example~\ref{example:binary-sum-of-digits}
is the special case of $q=2$ and the transducer consisting of a single state
which implements the identity map. For additional special cases of this
concept, see \cite{Heuberger-Kropf-Prodinger:2015:output}. Note that our result
here for the summatory function contains (fluctuating) terms for all
eigenvalues $\lambda$ of the adjacency matrix of the underlying digraph with
$1<\abs{\lambda}$ whereas in  \cite{Heuberger-Kropf-Prodinger:2015:output} only
contributions of those eigenvalues $\lambda$ with $\abs{\lambda}=q$ are
available, all other contributions are absorbed by the error term there.

\newcommand{\movepfa}{We will need the following consequence of Perron--Frobenius theory.}
By a \emph{component} of a digraph we always mean a strongly
connected component.  We call a component  \emph{final} if there
are no arcs leaving the component. The \emph{period} of a component
is the greatest common divisor of its cycle lengths. The \emph{final period} of
a digraph is the least common multiple of the periods of its final components.

\newcommand{\movepfb}{
\begin{lemma}\label{lemma:Perron--Frobenius-again}
  Let $D$ be a directed graph  where each
  vertex has outdegree $q$. Let $M$ be its adjacency matrix and $p$ be its
  final period.
  Then $M$ has spectral radius $q$, $q$ is an
  eigenvalue of $M$ and for all eigenvalues $\lambda$ of $M$ of modulus $q$, the
  algebraic and geometric multiplicities coincide and $\lambda = q\zeta$ for
  some $p$th root of unity $\zeta$.
\end{lemma}
This lemma follows from setting $t=0$ in
\cite[Lemma~2.3]{Heuberger-Kropf-Prodinger:2015:output}. As
\cite[Lemma~2.3]{Heuberger-Kropf-Prodinger:2015:output} proves more than we
need here and depends on the notions of that article, we extract the relevant
parts of \cite{Heuberger-Kropf-Prodinger:2015:output} to provide a
self-contained (apart from Perron--Frobenius theorem) proof of
Lemma~\ref{lemma:Perron--Frobenius-again}.
}

  We consider the states of $\calT$ to be numbered by $\set{1, \ldots, d}$ for some
  positive integer $d\ge 1$ such that the initial state is state~$1$. We set
  $\calT_j(n)$ to be the sum of the output labels (including the final output
  label) encountered when the transducer reads the $q$ary expansion of $n$ when
  starting in state~$j$. By construction, we have $\calT(n)=\calT_1(n)$ and
  $\calT_j(0)$ is the final output label of state~$j$. We set
  $y(n)=(\calT_1(n), \ldots, \calT_d(n))$.
  For $0\le r<q$, we define the $(d\times d)$-$\set{0, 1}$-matrix $P_r$ in such a
  way that there is a one in row~$j$, column~$k$ if and only if there is a
  transition from state~$j$ to state~$k$ with input label $r$. The vector $o_r$
  is defined by setting its $j$th coordinate to be the output label of the transition
  from state~$j$ with input label $r$.

For $n_0\ge 1$, we set
\begin{equation*}
\calX(s)=\sum_{n\ge 1}n^{-s}\calT(n),\qquad
\calY_{n_0}(s)=\sum_{n\ge n_0}n^{-s}y(n),\qquad
\zeta_{n_0}(s, \alpha)=\sum_{n\ge n_0}(n+\alpha)^{-s}.
\end{equation*}
The last Dirichlet series is a truncated version of the Hurwitz zeta function.

\begin{corollary}\label{corollary:transducer}
  Let $\calT$ be a transducer as described at the beginning of
  this section. Let $M$ and $p$ be the adjacency matrix and the final period of
  the underlying digraph, respectively. For $\lambda\in\C$ let $m(\lambda)$
  be the size of the largest Jordan block associated with the eigenvalue
  $\lambda$ of $M$.

  Then
  $(\calT(n))_{n\ge 0}$ is a $q$-regular sequence and
  \begin{equation}\label{eq:transducer:summatory-as-fluctuation}
  \begin{aligned}
    \sum_{0\le n<N}\calT(n) = e_\calT N\log_q N  &+ N\Phi(\log_q N)\\
    &+ \sum_{\substack{\lambda\in\sigma(M)\\
        1<\abs{\lambda}<q
      }} N^{\log_q \lambda} \sum_{0\le k<m(\lambda)}(\log_q N)^k\Phi_{\lambda
      k}(\log_q N)\\
    &+ \Oh[\big]{(\log N)^{\max\setm{m(\lambda)}{\abs{\lambda}=1}}}
  \end{aligned}
  \end{equation}
  for some continuous $p$-periodic function $\Phi$, some continuous
  $1$-periodic functions $\Phi_{\lambda k}$ for $\lambda\in\sigma(M)$ with $1<\abs{\lambda}<q$ and $0\le
  k<m(\lambda)$ and some constant
  $e_\calT$.

  Furthermore,
  \begin{equation*}
    \Phi(u)=\sum_{\ell\in\Z}\varphi_\ell\exp\Bigl(\frac{2\ell\pi i}{p}u\Bigr)
  \end{equation*}
  with
  \begin{equation*}
    \varphi_\ell = \Res[\Big]{\frac{\calX(s)}{s}}{s=1+\frac{2\ell\pi
        i}{p\log q}}
  \end{equation*}
  for $\ell\in\Z$.
  The Fourier series expansion of $\Phi_{\lambda k}$ for $\lambda\in\sigma(M)$
  with $1<\abs{\lambda}<q$ is given in Theorem~\ref{theorem:simple}.

  The Dirichlet series $\calY_{n_0}(s)$ satisfies the functional equation
    \begin{align*}
    (I-q^{-s}M)\calY_{n_0}(s) &= \sum_{n_0\le n<qn_0} n^{-s}y(n)
    + q^{-s}\sum_{0\le r<q}\zeta_{n_0}\bigl(s, \tfrac{r}{q}\bigr)o_r\\
    &\phantom{={}}+ q^{-s}\sum_{0\le r<q}P_r\sum_{k\ge 1}\binom{-s}{k}\Bigl(\frac rq\Bigr)^k\calY_{n_0}(s+k).
  \end{align*}
\end{corollary}
\begin{proof}
  The proof is split into several steps.

  \proofparagraph{Recursive Description}
  We set
  $v(n)=\bigl(\calT_1(n), \ldots, \calT_d(n), 1\bigr)^\top$.
  For $1\le j\le d$ and $0\le r<q$, we define $t(j, r)$ and $o(j, r)$ to be the
  target state and output label of the unique transition from
  state $j$ with input label $r$, respectively. Therefore,
  \begin{equation}\label{eq:transducer-to-matrix-product}
    \calT_j(qn+r) = \calT_{t(j, r)}(n) + o(j, r)
  \end{equation}
  for $1\le j\le d$, $n\ge 0$, $0\le r<q$ with $qn+r>0$.

  For $0\le r<q$, define $A_r=(a_{rjk})_{1\le j,\, k\le d+1}$ by
  \begin{equation*}
    a_{rjk} =
    \begin{cases}
      \iverson{t(j, r) = k}& \text{if $j$, $k\le d$,}\\
      o(j, r)& \text{if $j\le d$, $k=d+1$,}\\
      \iverson{k=d+1}& \text{if $j=d+1$.}
    \end{cases}
  \end{equation*}
  Then \eqref{eq:transducer-to-matrix-product} is equivalent to
  \begin{equation*}
    v(qn+r) = A_r v(n)
  \end{equation*}
  for $n\ge 0$, $0\le r<q$ with $qn+r>0$.

  \proofparagraph{$q$-Regular Sequence}
  If we insist on a proper formulation as a regular sequence, we rewrite
  \eqref{eq:transducer-to-matrix-product} to
  \begin{equation}\label{eq:transducer-to-regular-sequence}
    \calT_j(qn+r)= \calT_{t(j,r)}(n) + o(j, r) +
    \iverson{r=0}\iverson{n=0}(\calT_j(0)-\calT_{t(j,0)}(0)-o(j, 0))
  \end{equation}
  for $1\le j\le d$, $n\ge 0$, $0\le r<q$. Setting $\tildev(n)=(\calT_1(n), \ldots,
  \calT_d(n), 1, \iverson{n=0})$ and
  $\tildeA_r=(\tildea_{rjk})_{1\le j,\, k\le d+2}$ with
  \begin{equation*}
    \tildea_{rjk} =
    \begin{cases}
      \iverson{t(j, r) = k}& \text{if $j$, $k\le d$,}\\
      o(j, r)& \text{if $j\le d$, $k=d+1$,}\\
      \iverson{r=0}(\calT_j(0)-\calT_{t(j,0)}(0)-o(j, 0))& \text{if $j\le d$, $k=d+2$,}\\
      \iverson{k=d+1}& \text{if $j=d+1$,}\\
      \iverson{k=d+2}\iverson{r=0}& \text{if $j=d+2$,}
    \end{cases}
  \end{equation*}
  the system~\eqref{eq:transducer-to-regular-sequence} is equivalent to
  \begin{equation*}
    \tildev(qn+r) = \tildeA_r \tildev(n)
  \end{equation*}
  for $n\ge 0$, $0\le r<q$.

  The rest of the proof (relating the eigenvalues of $M$ with those of $C$) can
  be found in Appendix~\ref{appendix:transducers}.

\gdef\movetransducer{
  \proofparagraph{Eigenvalue~$1$}
  By construction, the matrices $A_r$ have the shape
  \begin{equation*}
    A_r = \left(
      \begin{array}{c|c}
        P_r&o_r\\\hline
        0&1
      \end{array}
      \right).
  \end{equation*}
  It is
  clear that $(0, \ldots, 0, 1)$ is a left eigenvector of $A_r$ associated with
  the eigenvalue~$1$.

  \proofparagraph{Joint Spectral Radius}
  We claim that $A_0, \ldots, A_{q-1}$ have joint spectral radius $1$. Let
  $\inftynorm{\,\cdot\,}$ denote the maximum norm of complex vectors as well as the induced
  matrix norm, i.e., the maximum row sum norm. Let $j_1$, \ldots,
  $j_\ell\in\set{0,\ldots, q-1}$. It is easily shown by induction on $\ell$
  that
  \begin{equation*}
    A_{j_1}\dotsm A_{j_\ell}=\left(
      \begin{array}{c|c}
        P&b\\\hline
        0&1
      \end{array}
\right)
  \end{equation*}
  for some $P\in\C^{d\times d}$ and $b\in\C^d$ with $\inftynorm{P}\le 1$ and $\inftynorm{b}\le \ell \max_{0\le
    r<q}\inftynorm{o_r}$.
  Thus, we obtain
  \begin{equation*}
    \inftynorm{A_{j_1}\dotsm A_{j_\ell}}\le 1+\ell\max_{0\le
      r<q}\inftynorm{o_r}.
  \end{equation*}
  As $1$ is an eigenvalue of each matrix~$A_r$ for $0\le r<q$,
  the joint spectral radius equals~$1$, which proves the claim.

  \proofparagraph{Eigenvectors and Asymptotics}
  We now consider $C=\sum_{0\le r<q}A_r$. It has the shape
  \begin{equation*}
    C = \left(
      \begin{array}{c|c}
        M&b\\\hline
        0&q
      \end{array}
      \right)
  \end{equation*}
  where $b$ is some complex vector.

  Let $w_1$, \ldots, $w_\ell$ be a linearly independent system of left
  eigenvectors of $M$ associated with the eigenvector $q$.
  If $w_j b=0$ for $1\le j\le \ell$, then $(w_1, 0)$,
  \ldots, $(w_\ell, 0), (0, 1)$ is a linearly independent system of left
  eigenvectors of $C$ associated with the eigenvalue $q$. In that case
  and because of Lemma~\ref{lemma:Perron--Frobenius-again},
  algebraic and geometric multiplicities of $q$ as an eigenvalue of $C$ are
  both equal to $\ell+1$.

  Otherwise, assume w.l.o.g.\ that $w_1 b=1$. Then
  \begin{equation*}
    (w_2 - (w_2 b)w_1, 0),\, \ldots,\, (w_\ell - (w_\ell b)w_1, 0),\, (0, 1)
  \end{equation*}
  is a linearly independent
  system of left eigenvectors of $C$ associated with the eigenvalue
  $q$. Additionally, $(w_1, 0)$ is a generalised left eigenvector of rank $2$
  of $C$ associated with the eigenvalue $q$ with $(w_1, 0)(C-qI)=(0, 1)$. As
  noted above, the vector
  $(0, 1)$ is a left eigenvector to each matrix $A_0$, \ldots, $A_{q-1}$.

  Similarly, it is easily seen that any left eigenvector of $M$ associated with
  some eigenvalue $\lambda\neq q$ can  be extended uniquely to a left
  eigenvector of $C$ associated with the same eigenvalue. The same is true for
  chains of generalised left eigenvectors associated with $\lambda\neq q$.

  Therefore, in both of the above cases, Theorem~\ref{theorem:contribution-of-eigenspace}
  yields
  \begin{equation*}
    \begin{aligned}
    \sum_{0\le n<N}\calT(N) = e_\calT N\log_q N &+ \sum_{\zeta} N^{\log_q
      (q\zeta)}\Phi_{q\zeta}(\fractional{\log_q N}) \\
    &+ \sum_{\substack{\lambda\in\sigma(M)\\
        1<\abs{\lambda}<q
      }} N^{\log_q \lambda} \sum_{0\le k<m(\lambda)}(\log_q N)^k\Phi_{\lambda
      k}(\log_q N)\\
    &+ \Oh[\big]{(\log N)^{\max\setm{m(\lambda)}{\abs{\lambda}=1}}}
    \end{aligned}
  \end{equation*}
  for some constant $e_\calT$ (which vanishes in the first case) and some
  $1$-periodic continuous functions $\Phi_{q\zeta}$ and $\Phi_{\lambda k}$ where $\zeta$ runs through
  the $p$th roots of unity and $\lambda$ through the eigenvalues of $M$ with
  $1<\abs{\lambda}<q$ and $0\le k<m(\lambda)$. Writing $N^{\log_q(q\zeta)}=N\zeta^{\log_q N}$
  and setting
  \begin{equation*}
    \Phi(u)\coloneqq \sum_\zeta \zeta^u\Phi_{\zeta q}(u)
  \end{equation*}
  leads to \eqref{eq:transducer:summatory-as-fluctuation}.

  \proofparagraph{Fourier Coefficients}
  By Theorem~\ref{theorem:simple}, we have
  \begin{equation*}
    \Phi_{\zeta q}(u)=\sum_{\ell\in\Z}\varphi_{(\zeta q)\ell}\exp(2\ell\pi i u)
  \end{equation*}
  with
  \begin{equation*}
    \varphi_{(\zeta q)\ell}=\Res[\Big]{\frac{\calT(0)+\calX(s)}{s}}{s=1+\log_q \zeta + \frac{2\ell\pi
        i}{\log q}}
  \end{equation*}
  for a $p$th root of unity $\zeta$ and $\ell\in\Z$. Writing
  $\zeta=\exp(2k\pi i/p)$ for a suitable $0\le k<p$, we get
  \begin{equation*}
    \Phi(u)=\sum_{\ell\in\Z} \sum_{0\le k<p} \Res[\Big]{\frac{\calT(0)+\calX(s)}{s}}{s=1 + \frac{2(\ell+\frac{k}{p})\pi
        i}{\log q}} \f[\Big]{\exp}{2\pi i \Bigl(\ell+\frac kp\Bigr)u}.
  \end{equation*}
  Replacing $\ell p+k$ by $\ell$ and noting that $\calT(0)$ does not contribute
  to the residue leads to the Fourier series given in the
  corollary.

  \proofparagraph{Functional Equation}
  By \eqref{eq:transducer-to-matrix-product}, we have
  \begin{align*}
    \calY_{n_0}(s) &= \sum_{n_0\le n<qn_0} n^{-s}y(n) + \sum_{n\ge
      n_0}\sum_{0\le r<q}(qn+r)^{-s}y(qn+r)\\
    &= \sum_{n_0\le n<qn_0} n^{-s}y(n) + \sum_{n\ge
      n_0}\sum_{0\le r<q}(qn+r)^{-s}\bigl(P_r y(n) + o_r\bigr)\\
    &= \sum_{n_0\le n<qn_0} n^{-s}y(n) + q^{-s}\sum_{0\le r<q}P_r
\sum_{n\ge
      n_0}\Bigl(n+\frac{r}{q}\Bigr)^{-s}y(n) \\
    &\hspace*{8.95em}
    + q^{-s}\sum_{0\le r<q}\zeta_{n_0}\bigl(s, \tfrac{r}{q}\bigr)o_r.
  \end{align*}
  Using Lemma~\ref{lemma:shifted-Dirichlet}
  yields the result.
} 
\end{proof}


\section{Pascal's Rhombus}
\label{sec:pascal}
We consider Pascal's rhombus~$\mathfrak{R}$ which is,
for integers~$i\geq0$ and $j$, the array with entries $r_{i,j}$, where
\begin{itemize}
\item $r_{0,j} = 0$ all $j$,
\item $r_{1,0}=1$ and $r_{1,j}=0$ for all $j\neq0$,
\item and
\begin{equation*}
  r_{i,j} = r_{i-1,j-1} + r_{i-1,j} + r_{i-1,j+1} + r_{i-2,j}
\end{equation*}
for $i \geq 1$.
\end{itemize}

\begin{figure}
  \centering
  \includegraphics[width=0.5\linewidth]{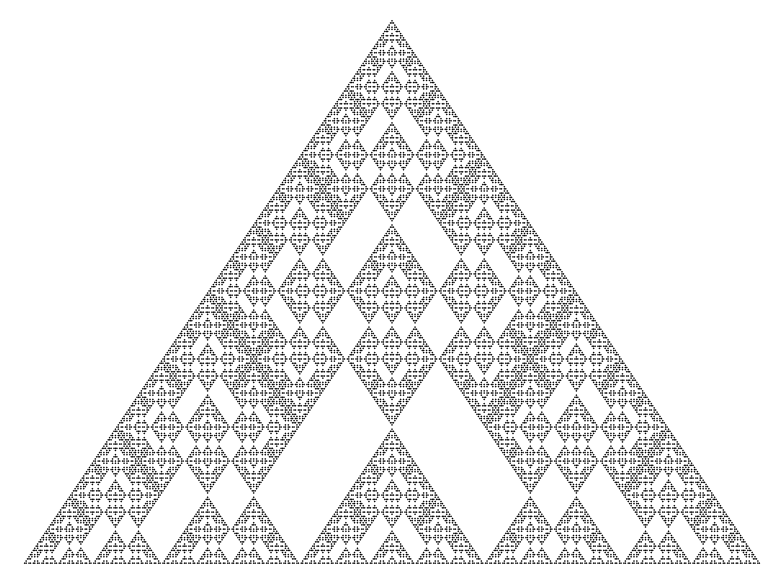}
  \caption{Pascal's rhombus modulo~$2$.}
  \label{fig:pascal-one}
\end{figure}

Let $\mathfrak{X}$ be equal to $\mathfrak{R}$ but with
entries takes modulo~$2$; see also Figure~\ref{fig:pascal-one}.
We partition $\mathfrak{X}$ into the four sub-arrays
\begin{itemize}
\item $\mathfrak{E}$ consisting only of the rows and columns of
  $\mathfrak{X}$ with even indices, i.e., the entries~$r_{2i, 2j}$,
\item $\mathfrak{Y}$ consisting only of the rows with odd indices and
  columns with even indices, i.e., the entries~$r_{2i-1, 2j}$,
\item $\mathfrak{Z}$ consisting only of the rows with even indices and
  columns with odd indices, i.e., the entries~$r_{2i, 2j-1}$, and
\item $\mathfrak{N}$ consisting only of the rows and columns with odd
  indices, i.e., the entries~$r_{2i-1, 2j-1}$.
\end{itemize}
Note that $\mathfrak{E} = \mathfrak{X}$ and $\mathfrak{N}=0$;
see~\cite{Goldwasser-Klostermeyer-Mays-Trapp:1999:Pascal-rhombus}.

\subsection{Recurrence Relations and $2$-Regular Sequences}
\label{sec:recurrences}

Let $X(N)$, $Y(N)$ and $Z(N)$ be the number of ones in the first $n$ rows
(starting with row index~$1$)
of $\mathfrak{X}$, $\mathfrak{Y}$ and $\mathfrak{Z}$ respectively.

\def\movepascalrecurrences{

Goldwasser, Klostermeyer, Mays and
Trapp~\cite[(12)--(14)]{Goldwasser-Klostermeyer-Mays-Trapp:1999:Pascal-rhombus}
get the recurrence relations
\begin{align*}
  X(N) &= X(\floor{\tfrac N2}) + Y(\ceil{\tfrac N2}) + Z(\floor{\tfrac N2}), \\
  Y(N) &= X(\ceil{\tfrac N2}) + X(\floor{\tfrac N2}-1) + Z(\floor{\tfrac N2}) + Z(\ceil{\tfrac N2}-1), \\
  Z(N) &= 2 X(\floor{\tfrac N2}) + 2 Y(\ceil{\tfrac N2}).
\end{align*}
for $N\ge2$, and $X(0)=Y(0)=Z(0)=0$, $X(1)=1$, $Y(1)=1$ and $Z(1)=2$
(cf.~\cite[Figures~2 and~3]{Goldwasser-Klostermeyer-Mays-Trapp:1999:Pascal-rhombus}).
Distinguishing between even and odd indices gives
\begin{align*}
  X(2N) &= X(N) + Y(N) + Z(N), \\
  X(2N+1) &= X(N) + Y(N+1) + Z(N), \\
  Y(2N) &= X(N) + X(N-1) + Z(N) + Z(N-1), \\
  Y(2N+1) &= X(N+1) + X(N-1) + 2Z(N), \\
  Z(2N) &= 2X(N) + 2Y(N), \\
  Z(2N+1) &= 2X(N) + 2Y(N+1)
\end{align*}
for all $N\ge1$.
Now we build the backward differences
$x(n) = X(n) - X(n-1)$, $y(n) = Y(n) - Y(n-1)$ and $z(n) = Z(n) - Z(n-1)$.
These $x(n)$, $y(n)$ and $z(n)$ are the number
of ones in the $n$th row of $\mathfrak{X}$, $\mathfrak{Y}$ and
$\mathfrak{Z}$ respectively and clearly
\begin{equation*}
  X(N) = \sum_{1\leq n \leq N} x(n),
  \qquad
  Y(N) = \sum_{1\leq n \leq N} y(n),
  \qquad
  Z(N) = \sum_{1\leq n \leq N} z(n).
\end{equation*}

} 

Using results by Goldwasser, Klostermeyer, Mays and
Trapp~\cite{Goldwasser-Klostermeyer-Mays-Trapp:1999:Pascal-rhombus}
leads to recurrence relations for the backward differences
$x(n) = X(n) - X(n-1)$, $y(n) = Y(n) - Y(n-1)$ and $z(n) = Z(n) - Z(n-1)$,
namely
\begin{subequations}
  \label{eq:rec-pascal-rhombus:main}
  \begin{align}
    x(2n)&=x(n)+z(n), &
    x(2n+1)&=y(n+1), \label{eq:rec-x}\\
    y(2n)&= x(n-1)+z(n), &
    y(2n+1)&=x(n+1) +z(n), \label{eq:rec-y}\\
    z(2n)&= 2x(n), &
    z(2n+1)&=2y(n+1) \label{eq:rec-z}
  \end{align}
\end{subequations}
for $n\ge1$, and $x(0)=y(0)=z(0)=0$, $x(1)=1$, $y(1)=1$ and $z(1)=2$.
(See Appendix~\ref{appendix:pascal:recurrence} for details.)

Let use write our coefficients as the
vector
\begin{equation}\label{eq:pascal:vec-v}
  v(n) = \bigl(x(n), x(n+1), y(n+1), z(n), z(n+1)\bigr)^\top.
\end{equation}
It turns out that the components included into $v(n)$ are
sufficient for a self-contained linear representation of~$v(n)$.
In particular, it is not necessary to include~$y(n)$.
By using the recurrences~\eqref{eq:rec-pascal-rhombus:main}, we find that
\begin{equation*}
  v(2n) = A_0 v(n)
  \qquad\text{and}\qquad
  v(2n+1) = A_1 v(n)
\end{equation*}
for all\footnote{ Note that $v(0) = A_0 v(0)$ and $v(1) = A_1 v(0)$ are indeed
  true.}
 $n\ge0$ with the matrices
\begin{equation*}
  A_0 =
  \begin{pmatrix}
    1 & 0 & 0 & 1 & 0 \\
    0 & 0 & 1 & 0 & 0 \\
    0 & 1 & 0 & 1 & 0 \\
    2 & 0 & 0 & 0 & 0 \\
    0 & 0 & 2 & 0 & 0
  \end{pmatrix}
  \qquad\text{and}\qquad
  A_1 =
  \begin{pmatrix}
    0 & 0 & 1 & 0 & 0 \\
    0 & 1 & 0 & 0 & 1 \\
    1 & 0 & 0 & 0 & 1 \\
    0 & 0 & 2 & 0 & 0 \\
    0 & 2 & 0 & 0 & 0
  \end{pmatrix},
\end{equation*}
and with $v(0) = (0,1,1,0,2)^\top$.
Therefore, the sequences $x(n)$, $y(n)$ and $z(n)$ are $2$-regular.

\subsection{Asymptotics}
\label{sec:asymptotics}

\begin{corollary}\label{corollary:pascal-rhombus:main}
  We have
  \begin{equation}\label{eq:pascal-rhombus:main-asy}
    X(N) = \sum_{1\leq n \leq N} x(n)
    = N^\kappa \f{\Phi}{\fractional{\log_2 N}} + \Oh{N \log_2 N}
  \end{equation}
  with $\kappa = \log_2 \bigl(3+\sqrt{17}\,\bigr)-1 = 1.83250638358045\ldots$ and
  a $1$-periodic function $\Phi$ which is Hölder continuous with
  any exponent smaller than $\kappa-1$.

  Moreover, we can effectively compute the Fourier coefficients
  of~$\Phi$.
\end{corollary}

We get analogous results for the sequences~$Y(N)$ and $Z(N)$ (each with
its own periodic function~$\Phi$, but the same exponent $\kappa$).
The fluctuation~$\Phi$ of $X(N)$ is visualized in Figure~\ref{fig:fluct-a} and its
first few Fourier coefficients are shown in Table~\ref{table:pascal-rhombus:fourier}.

\begin{figure}
  \centering
  \includegraphics[width=0.75\linewidth]{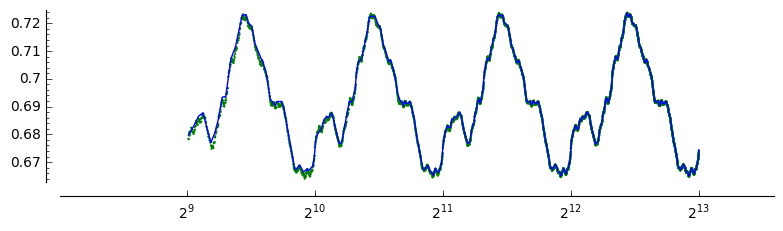}
  \caption{Fluctuation in the main term of the asymptotic expansion of $X(N)$.
    The figure shows $\f{\Phi}{\log_2 N}$ (blue) approximated by
    its trigonometric polynomial of degree~$99$ and
    $X(N) / N^\kappa$ (green).}
  \label{fig:fluct-a}
\end{figure}
\begin{table}
  \centering
  \begin{equation*}\footnotesize
    \begin{array}{r|l}
\multicolumn{1}{c|}{\ell} &
\multicolumn{1}{c}{\alpha_\ell} \\
\hline
0 & \phantom{-}0.6911615112341912755021246 \\
1 & -0.01079216311240407872950510 - 0.0023421761940286789685827i \\
2 & \phantom{-}0.00279378637350495172116712 - 0.00066736128659728911347756i \\
3 & -0.00020078258323645842522640 - 0.0031973663977645462669373i \\
4 & \phantom{-}0.00024944678921746747281338 - 0.0005912995467076061497650i \\
5 & -0.0003886698612765803447578 + 0.00006723866319930148568431i \\
6 & -0.0006223575988893574655258 + 0.00043217220614939859781542i \\
7 & \phantom{-}0.00023034317364181383130476 - 0.00058663168772856091427688i \\
8 & \phantom{-}0.0005339060804798716172593 - 0.0002119380802590974909465i \\
9 & \phantom{-}0.0000678898389770175928529 - 0.00038307823285486235280185i \\
10 & -0.00019981745997355255061991 - 0.00031394569060142799808175i \\
    \end{array}
  \end{equation*}
  \caption{Fourier coefficients of~$\Phi$
    (Corollary~\ref{corollary:pascal-rhombus:main}). All stated digits are
  correct.}
\label{table:pascal-rhombus:fourier}
\end{table}

At this point, we only prove~\eqref{eq:pascal-rhombus:main-asy} of
Corollary~\ref{corollary:pascal-rhombus:main}. We deal with the
Fourier coefficients in Appendix~\ref{sec:fourier}. As in the
introductory example of the binary sum-of-digits functions
(Example~\ref{example:binary-sum-of-digits}), we could get Fourier
coefficients by Theorem~\ref{theorem:simple} and the $2$-linear
representation of Section~\ref{sec:recurrences} directly. However,
the information in the vector~$v(n)$ (see \eqref{eq:pascal:vec-v})
is redundant with respect to the asymptotic main term as it contains
$x(n)$ and $z(n)$ as well as $x(n+1)$ and $z(n+1)$; both pairs are
asymptotically equal in the sense of~\eqref{eq:pascal-rhombus:main-asy}.
Therefore, we head for an only $3$-dimensional functional system of equations
for our Dirichlet series of $x(n)$, $y(n)$ and $z(n)$
(instead of a $5$-dimensional system).

\begin{proof}[Proof of~\eqref{eq:pascal-rhombus:main-asy}]
    We use Theorem~\ref{theorem:simple}.

    \proofparagraph{Joint Spectral Radius}
    First we compute the joint spectral radius $\rho$ of
    $A_0$ and $A_1$. Both matrices have a maximum absolute
    row sum equal to $2$, thus $\rho\leq 2$, and both
    matrices have~$2$ as an eigenvalue. Therefore we obtain
    $\rho=2$. Moreover, the finiteness property of the linear
    representation is satisfied by considering only products with exactly one
    matrix factor $A_0$ or $A_1$.

    Thus, we have $R=\rho=2$.

    \proofparagraph{Eigenvalues}
    Next, we compute the spectrum~$\sigma(C)$ of $C=A_0+A_1$. The
    matrix $C$ has the eigenvalues~$\lambda_1=\bigl(3+\sqrt{17}\,\bigr)/2=3.5615528128088\ldots$,
    $\lambda_2=2$, $\lambda_3=-2$, $\lambda_4=-1$ and
    $\lambda_5=\bigl(3-\sqrt{17}\,\bigr)/2=-0.5615528128088\ldots$ (each with multiplicity one).
    (Note that $\lambda_1$ and $\lambda_5$ are the zeros of the
    polynomial~$U^2-3U-U$.)

    \proofparagraph{Asymptotic Formula}
    By using Theorem~\ref{theorem:simple}, we obtain an
    asymptotic formula for $X(N-1)$. Shifting from $N-1$ to $N$ does not
    change this asymptotic formula, as this shift is absorbed by the
    error term $\Oh{n \log_2 N}$.
\end{proof}

\subsection{Dirichlet Series and Meromorphic Continuation}
\label{sec:meromorphic}

Let $n_0\ge2$ be an integer and define
\begin{align*}
  \f{\calX_{n_0}}{s} &= \sum_{n\geq n_0} \frac{x(n)}{n^s}, &
  \f{\calY_{n_0}}{s} &= \sum_{n\geq n_0} \frac{y(n)}{n^s}, &
  \f{\calZ_{n_0}}{s} &= \sum_{n\geq n_0} \frac{z(n)}{n^s}.
\end{align*}

\begin{lemma}\label{lemma:meromorphic}
  Set
  \begin{equation*}
    C = I -
    \begin{pmatrix}
      2^{-s} & 2^{-s} & 2^{-s} \\
      2^{1-s} & 0 & 2^{1-s} \\
      2^{1-s} & 2^{1-s} & 0 \\
    \end{pmatrix}.
  \end{equation*}
  Then
  \begin{equation}\label{eq:pascal:functional-equation}
    C
    \begin{pmatrix}
      \f{\calX_{n_0}}{s} \\ \f{\calY_{n_0}}{s} \\ \f{\calZ_{n_0}}{s}
    \end{pmatrix}
    =
    \begin{pmatrix}
      \f{\calJ_{n_0}}{s} \\ \f{\calK_{n_0}}{s} \\ \f{\calL_{n_0}}{s}
    \end{pmatrix}\!,
  \end{equation}
  where
  \begin{align*}
    \f{\calJ_{n_0}}{s} &= 2^{-s} \f{\Sigma}{s, -\tfrac12, \calY_{n_0}}
    + \calI_{\calJ_{n_0}}(s), \\
    &\;\calI_{\calJ_{n_0}}(s) = - \frac{y(n_0)}{(2n_0-1)^s}
    + \sum_{n_0\leq n<2n_0} \frac{x(n)}{n^s}, \\
    \f{\calK_{n_0}}{s} &=
    2^{-s} \f{\Sigma}{s, 1, \calX_{n_0}} + 2^{-s} \f{\Sigma}{s, -\tfrac12, \calX_{n_0}}
    + 2^{-s} \f{\Sigma}{s, \tfrac12, \calZ_{n_0}}
    + \calI_{\calK_{n_0}}(s), \\
    &\;\calI_{\calK_{n_0}}(s) = \frac{x(n_0-1)}{(2n_0)^s} - \frac{x(n_0)}{(2n_0-1)^s}
    + \sum_{n_0\leq n<2n_0} \frac{y(n)}{n^s}, \\
    \f{\calL_{n_0}}{s} &= 2^{1-s} \f{\Sigma}{s, -\tfrac12, \calY_{n_0}}
    + \calI_{\calL_{n_0}}(s), \\
    &\;\calI_{\calL_{n_0}}(s) = - \frac{2 y(n_0)}{(2n_0-1)^s}
    + \sum_{n_0\leq n<2n_0} \frac{z(n)}{n^s},
  \end{align*}
  with
  \begin{equation*}
    \f{\Sigma}{s, \beta, \calD} = \sum_{k\ge 1}
    \binom{-s}{k} \beta^k \calD(s+k)
  \end{equation*}
  provides meromorphic continuations
  of the Dirichlet series~$\f{\calX_{n_0}}{s}$, $\f{\calY_{n_0}}{s}$,
  and $\f{\calZ_{n_0}}{s}$ for $\Re s > \kappa_0=1$ with the only possible
  poles at $\kappa + \chi_\ell$ for $\ell\in\Z$,
  all of which are simple poles.
\end{lemma}

The proof of Lemma~\ref{lemma:meromorphic} can be found in
Appendix~\ref{appendix:pascal}. The idea is to rewrite the Dirichlet
series corresponding to \eqref{eq:rec-x}, \eqref{eq:rec-y} and
\eqref{eq:rec-z} to obtain the functional equation. The meromorphic
continuation uses Lemma~\ref{lemma:shifted-Dirichlet}; its poles come
from
\begin{equation*}
  \f{\Delta}{s} = \det C
  = 2^{-3s} (2^{2s} - 3\cdot 2^s - 2) (2^s + 2).
\end{equation*}

The Fourier coefficients (rest of
Corollary~\ref{corollary:pascal-rhombus:main}) can then be computed by
applying Theorem~\ref{theorem:simple}.

\def\moveproofpascalmeromorphic{
\begin{proof}[Proof of Lemma~\ref{lemma:meromorphic}]
  We split the proof into several steps.
  Note that rewriting $\f{\Sigma}{s, \beta, \calD}$ as binomial series
  is done in Lemma~\ref{lemma:shifted-Dirichlet}.

  \proofparagraph{Functional Equation}
  From \eqref{eq:rec-x} we obtain
  \begin{equation*}
    \f{\calX_{n_0}}{s} = \sum_{n_0\leq n<2n_0} \frac{x(n)}{n^s}
    + \sum_{n\geq n_0} \frac{x(n)}{(2n)^s}
    + \sum_{n\geq n_0} \frac{z(n)}{(2n)^s}
    + \sum_{n\geq n_0} \frac{y(n+1)}{(2n+1)^s}
  \end{equation*}
  The second and third summands become $2^{-s} \f{\calX_{n_0}}{s}$ and $2^{-s} \f{\calZ_{n_0}}{s}$.
  respectively, and we are left to rewrite the fourth summand. By
  using Lemma~\ref{lemma:shifted-Dirichlet} with $\beta=-1/2$ we
  get
  \begin{align*}
    \sum_{n\geq n_0} \frac{y(n+1)}{(2n+1)^s}
    &= 2^{-s} \sum_{n\geq n_0} \frac{y(n)}{(n-\frac12)^s}
    - \frac{y(n_0)}{(2n_0-1)^s} \\
    &= 2^{-s} \f{\calY_{n_0}}{s}
    + 2^{-s} \f{\Sigma}{s, -\tfrac12, \calY_{n_0}} - \frac{y(n_0)}{(2n_0-1)^s}.
  \end{align*}
  The first row of \eqref{eq:pascal:functional-equation} now follows.

  Similarly, from \eqref{eq:rec-y} we obtain
  \begin{equation}\label{eq:func:Ys}
    \begin{split}
    \f{\calY_{n_0}}{s} &= \sum_{n_0\leq n<2n_0} \frac{y(n)}{n^s}
    + \sum_{n\geq n_0} \frac{x(n-1)}{(2n)^s}
    + \sum_{n\geq n_0} \frac{z(n)}{(2n)^s} \\
    &\phantom{=}\hphantom{0}
    + \sum_{n\geq n_0} \frac{x(n+1)}{(2n+1)^s}
    + \sum_{n\geq n_0} \frac{z(n)}{(2n+1)^s} \\
    &= \sum_{n_0\leq n<2n_0} \frac{y(n)}{n^s}
    + 2^{-s} \sum_{n\geq n_0} \frac{x(n)}{(n+1)^s} + \frac{x(n_0-1)}{(2n_0)^s}
    + 2^{-s} \sum_{n\geq n_0} \frac{z(n)}{n^s} \\
    &\phantom{=}\hphantom{0}
    + 2^{-s} \sum_{n\geq n_0} \frac{x(n)}{(n-\frac12)^s} - \frac{x(n_0)}{(2n_0-1)^s}
    + 2^{-s} \sum_{n\geq n_0} \frac{z(n)}{(n+\frac12)^s}.
    \end{split}
  \end{equation}
  The second row of \eqref{eq:pascal:functional-equation} again
  follows by using Lemma~\ref{lemma:shifted-Dirichlet}.

  Similarly, \eqref{eq:rec-z} yields
  \begin{align*}
    \f{\calZ_{n_0}}{s} &= \sum_{n_0\leq n<2n_0} \frac{z(n)}{n^s}
    + 2 \sum_{n\geq n_0} \frac{x(n)}{(2n)^s}
    + 2 \sum_{n\geq n_0} \frac{y(n+1)}{(2n+1)^s} \\
    &= \sum_{n_0\leq n<2n_0} \frac{z(n)}{n^s}
    + 2^{1-s} \sum_{n\geq n_0} \frac{x(n)}{n^s}
    + 2^{1-s} \sum_{n\geq n_0} \frac{y(n)}{(n-\tfrac12)^s} - \frac{2 y(n_0)}{(2n_0-1)^s},
  \end{align*}
  and the third row of \eqref{eq:pascal:functional-equation} follows.

  \proofparagraph{Determinant and Zeros}
  The determinant of $C$ is
  \begin{equation*}
    \f{\Delta}{s} = \det C
    = 2^{-3s} \bigl(2^{2s} - 3\cdot 2^s - 2\bigr) \bigl(2^s + 2\bigr).
  \end{equation*}
  It is an entire function.

  All zeros of $\Delta$ are simple zeros.
  In particular, solving $\f{\Delta}{s} = 0$ gives $2^s = 3/2 \pm \sqrt{17}/2$ (the two zeros of $X^2-3X-2$) and $2^s = -2$.
  A solution $\f{\Delta}{s_0} = 0$
  implies that $s_0 + 2\pi i \ell/\log 2$ with $\ell\in\Z$ satisfies
  the same equation as well.

  Moreover, set $\kappa=\log_2 \bigl(3+\sqrt{17}\,\bigr) - 1 = 1.8325063835804\dots$.
  Then the only zeros with $\Re s > \kappa_0=1$ are at
  $\kappa + \chi_\ell$ with $\chi_\ell = 2\pi i \ell / \log 2$ for $\ell\in\Z$.

  It is no surprise that the $\kappa$ of this lemma and the $\kappa$
  in the proof of Corollary~\ref{corollary:main} which comes from the
  $2$-linear representation of Section~\ref{sec:recurrences} coincide.

  \proofparagraph{Meromorphic Continuation}
  The Dirichlet series
  $\calD_{n_0}\in\set{\calX_{n_0},\calY_{n_0},\calZ_{n_0}}$ is
  analytic for $\Re s > 2 = \log_2 \rho + 1$ with $\rho=2$ being the
  joint spectral radius by Theorem~\ref{theorem:Dirichlet-series}.
  We use the
  functional equation~\eqref{eq:pascal:functional-equation} which
  provides the continuation, as we write $\f{\calD_{n_0}}{s}$ in terms of
  $\f{\calJ_{n_0}}{s}$, $\f{\calK_{n_0}}{s}$ and $\f{\calL_{n_0}}{s}$.
  By Lemma~\ref{lemma:shifted-Dirichlet},
  these three functions are analytic for $\Re s > 1$.

  The zeros (all are simple zeros)
  of the denominator~$\f{\Delta}{s}$ are the only possibilities
  for the poles of $\f{\calD_{n_0}}{s}$ for $\Re s > 1$.
\end{proof}

} 

\def\movesectionpascalfourier{

\subsection{Fourier Coefficients}
\label{sec:fourier}

We are now ready to prove the rest of Corollary~\ref{corollary:main}.

\begin{proof}[Proof of Corollary~\ref{corollary:main}]
  We verify that we can apply Theorem~\ref{theorem:use-Mellin--Perron}.

  The steps of this proof in Section~\ref{sec:asymptotics} provided us
  already with an asymptotic
  expansion~\eqref{eq:pascal-rhombus:main-asy}. Lemma~\ref{lemma:meromorphic}
  gives us the meromorphic function for $\Re s>\kappa_0=1$ which comes from
  the Dirichlet series
  $\bigl(\f{\calX_{n_0}}{s}, \f{\calY_{n_0}}{s}, \f{\calZ_{n_0}}{s}\bigr)^\top$.
  It has simple poles at $\kappa + \chi_\ell$ for all $\ell\in\Z$ and
  satisfies the assumptions in
  Theorem~\ref{theorem:use-Mellin--Perron} by
  Theorem~\ref{theorem:Dirichlet-series} and
  Remark~\ref{remark:Dirichlet-series:bound}.

  Therefore a computation of the Fourier coefficients via computing
  residues (see \eqref{eq:Fourier-coefficient}) is possible by
  Theorem~\ref{theorem:use-Mellin--Perron}, and this residue may be
  computed from~\eqref{eq:pascal:functional-equation} via Cramer's
  rule.
\end{proof}

} 

\def\movesectionpascalexplicitbounds{

\subsection{Explicit Bounds and Computation of the Fourier Coefficients}
\label{sec:fourier-compute}

It turns out that it is more convenient to use the following set of
functional equations (in particular, they provide more stable
numerical calculations).

\begin{remark}\label{remark:modified-functional-equation}
  We modify the functional equation~\eqref{eq:pascal:functional-equation}
  of Lemma~\ref{lemma:meromorphic} in the following way.
  We can expand further on in Equation~\eqref{eq:func:Ys} and obtain
  \begin{align*}
    \sum_{n\geq n_0} \frac{x(n)}{(n+1)^s}
    &= \sum_{n_0\leq n<2n_0} \frac{x(n)}{(n+1)^s}
    + \sum_{n\geq n_0} \frac{x(2n)}{(2n+1)^s}
    + \sum_{n\geq n_0} \frac{x(2n+1)}{(2n+2)^s} \\
    &= \sum_{n_0\leq n<2n_0} \frac{x(n)}{(n+1)^s}
    + 2^{-s} \sum_{n\geq n_0} \frac{x(n)}{(n+\frac12)^s}
    + 2^{-s} \sum_{n\geq n_0} \frac{z(n)}{(n+\frac12)^s} \\
    &\phantom{=}\hphantom{0}
    + 2^{-s} \sum_{n\geq n_0} \frac{y(n+1)}{(n+1)^s} \\
    &= 2^{-s} \bigl( \f{\calX_{n_0}}{s} + \f{\calY_{n_0}}{s} + \f{\calZ_{n_0}}{s} \bigr)
    + \sum_{n_0\leq n<2n_0} \frac{x(n)}{(n+1)^s} \\
    &\phantom{=}\hphantom{0}
    + 2^{-s} \f{\Sigma}{s, \tfrac12, \calX_{n_0}}
    + 2^{-s} \f{\Sigma}{s, \tfrac12, \calZ_{n_0}} - \frac{y(n_0)}{(2n_0)^s}
  \end{align*}
  This changes our matrix~$C$ and the function~$\f{\calK_{n_0}}{s}$ which are now
  \begin{equation*}
    C = I -
    \begin{pmatrix}
      2^{-s} & 2^{-s} & 2^{-s} \\
      2^{-2s}+2^{-s} & 2^{-2s} & 2^{-2s}+2^{1-s} \\
      2^{1-s} & 2^{1-s} & 0 \\
    \end{pmatrix}
  \end{equation*}
  and
  \begin{align*}
    \f{\calK_{n_0}}{s} &= 2^{-2s} \f{\Sigma}{s, \tfrac12, \calX_{n_0}}
    + 2^{-s} \f{\Sigma}{s, -\tfrac12, \calX_{n_0}} \\
    &\hspace*{8.6em}
    + 2^{-s} \bigl(2^{-s}+1\bigr) \f{\Sigma}{s, \tfrac12, \calZ_{n_0}} + \calI_{\calK_{n_0}}(s), \\
    &\;\calI_{\calK_{n_0}}(s) = \frac{x(n_0-1)}{(2n_0)^s} - \frac{x(n_0)}{(2n_0-1)^s}
    - \frac{2^{-s} y(n_0)}{(2n_0)^s} \\
    &\phantom{\;\calI_{\calK_{n_0}}(s) =}\hphantom{0}
    + 2^{-s} \sum_{n_0\leq n<2n_0} \frac{x(n)}{(n+1)^s}
    + \sum_{n_0\leq n<2n_0} \frac{y(n)}{n^s},
  \end{align*}
  and we have $\calJ_{n_0}$ and $\calL_{n_0}$ as in Lemma~\ref{lemma:meromorphic}.
  Our Dirichlet series is now again the solution of
  \begin{equation*}
    C
    \begin{pmatrix}
      \f{\calX_{n_0}}{s} \\ \f{\calY_{n_0}}{s} \\ \f{\calZ_{n_0}}{s}
    \end{pmatrix}
    =
    \begin{pmatrix}
      \f{\calJ_{n_0}}{s} \\ \f{\calK_{n_0}}{s} \\ \f{\calL_{n_0}}{s}
    \end{pmatrix}\!.
  \end{equation*}
  Note that this new system does not influence the zeros of the
  determinant $\f{\Delta}{s} = \det C$ (see step ``Determinant and
  Zeros'' in the proof of Lemma~\ref{lemma:meromorphic}).

  Using Cramer's rule yields
\begin{align*}
  \f{\calX_{n_0}}{s} &= \frac{1}{\f{\Delta}{s}}
  \bigl( (-2^{1-3s} - 5\cdot2^{-2s} + 1) \f{\calJ_{n_0}}{s} \\
  &\hspace*{4em}
  + 2^{-s}(1 + 2^{1-s}) \f{\calK_{n_0}}{s}
  + 2^{-s}(1 + 2^{1-s}) \f{\calL_{n_0}}{s} \bigr), \\
  \f{\calY_{n_0}}{s} &= \frac{1}{\f{\Delta}{s}}
  \bigl( 2^{-s}(2^{1-2s} + 5\cdot2^{-s} + 1) \f{\calJ_{n_0}}{s} \\
  &\hspace*{4em}
  + (1 - 2^{1-s})(1 + 2^{-s}) \f{\calK_{n_0}}{s}
  + 2^{1-s} \f{\calL_{n_0}}{s} \bigr), \\
  \f{\calZ_{n_0}}{s} &= \frac{1}{\f{\Delta}{s}}
  \bigl( 2^{1-s}(2^{-s} + 1) \f{\calJ_{n_0}}{s} \\
  &\hspace*{4em}
  + 2^{1-s} \f{\calK_{n_0}}{s}
  + (1 - 2^{1-s})(1 + 2^{-s}) \f{\calL_{n_0}}{s} \bigr).
\end{align*}
\end{remark}

\begin{remark}\label{remark:modified-functional-equation-residue}
  Corollary~\ref{corollary:pascal-rhombus:main} provides an
  asymptotic expansion for $x(n)$ in
  \eqref{eq:pascal-rhombus:main-asy} containing a continuous and
  $1$-periodic function
  \begin{equation*}
    \f{\Phi}{u} = \sum_{\ell\in\Z} \varphi_\ell \fexp{2\ell\pi i u}.
  \end{equation*}
  By using the functional equation of
  Remark~\ref{remark:modified-functional-equation} and Cramer's rule,
  we get the expression
  \begin{align*}
    \varphi_\ell = \frac{1}{\f{\Delta'}{\kappa+\chi_\ell}}
    \bigl( &(-2^{1-3\kappa-3\chi_\ell} - 5\cdot2^{-2\kappa-2\chi_\ell} + 1) \f{\calJ_{n_0}}{\kappa+\chi_\ell} \\
    &+ 2^{-\kappa-\chi_\ell}(1 + 2^{1-\kappa-\chi_\ell}) \f{\calK_{n_0}}{\kappa+\chi_\ell} \\
    &+ 2^{-\kappa-\chi_\ell}(1 + 2^{1-\kappa-\chi_\ell}) \f{\calL_{n_0}}{\kappa+\chi_\ell} \bigr).
  \end{align*}
  for the $\ell$th Fourier coefficient. Here $\f{\Delta'}{\kappa+\chi_\ell}$
  denotes the derivative of
  \begin{equation*}
    \f{\Delta}{s} = \det C
    = -(2\cdot 2^{-2s} + 3\cdot 2^{-s} - 1) (2\cdot 2^{-s} + 1)
  \end{equation*}
  evaluated at $\kappa+\chi_\ell$.
\end{remark}

To evaluate the expression for $\varphi_\ell$, thus obtaining the
values in Table~\ref{table:pascal-rhombus:fourier}, we need explicit
bounds. These bounds are provided by the following lemmata.

\begin{lemma}\label{lem:bound-dirichlet}
  Let $n_0\geq2$, and set $\sigma=\Re s$ and suppose $\sigma>2$.
  Let $\calD_{n_0}\in\set{\calX_{n_0},\calY_{n_0},\calZ_{n_0}}$.
  Then
  \begin{equation*}
    \abs{\f{\calD_{n_0}}{s}} \leq \frac{2(n_0-1)^{2-\sigma}}{\sigma-2}.
  \end{equation*}
\end{lemma}

\begin{proof}
  The coefficients $x(n)$, $y(n)$ and $z(n)$ are bounded by $2n$
  which follows from $r_{i,j}=0$ for $\abs{j}>i$. With $d(n)$ being the
  coefficients corresponding to $\f{\calD_{n_0}}{s}$, we obtain
  \begin{align*}
    \abs{\f{\calD_{n_0}}{s}} &\leq \sum_{n\geq n_0} d(n) n^{-\sigma}
    \leq 2 \sum_{n\geq n_0} n^{1-\sigma} \\
    &\leq 2 \int_{n=n_0-1}^\infty n^{1-\sigma} \dd n
    = 2 \left.\frac{n^{2-\sigma}}{2-\sigma}\right\rvert_{n=n_0-1}^\infty
    = \frac{2(n_0-1)^{2-\sigma}}{\sigma-2},
  \end{align*}
  and the result follows.
\end{proof}

\begin{lemma}\label{lem:bound:shift-sum}
  Suppose $\Re s > 1$, and let $\beta \neq 0$ be a complex number
  and $n_0 > 1 + \abs{\beta}$ be an integer.
  With $\gamma_{k_0}=\abs{1+(s-1)/(k_0+1)}$, choose $k_0\in\N$ such that
  \begin{equation}\label{eq:bound:shift-sum:cond}
    \gamma_{k_0} < \frac{n_0-1}{\abs\beta}.
  \end{equation}
  Let $\calD_{n_0}\in\set{\calX_{n_0},\calY_{n_0},\calZ_{n_0}}$.
  Then
  \begin{equation*}
    \abs[\Bigg]{\sum_{k\geq k_0} \binom{-s}{k} \beta^k \f{\calD_{n_0}}{s+k}}
    \leq \frac{2\, \abs\beta^{k_0} (n_0-1)^{3-\sigma-k_0}}{
      (\sigma+k_0-2)(n_0-1-\gamma_{k_0}\abs\beta)}
    \,\abs[\bigg]{\binom{-s}{k_0}}.
  \end{equation*}
\end{lemma}

\begin{proof}
  Let $k \ge k_0$. We rewrite the binomial coefficient as
  \begin{equation*}
    \binom{-s}{k} = \binom{-s}{k_0}
    (-1)^{k-k_0} \prod_{k'=k_0+1}^k \left(1+\frac{s-1}{k'}\right).
  \end{equation*}
  As condition~\eqref{eq:bound:shift-sum:cond} and $\Re s\geq1$ imply
  \begin{equation*}
    \abs[\bigg]{1+\frac{s-1}{k'}} \leq \gamma_{k_0}
  \end{equation*}
  for all $k'\geq k_0+1$, we obtain the bound
  \begin{equation*}
    \abs[\bigg]{\binom{-s}{k}}
    \leq \abs[\bigg]{\binom{-s}{k_0}} \gamma_{k_0}^{k-k_0}.
  \end{equation*}
  Using the above estimate and Lemma~\ref{lem:bound-dirichlet} yields
  \begin{align*}
    \abs[\bigg]{\sum_{k\geq k_0} \binom{-s}{k} \beta^k \f{\calD_{n_0}}{s+k}}
    &\leq \abs[\bigg]{\binom{-s}{k_0}}
    \sum_{k\geq k_0} \gamma_{k_0}^{k-k_0} \abs\beta^k 
    \frac{2 (n_0-1)^{2-\sigma-k}}{\sigma+k-2} \\
    &\leq \frac{2\, \abs\beta^{k_0} (n_0-1)^{2-\sigma-k_0}}{\sigma+k_0-2}
    \,\abs[\bigg]{\binom{-s}{k_0}}
    \sum_{k\geq k_0} \left(\frac{\gamma_{k_0}\abs\beta}{n_0-1}\right)^{k-k_0}.
  \end{align*}
  The result follows by evaluating the geometric sum.
\end{proof}

Lemma~\ref{lem:bound:shift-sum} is the key for computing the Fourier
coefficients using reliable arithmetic.
Following an approach found in Grabner and Hwang~\cite{Grabner-Hwang:2005:digit}
and Grabner and Heuberger~\cite{Grabner-Heuberger:2006:Number-Optimal}, we
choose $k_0$ sufficiently large such that Lemma~\ref{lem:bound:shift-sum}
yields a sufficiently good bound for computing the Dirichlet series as
explained in Remark~\ref{remark:modified-functional-equation}.
We do this recursively: for large $\Re s$, the bounds of
Lemma~\ref{lem:bound:shift-sum} are sufficient to compute the value of the
Dirichlet series with high precision. Then the formul\ae{} in
Remark~\ref{remark:modified-functional-equation} are used recursively to
compute the series also for smaller $\Re s$. Finally, we compute the required
residues using Remark~\ref{remark:modified-functional-equation-residue}.
} 


\bibliography{bib/cheub}
\bibliographystyle{bibstyle/amsplainurl}

\clearpage
\appendix

\movesectionresults

\section{Additional Notations}\label{additional-notation}
We use Iverson's convention $\iverson{\mathit{expr}}=1$ if
$\mathit{expr}$ is true and $0$ otherwise, popularised by Graham, Knuth, and Patashnik~\cite{Graham-Knuth-Patashnik:1994}. We use the notation
$z^{\underline{\ell}}\coloneqq z(z-1)\dotsm (z-\ell+1)$ for falling factorials.
We use $\binom{n}{k_1, \dotsc, k_r}$ for multinomial coefficients. We sometimes
write a binomial coefficient $\binom{n}{a}$ as $\bibinom{n}{a}{b}$ with $a+b=n$ when we want
to emphasise the symmetry and analogy to a multinomial coefficient.

\section{Decomposition into Periodic Fluctuations: Proof of Theorem~\ref{theorem:contribution-of-eigenspace}}
\label{section:proof-contribution-of-eigenspace}
\subsection{Upper Bound for Eigenvalues of \texorpdfstring{$C$}{C}}
We start with an upper bound for the eigenvalues of $C$.
\begin{lemma}\label{lemma:eigenvalue-spectral-radius-bound}
  Let $\lambda\in\sigma(C)$. Then $\abs{\lambda}\le q\rho$.
\end{lemma}
\begin{proof}
  For $\ell\to\infty$, we have
  \begin{equation*}
    \abs{\lambda}^\ell
    \le \max \setm{\abs{\lambda}}{\lambda \in \sigma(C)}^\ell
    = \Oh[\big]{\norm{C^\ell}}
  \end{equation*}
  and
  \begin{equation*}
    \norm{C^\ell}
    \le \sum_{0\le r_1, \ldots, r_\ell<q}
    \norm{A_{r_1}\dotsm A_{r_\ell}}
    = \Oh{q^\ell R^\ell}
  \end{equation*}
  by \eqref{eq:bound-prod}. Taking $\ell$th roots and the limit $\ell\to\infty$
  yields $\abs{\lambda}\le qR$. This last inequality does not depend on
  our particular (cf.\@ Appendix~\ref{sec:definitions-notations}) choice
  of $R>\rho$, so the inequality is valid for all $R>\rho$, and
  we get the result.
\end{proof}

\subsection{Explicit Expression for the Summatory Function}
In this section, we give an explicit formula for $F(N)=\sum_{0\le n<N} f(n)$ in
terms of the matrices $A_r$, $B_r$ and $C$.

\begin{lemma}\label{lemma:explicit-summatory}
  Let $N$ be an integer with $q$-ary expansion
  $r_{\ell-1}\ldots r_0$. Then
  \begin{equation*}
    F(N)=\sum_{0\le j<\ell} C^j B_{r_j} A_{r_{j+1}}\dotsm
    A_{r_{\ell-1}} + \sum_{0\le j<\ell} C^j (I-A_0).
  \end{equation*}
\end{lemma}
\begin{proof}
  We claim that
  \begin{equation}\label{eq:sum-recursion}
    F(qN+r)=C F(N) + B_r f(N) + (I-A_0)\iverson{qN+r > 0}
  \end{equation}
  holds for non-negative integers $N$ and $r$  with $0\le r<q$.

  We now prove \eqref{eq:sum-recursion}: Using
  \eqref{eq:regular-matrix-sequence} and $f(0)=I$ yields
  \begin{align*}
    F(qN+r)
    &= f(0)\, \iverson{qN+r > 0}
    + \sum_{\substack{0<qn+r'<qN+r\\0\le n\\ 0\le r'<q}} f(qn+r')\\
    &= f(0)\, \iverson{qN+r > 0}
    + \sum_{\substack{0<qn+r'<qN+r\\0\le n\\ 0\le r'<q}} A_{r'}f(n)\\
    &= \bigl(f(0)-A_{0}f(0)\bigr) \iverson{qN+r > 0}
    + \sum_{\substack{0\le qn+r'<qN+r\\0\le n\\ 0\le r'<q}} A_{r'}f(n)\\
    &= (I-A_0) \iverson{qN+r > 0}
    + \sum_{0\le n<N}\sum_{0\le r'<q} A_{r'}f(n)
    + \sum_{0\le r'<r} A_{r'}f(N)\\
    &= (I-A_0) \iverson{qN+r > 0}
    + CF(N)+B_{r}f(N).
  \end{align*}
  This concludes the proof of \eqref{eq:sum-recursion}.

  Iteration of \eqref{eq:sum-recursion} and using~\eqref{eq:f-as-product} yield the assertion of the lemma,
  cf.~\cite[Proposition~3.6]{Heuberger-Kropf-Prodinger:2015:output}.
\end{proof}

\subsection{Proof of Theorem~\ref{theorem:contribution-of-eigenspace}}

\begin{proof}[Proof of Theorem~\ref{theorem:contribution-of-eigenspace}]
  For readability, this proof is split into several
  steps.

  \proofparagraph{Setting}
  Before starting the actual proof, we introduce the setting which will be used
  to define the fluctuations $\Phi_k$.
  We will first introduce
  functions $\Psi_k$ defined on the infinite product space
  \begin{equation*}
    \Omega\coloneqq
    \setm[\big]{\bfx=(x_0, x_1, \ldots)}{
      x_j\in\set{0, \ldots,q-1} \text{ for $j\ge 0$}, x_0\neq 0}.
  \end{equation*}
  We equip it with the metric such that two elements~$\bfx\neq\bfx'$ with
  a common prefix of length~$j$ and $x_j\neq x'_j$ have distance~$q^{-j}$.
  We consider the map $\val\colon \Omega\to [0, 1]$ with
  \begin{equation*}
    \val(\bfx) \coloneqq \log_q\sum_{j\ge 0}x_jq^{-j},
  \end{equation*}
  cf. Figure~\ref{fig:commutative-diagram}. By using the assumption that the
  zeroth component of elements of $\Omega$ is assumed to be non-zero, we easily check that $\val$ is
  Lipschitz-continuous; i.e.,
  \begin{equation}\label{eq:lval:Lipschitz}
    \abs[\big]{\val(\bfx)-\val(\bfx')} = \Oh{q^{-j}}
  \end{equation}
  for $\bfx\neq\bfx'$ with a common prefix of length~$j$.

  \begin{figure}[htbp]
    \centering
    \begin{tikzcd}
      \Omega \arrow{rr}{\Psi}\arrow[shift left=0.5ex]{rd}{\val}&&\C^d\\
      &{[0, 1]}\arrow{ru}{\Phi}\arrow[dotted,shift left=0.5ex]{lu}{\repr}
    \end{tikzcd}
    \caption{Maps}
    \label{fig:commutative-diagram}
  \end{figure}
  For $y\in[0, 1)$, let $\repr(y)$ be the unique $\bfx\in\Omega$ with
  $\val(\bfx)=y$ such that $\bfx$ does not end on infinitely many $q-1$'s, i.e.,
  $\repr(y)$ represents a $q$-ary expansion of $q^y$. This means that
  $\val\circ\repr$ is the identity on $[0, 1)$.

  From the definition of the metric on $\Omega$,
  recall that a function
  $\Psi\colon \Omega\to\C^d$ is continuous if and only if for each
  $\varepsilon>0$, there is a $j$ such that
  $\norm{\Psi(\bfx')-\Psi(\bfx)}<\varepsilon$ holds for all $\bfx$ and
  $\bfx'$ that have a common prefix of length $j$.
  Further recall from the universal property of quotients that if such a continuous function $\Psi$ satisfies
  $\Psi(\bfx)=\Psi(\bfx')$ whenever $\val(\bfx)=\val(\bfx')$, then there is a
  unique continuous function $\Phi\colon [0, 1]\to\C^d$ such that $\Phi\circ\val=\Psi$.
  This will be used in the ``Descent''-step of the proof.

  \proofparagraph{Notation}
  Let $N$ have the $q$-ary expansion
  $r_{\ell-1}\ldots r_0$ and set
  \begin{equation*}
    F_1(N) \coloneqq \sum_{0\le j<\ell} C^j B_{r_j} A_{r_{j+1}}\ldots
    A_{r_{\ell-1}}, \qquad F_2(N) \coloneqq \sum_{0\le j<\ell} C^j(I-A_0)
  \end{equation*}
  so that $F(N)=F_1(N)+F_2(N)$ by Lemma~\ref{lemma:explicit-summatory}.

  We consider the Jordan chain $w=v_{0}'$, \ldots, $v_{m-1}'$ generated by $w$,
  i.e., $v_k'=w(C-\lambda I)^k$ for $0\le k<m$ and $v_{m-1}'$ is a left eigenvector
  of $C$. Thus we have $wC^j=\sum_{0\le
    k<m}\binom{j}{k}\lambda^{j-k}v_k'$  for all $j\ge 0$.
  If $\lambda\neq 0$, choose vectors $v_0$, \ldots, $v_{m-1} \in \C^d$ such that
  \begin{equation}\label{eq:C-sum-eigenvectors}
    wC^j=\lambda^j\sum_{0\le k<m}j^kv_k
  \end{equation}
  holds for all $j\ge 0$. These vectors are
  suitable linear combinations of the vectors $v_0'$, \ldots,
  $v_{m-1}'$. We note that we have
  \begin{equation}\label{eq:v_m-1-expression}
    v_{m-1}=\frac1{\lambda^{m-1}(m-1)!}v_{m-1}'.
  \end{equation}

  \proofparagraph{Second Summand}
  We claim that
  \begin{multline}\label{eq:constant-term}
    wF_2(N) = wK + N^{\log_q \lambda}\sum_{0\le k<m} (\log_q
    N)^k\Phi^{(2)}_{k}(\fractional{\log_q N}) \\
    + (\log_q N)^mw\vartheta_m
    + \iverson{\lambda = 0} \Oh{N^{\log_q R}}
  \end{multline}
  for suitable continuously differentiable functions $\Phi^{(2)}_{ k}$ on $\R$,
  $0\le k<m$. If $R=0$, then $\Oh{N^{\log_q R}}$ shall mean that the error
  vanishes for almost all $N$.

  Consider first the case that $\lambda \neq 1$.
  Because of $wC^j=w{C'}^j$ and $wT^{-1}DT=w$, see Appendix~\ref{section:constants-for-theorem}, we have
  \begin{align*}
    wF_2(N)&=\sum_{0\le j<\ell}w{C'}^j (I-A_0)\\
    &=w(I-{C'}^\ell)(I-C')^{-1}(I-A_0) = wK - wC^\ell (I-C')^{-1}(I-A_0).
  \end{align*}
  If $\lambda=0$, then $wC^\ell=0$ for almost all $\ell$. We may set
  $\Phi^{(2)}_k=0$ for $0\le k<m$ and \eqref{eq:constant-term} is shown.
  Otherwise, as
  we have $\ell-1=\floor{\log_q N}=\log_q N - \fractional{\log_q N}$ and
  by~\eqref{eq:C-sum-eigenvectors}, we can
  rewrite $wC^\ell$ as
  \begin{equation*}
    wC^\ell=\lambda^{\ell}\sum_{0\le k'<m}\ell^{k'}
    v_{k'}=\lambda^{1+\log_q N-\fractional{\log_q N}}\sum_{0\le k'<m}(\log_q
    N+1-\fractional{\log_q N})^{k'} v_{k'}.
  \end{equation*}
  Let
  \begin{equation*}
    G_2(L, \nu)\coloneqq-\lambda^{1-\nu}\sum_{0\le k'<m}(L+1-\nu)^{k'} v_{k'}(I-C')^{-1}(I-A_0)
  \end{equation*}
  for reals $L$ and $\nu$,
  i.e.,
  \begin{equation*}
    wF_2(N)=wK + \lambda^{\log_q N} G_2(\log_q N, \fractional{\log_q N}).
  \end{equation*}
  By the multinomial theorem, we have
  \begin{equation*}
    G_2(L, \nu)=-\lambda^{1-\nu}\sum_{0\le k<m}L^k\sum_{\substack{0\le r,\ 0\le s\\ k+r+s<m}}\trinom{k+r+s}{k}{r}{s}(-\nu)^s v_{k+r+s}(I-C')^{-1}(I-A_0).
  \end{equation*}
  This leads to a representation
  $G_2(L, \nu)=\sum_{0\le k<m}L^k\Phi^{(2)}_{ k}(\nu)$ for continuously differentiable functions
  \begin{equation*}
    \Phi_k^{(2)}(\nu)=-\lambda^{1-\nu}\sum_{\substack{0\le r,\ 0\le s\\ r+s<m-k}}\trinom{k+r+s}{k}{r}{s}(-\nu)^s v_{k+r+s}(I-C')^{-1}(I-A_0)
  \end{equation*}
  for $0\le k<m$. As the functions~$\Phi^{(2)}_{k}$ are continuously differentiable,
  they are Lipschitz
  continuous on compact subsets of $\R$. We note that in the case $k=m-1$, the
  only occurring summand is with $r=0$ and $s=0$, which implies that
  \begin{equation}\label{eq:fluctuation-2-m-1}
    \Phi_{m-1}^{(2)}(\nu) = -\lambda^{1-\nu}v_{m-1}(I-C')^{-1}(I-A_0).
  \end{equation}
  Rewriting $\lambda^{\log_q N}$ as
  $N^{\log_q \lambda}$ and recalling that $w\vartheta_m=0$ yields \eqref{eq:constant-term} for $\lambda\neq 1$.

  We now turn to the case $\lambda=1$. We use $wC^j=\sum_{0\le
    k<m}\binom{j}{k}v_k'$ for $j\ge 0$ as above.
  Thus
  \begin{align*}
    wF_2(N) &= \sum_{0\le j<\ell}\sum_{0\le k<m}\binom{j}{k}v'_{k}(I-A_0)\\
    &= \sum_{0\le k<m}v'_k(I-A_0)\sum_{0\le j<\ell}\binom{j}{k}\\
    &= \sum_{0\le k<m}v'_k(I-A_0) \binom{\ell}{k+1},
  \end{align*}
  where the identity \cite[(5.10)]{Graham-Knuth-Patashnik:1994} (``summation on the upper index'')
  has been used in the last step.

  Thus $wF_2(N)$ is a polynomial in $\ell$ of degree $m$. By writing
  $\ell=1+\log_qN-\fractional{\log_q N}$, we can again rewrite this as a
  polynomial in $\log_q N$ whose coefficients depend on $\fractional{\log_q N}$.
  The coefficient of $(\log_q N)^m$ comes from
  $v_{m-1}'(I-A_0)\binom{\ell}{m}$, therefore, this coefficient is
  \begin{equation*}
    \frac1{m!}v_{m-1}'(I-A_0)=\frac1{m!}w(C-I)^{m-1}(I-A_0)=w\vartheta_m.
  \end{equation*}
  The additional factor $T^{-1}(I-D)T$ in $\vartheta_m$ has been introduced in order
  to annihilate generalised eigenvectors to other eigenvalues.  By construction
  of $K$, we have $wK=0$. Thus we have shown \eqref{eq:constant-term} for
  $\lambda=1$, too.

  \proofparagraph{Lifting the Second Summand}
  For later use---at this point, this may seem to be quite artificial---we
  set $\Psi^{(2)}_{k}=\Phi^{(2)}_{k}\circ \val$. As $\Phi^{(2)}_{k}$ is continuously differentiable, it
  is Lipschitz continuous on $[0, 1]$. As $\val$ is also Lipschitz continuous,
  so is $\Psi_k^{(2)}$.

  \proofparagraph{First Summand}
  We now turn to $wF_1(N)$.
  If $\lambda=0$, we certainly have $\abs{\lambda}\le R$ and we are in one of
  the first two cases of this theorem. Furthermore, we have
  $wC^j=0$ for $j\ge m$, thus
  \begin{equation*}
    wF_1(N)=\Oh[\bigg]{\sum_{0\le j<m} R^{\ell-j}}
    = \Oh{R^{\ell}}
    = \Oh{N^{\log_q R}}
  \end{equation*}
  by using~\eqref{eq:bound-prod}.
  Together with~\eqref{eq:constant-term}, the result follows.

  From now on, we may assume that $\lambda\neq 0$.
  By using~\eqref{eq:C-sum-eigenvectors}, we have
  \begin{equation}\label{eq:w-F-1-n}
    wF_1(N)=\sum_{0\le j<\ell} \lambda^j\biggl(\sum_{0\le k<m}j^kv_k \biggr)B_{r_j} A_{r_{j+1}}\ldots
    A_{r_{\ell-1}}.
  \end{equation}
  We first consider the case that $\abs{\lambda}<R$ (corresponding to
  Theorem~\ref{theorem:contribution-of-eigenspace}, \itemref{item:small-eigenvalue}). We get
  \begin{align*}
    wF_1(N) &= \Oh[\bigg]{\sum_{0\le j<\ell}\abs{\lambda}^j j^{m-1}
    R^{\ell-j}} \\
    &= \Oh[\bigg]{R^{\ell} \sum_{0\le j<\ell}
    j^{m-1}\Bigl(\frac{\abs{\lambda}}{R}\Bigr)^j}
    =\Oh{R^{\ell}}
    =\Oh{N^{\log_q R}},
  \end{align*}
  where \eqref{eq:bound-prod} was used.
  Together with \eqref{eq:constant-term}, the result follows.

  Next, we consider the case that $\abs{\lambda}=R$
  (Theorem~\ref{theorem:contribution-of-eigenspace}, \itemref{item:R-eigenvalue}).
  In that case, we get
  \begin{equation*}
    wF_1(N)=\Oh[\bigg]{\sum_{0\le j<\ell}\abs{\lambda}^j j^{m-1}
      R^{\ell-j}}
    = \Oh[\bigg]{R^\ell\sum_{0\le j<\ell}j^{m-1}}
    =\Oh{R^\ell \ell^m}.
  \end{equation*}
  Again, the result follows.

  From now on, we may assume that $\abs{\lambda}>R$. We set $Q\coloneqq
  \abs{\lambda}/R$ and note that $1<Q\le q$ by assumption and Lemma~\ref{lemma:eigenvalue-spectral-radius-bound}.
  We claim that there are continuous functions $\Psi^{(1)}_{k}$ on $\Omega$ for $0\le k<m$
  such that
  \begin{equation}\label{eq:first-term}
    wF_1(N) = N^{\log_q \lambda}\sum_{0\le k<m} (\log_q N)^k
    \f[\big]{\Psi^{(1)}_{k}}{\repr(\fractional{\log_q N})}
  \end{equation}
  and such that
  \begin{equation}\label{eq:quasi-Hoelder}
    \norm[\big]{\Psi^{(1)}_{ k}(\bfx)-\Psi^{(1)}_{ k}(\bfx')}=\Oh{j^{m-1} Q^{-j}}
  \end{equation}
  when the first $j$ entries of $\bfx$ and $\bfx'\in\Omega$ coincide.

    Write $N=q^{\ell-1+\fractional{\log_q N}}$ and let $\bfx=\repr(\fractional{\log_q N})$,
    i.e., $\bfx$ is the $q$-ary expansion of $q^{\fractional{\log_q N}}=N/q^{\ell-1}\in[1, q)$
    ending on infinitely many zeros. This means that $x_j=r_{\ell-1-j}$
    for $0\le j<\ell$ and $x_j=0$ for $j\ge \ell$.
    Reversing the order of summation in \eqref{eq:w-F-1-n} yields
    \begin{align*}
      wF_1(N)=\lambda^{\ell-1}\sum_{0\le j<\ell}\lambda^{-j}\biggl(\sum_{0\le k<m}(\ell-1-j)^kv_k \biggr)B_{x_j} A_{x_{j-1}}\ldots
    A_{x_0}.
    \end{align*}
    For $j\ge \ell$, we have $x_j=0$ and therefore $B_{x_j}=0$. Thus we may
    extend the sum to run over all $j\ge 0$, i.e.,
    \begin{equation*}
      wF_1(N)=\lambda^{\ell-1}\sum_{j\ge 0}\lambda^{-j}\biggl(\sum_{0\le k<m}(\ell-1-j)^kv_k \biggr)B_{x_j} A_{x_{j-1}}\ldots
    A_{x_0}.
    \end{equation*}
    We insert $\ell-1=\log_q N - \fractional{\log_q N}$ and obtain
    \begin{equation*}
      wF_1(N)=\lambda^{\log_q
        N}
      \f[\big]{G_1}{\log_q N, \repr(\fractional{\log_q N})}
    \end{equation*}
    where
    \begin{align*}
      G_1(L, \bfx)&=\lambda^{-\val(\bfx)}\sum_{j\ge 0}\lambda^{-j}\biggl(\sum_{0\le
        k<m}(L-\val(\bfx) - j)^kv_k \biggr)B_{x_j} A_{x_{j-1}}\ldots
    A_{x_0}\\
&=\lambda^{-\val(\bfx)}\sum_{j\ge 0}\lambda^{-j}\biggl(\sum_{\substack{0\le a,\ 0\le r,\
  0\le s\\a+r+s<m}}L^a
  (-j)^r \trinom{a+r+s}{a}{r}{s}\\&\hspace*{11.225em}\times\bigl(-\val(\bfx)\bigr)^{s}v_{a+r+s} \biggr)B_{x_j} A_{x_{j-1}}\ldots
    A_{x_0}
    \end{align*}
    for $L\in\R$ and $\bfx\in\Omega$. Note that in contrast to $G_2$, the second argument of $G_1$ is an element of
    $\Omega$ instead of $\R$.
    Collecting $G_1(L, \bfx)$ by powers of $L$, we get
    \begin{equation*}
      G_1(L, \bfx) = \sum_{0\le k<m} L^k \Psi^{(1)}_{ k}(\bfx)
    \end{equation*}
    where
    \begin{equation*}
      \Psi^{(1)}_k(\bfx) = \sum_{j\ge 0}\lambda^{-j}\sum_{0\le r<m-k}j^r
      \f[\big]{\psi_{kr}}{\val(\bfx)} B_{x_j}A_{x_{j-1}}\ldots A_{x_0}
    \end{equation*}
    for functions
    \begin{equation*}
      \psi_{kr}(\nu)=\lambda^{-\nu}
  (-1)^r\sum_{0\le s<m-k-r} \trinom{k+r+s}{k}{r}{s}(-\nu)^{s}v_{k+r+s}
    \end{equation*}
    which are continuously differentiable and therefore Lipschitz
    continuous on the unit interval.
    This shows \eqref{eq:first-term}.
    For $k=m-1$, only summands with $r=s=0$ occur, thus
    \begin{equation}\label{eq:fluctuation-1-m-1}
      \Psi_{m-1}^{(1)}(\bfx)=\sum_{j\ge 0}\lambda^{-j-\val(\bfx)}v_{m-1}B_{x_j}A_{x_{j-1}}\ldots A_{x_0}.
    \end{equation}

    Note that $\Psi^{(1)}_{ k}(\bfx)$ is majorised by
    \begin{equation*}
      \Oh[\bigg]{\sum_{j\ge 0} \abs{\lambda}^{-j} j^{m-1} R^{j}}
    \end{equation*}
    according to \eqref{eq:bound-prod}.
    We now prove \eqref{eq:quasi-Hoelder}. So let $\bfx$ and $\bfx'$ have a
    common prefix of length $i$. Consider the summand of $\Psi^{(1)}_k(\bfx)$ with index $j$.
    First consider the case that $j<i$. For all $r$, we have
    \begin{equation*}
      \norm[\big]{\f[\big]{\psi_{kr}}{\val(\bfx)}-\f[\big]{\psi_{kr}}{\val(\bfx')}}
      = \Oh{q^{-i}}
    \end{equation*}
    due to Lipschitz continuity of $\psi_{kr}\circ \val$.
    As the matrix product~$A_{x_{j-1}} \ldots A_{x_0}$
    is the same for $\bfx$ and $\bfx'$, the
    difference with respect to this summand is bounded by
    \begin{equation*}
      \Oh[\big]{\abs{\lambda}^{-j}j^{m-1}q^{-i}R^{j}}
      = \Oh{q^{-i}j^{m-1}Q^{-j}}.
    \end{equation*}
    Thus the total
    contribution of all summands with $j<i$ is $\Oh{q^{-i}}$.
    Any summand with $j \ge i$ is bounded by
    $\Oh[\big]{\abs{\lambda}^{-j}j^{m-1}R^{j}} = \Oh{j^{m-1}Q^{-j}}$,
    which leads to a total contribution of $\Oh{i^{m-1}Q^{-i}}$.
    Adding the two bounds leads to a bound of $\Oh{i^{m-1}Q^{-i}}$, as
    requested.

    \proofparagraph{Descent}
    By setting $\Psi_k(\bfx)=\Psi^{(1)}_{k}(\bfx)+\Psi^{(2)}_{k}(\bfx)$, we obtain
    \begin{equation}\label{eq:w-F-n}
      wF(N)=wK + N^{\log_q\lambda} \sum_{0\le k<m}(\log_q
      N)^k\Psi_k(\repr(\fractional{\log_q N})) +(\log_q N)^mw\vartheta_m
    \end{equation}
    and
    \begin{equation}\label{eq:Psi-continuity}
      \norm{\Psi_{k}(\bfx)-\Psi_{k}(\bfx')}=\Oh{j^{m-1}Q^{-j}}
    \end{equation}
    whenever $\bfx$ and $\bfx'\in\Omega$ have a common prefix of length $j$.

    It remains to show that $\Psi_k(\bfx)=\Psi_k(\bfx')$ holds whenever
    $\val(\bfx)=\val(\bfx')$ or $\val(\bfx)=0$ and $\val(\bfx')=1$.

    Choose $\bfx$ and $\bfx'$ such that one of the above
    two conditions on $\val$ holds and such that $x_j=0$ for $j\ge j_0$ and
    $x'_j=q-1$ for $j\ge j_0$. Be aware that now the prefixes of
    $\bfx$ and $\bfx'$ of length $j_0$ do not coincide except for the trivial
    case $j_0=0$.

    Fix some $j\ge j_0$ and set $\bfx''$ to be
    the prefix of $\bfx'$ of length $j$, followed by infinitely many zeros.
    Note that we have $q^{\val(\bfx'')}=q^{\val(\bfx')}-q^{-(j-1)}$. Set
    $n=q^{j-1+\val(\bfx'')}$. By construction, we have
    $n+1=q^{j-1+\val(\bfx)+\iverson{\val(\bfx)=0}}$. This implies
    $\repr(\fractional{\log_q n})=\bfx''$ and
    $\repr(\fractional{\log_q(n+1)})=\bfx$. Taking the difference of
    \eqref{eq:w-F-n} for $n+1$ and $n$ yields
    \begin{multline*}
      wf(n)=(n+1)^{\log_q \lambda} \sum_{0\le k<m}(\log_q (n+1))^k
      \Psi_k(\bfx)  - n^{\log_q \lambda} \sum_{0\le k<m} (\log_q n)^k
      \Psi_k(\bfx'') \\+\big((\log_q(n+1))^m-(\log_q n)^m\big)w\vartheta_m.
    \end{multline*}
    We estimate $n+1$ as $n(1+\Oh{1/n})$ and get
    \begin{equation}\label{eq:Tenenbaum-2}
      wf(n)=n^{\log_q \lambda }\sum_{0\le k<m} (\log_q n)^k
      \bigl(\Psi_k(\bfx)-\Psi_k(\bfx'')\bigr)
      + \Oh[\big]{n^{\log_q \abs{\lambda} -1}(\log n)^{m-1}}.
    \end{equation}
    We have $wf(n)=\Oh{R^j}=\Oh{R^{\log_q n}}=\Oh{n^{\log_q R}}$ by~\eqref{eq:f-as-product} and~\eqref{eq:bound-prod}. By \eqref{eq:Psi-continuity},
    \begin{equation*}
      \norm[\big]{\Psi_k(\bfx'')-\Psi_k(\bfx')}
      = \Oh[\big]{(\log n)^{m-1}n^{-\log_q Q}}
    \end{equation*}
    which is used below to replace $\bfx''$ by $\bfx'$.
    Inserting these estimates in \eqref{eq:Tenenbaum-2} and dividing by
    $n^{\log_q \lambda}$ yields
    \begin{equation}\label{eq:Tenenbaum-3}
      \sum_{0\le k<m}(\log_q n)^k(\Psi_k(\bfx')-\Psi_k(\bfx))
      = \Oh[\big]{n^{-\log_q Q} (\log n)^{2m-2}}.
    \end{equation}
    Note that $\Psi_k(\bfx')-\Psi_k(\bfx)$ does not depend on $j$. Now we let
    $j$ (and therefore $n$) tend to infinity. We see that
    \eqref{eq:Tenenbaum-3} can only remain true if
    $\Psi_k(\bfx')=\Psi_k(\bfx)$ for $0\le k<m$, which we had set out to show.

    Therefore, $\Psi_k$ descends to a continuous function $\Phi_k$ on $[0, 1]$ with
    $\Phi_k(0)=\Phi_k(1)$; thus $\Phi_k$ can be extended to a $1$-periodic continuous
    function.

    \proofparagraph{Hölder Continuity}For the proof of Hölder continuity, we
    first claim that for $0\le y<y'''<1$,
    we have
    \begin{equation}\label{eq:Hoelder-1}
      \norm[\big]{\Phi_k(y''')-\Phi_k(y)}
      = \Oh[\big]{(\log(q^{y'''}-q^{y}))^{m-1}(q^{y'''}-q^y)^{\log_q Q}}
    \end{equation}
    as $y'''\to y$.
    To prove this, let $\bfx\coloneqq \repr(y)$ and $\bfx'''\coloneqq
    \repr(y''')$.
    Let $\ell$ be the length of the longest common prefix of $\bfx$ and
    $\bfx'''$ and choose $j\ge 0$ such that $q^{-j}\le q^{y'''}-q^y< q^{-j+1}$.
    We define $\bfx'$ and $\bfx''\in\Omega$ such that
    \begin{alignat*}{4}
      \bfx&=(x_0,x_1,\ldots, x_{\ell-1}, x_{\ell},{}&& x_{\ell+1},{}&& x_{\ell+2},{}&& \ldots),\\
      \bfx'&=(x_0, x_1, \ldots, x_{\ell-1}, x_{\ell},{}&& q-1,{}&& q-1,{}&& \ldots),\\
      \bfx''&=(x_0, x_1, \ldots, x_{\ell-1}, x_{\ell}+1,{}&& 0,{}&& 0,{}&& \ldots),\\
      \bfx'''&=(x_0, x_1, \ldots, x_{\ell-1}, x'''_{\ell},{}&& x'''_{\ell+1},{}&&
      x'''_{\ell+2},{}&& \ldots)
    \end{alignat*}
    and set $y'\coloneqq\val(\bfx')$ and $y''\coloneqq\val(\bfx'')$.
    As $\val(\bfx)=y<y'''=\val(\bfx''')$, we have $x'''_\ell>x_\ell$. We
    conclude that $y\leq y'=y''\leq y'''$. Therefore,
    \begin{equation*}
      q^{y'}-q^{y} \le q^{y'''}-q^{y}< q^{-j+1},
    \end{equation*}
    so in view of the fact that each entry of
    $\bfx'$ is greater or equal than the corresponding entry of $\bfx$,
    the expansions $\bfx$ and $\bfx'$ must have a common prefix of length $j$.
    Similarly, the expansions $\bfx''$ and $\bfx'''$ must have a common prefix
    of length~$j$. Thus \eqref{eq:Psi-continuity} implies that
    \begin{align*}
      \norm[\big]{\Phi_k(y''')-\Phi_k(y)}
      &\le \norm[\big]{\Phi_k(y''')-\Phi_k(y'')}+
      \norm[\big]{\Phi_k(y')-\Phi_k(y)}\\
      &= \norm[\big]{\Psi_k(\bfx''')-\Psi_k(\bfx'')}+
      \norm[\big]{\Psi_k(\bfx')-\Psi_k(\bfx)}
      = \Oh{j^{m-1}Q^{-j}}.
    \end{align*}
    Noting that $-j = \log_q(q^{y'''}-q^y) + \Oh{1}$ leads to~\eqref{eq:Hoelder-1}.

    In order to prove Hölder continuity with exponent $\alpha<\log_q Q$, we first
    note that Lipschitz-continuity of $y\mapsto q^y$ on the interval $[0, 1]$
    shows that \eqref{eq:Hoelder-1} implies
    \begin{equation*}
      \norm[\big]{\Phi_k(y''')-\Phi_k(y)}
      = \Oh[\big]{(y'''-y)^{\alpha}}.
    \end{equation*}
    This can then easily be extended to arbitrary reals $y<y'''$
    by periodicity of $\Phi_k$ because it is sufficient to consider small
    $y'''-y$ and the interval may be subdivided at an integer between $y$ and $y'''$.

    \proofparagraph{Constant Dominant Fluctuation} Under the additional
    assumption that the vector~$w(C-I)^{m-1}=v_{m-1}'$ is a left eigenvector to all
    matrices $A_0$, \ldots, $A_{q-1}$ associated with the eigenvalue $1$, the same holds for $v_{m-1}$
    by~\eqref{eq:v_m-1-expression}. Then $v_{m-1}$ is also a left eigenvector of
    $C$ associated with the eigenvalue $q$. In particular, $\lambda=q\neq 1$.

    We can compute $\Phi_{m-1}^{(2)}(\nu)$ using
    \eqref{eq:fluctuation-2-m-1}. As $v_{m-1}\in W_{q}$, $v_{m-1}C=v_{m-1}C'$
    by definition of $C'$, see Appendix~\ref{section:constants-for-theorem}, which implies that
    $v_{m-1}(I-C')^{-1}=\frac1{1-q}v_{m-1}$. As $v_{m-1}(I-A_0)=0$ by
    assumption, we conclude that $\Phi_{m-1}^{(2)}(\nu)=0$ in this case.

    We use \eqref{eq:fluctuation-1-m-1} to compute $\Psi_{m-1}^{(1)}(\bfx)$.
    By assumption, $v_{m-1}B_{x_j}=x_j v_{m-1}$ which implies that
    \begin{equation*}
      \Psi_{m-1}^{(1)}(\bfx)
      = q^{-\val(\bfx)} \biggl(\sum_{j\ge 0}q^{-j}x_j\biggr) v_{m-1}
      =q^{-\val(\bfx)}q^{\val(\bfx)}v_{m-1}=v_{m-1}
    \end{equation*}
    by definition of $\val$.

    Together with \eqref{eq:v_m-1-expression}, we obtain the assertion.
\end{proof}

\subsection{Proof of Corollary~\ref{corollary:main}}\label{section:proof:corollary-main}

\begin{proof}[Proof of Corollary~\ref{corollary:main}]
  We denote the rows of $T$ as $w_1$, \ldots, $w_d$ and the columns of $T^{-1}$
  by $t_1$, \ldots, $t_d$. Thus $\sum_{j=1}^d t_jw_j=I$ and $w_j$ is a
  generalised left eigenvector of $C$ of some rank $m_j$ corresponding to some
  eigenvalue
  $\lambda_j\in\sigma(C)$. Theorem~\ref{theorem:contribution-of-eigenspace} and
  the fact that there are no eigenvalues of $C$ of absolute value between
  $\rho$ and $R$
  then immediately imply that
  \begin{align*}
    F(N) &= \sum_{j=1}^d t_j w_j F(N) \\
    &= K + \sum_{j=1}^d (\log_q N)^{m_j} t_jw_j \vartheta_{m_j}\\
    &\phantom{= K}\; + \sum_{\substack{1\le j\le d\\\abs{\lambda_j}>\rho }} N^{\log_q \lambda_j}
    \sum_{0\le k<m_j}(\log_q N)^k t_j\Psi_{jk}(\fractional{\log_q N}) \\
    &\phantom{= K}\; + \iverson{\exists \lambda\in\sigma(C)\colon \abs{\lambda}\le\rho}
    \Oh[\big]{N^{\log_q R}(\log N)^{\max\set{0}\cup \setm{m_j}{\abs{\lambda_j}=R}}}
  \end{align*}
  for some  $1$-periodic Hölder-continuous functions $\Psi_{jk}$ with exponent
  less than $\log_q\abs{\lambda_j}/R$. The first summand $K$ as well as the
  error term already coincide with the result stated in the corollary.
  From Appendix~\ref{section:constants-for-theorem} we recall that $w_j\vartheta_{m_j}=0$
  for $\lambda_j\neq 1$.

  We set
  \begin{equation*}
    \Phi_{\lambda k}(u)\coloneqq  \sum_{\substack{1\le j\le
        d\\\lambda_j=\lambda\\k<m_j}}
    \bigl(t_j\Psi_{jk}(u) +\iverson{\lambda=1}\iverson{m_j=k}t_jw_j\vartheta_{m_j}\bigr)
  \end{equation*}
  for $\lambda\in\sigma(C)$ with $\abs{\lambda}>\rho$ and $0\le k<m(\lambda)$.

  Then we still have to account for
  \begin{equation}\label{eq:phi-m-1-sum}
    (\log_q N)^{m(1)}\sum_{\substack{1\le j\le d\\\lambda_j=1\\m_j=m(1)}}t_jw_j\vartheta_{m(1)}.
  \end{equation}
  The factor $(C-I)^{m(1)-1}$ in the definition of $\vartheta_{m(1)}$ implies that
  $w_j\vartheta_{m(1)}$ vanishes unless $\lambda_j=1$ and $m_j=m(1)$. Therefore, the
  sum in \eqref{eq:phi-m-1-sum} equals $\vartheta$.
\end{proof}

\section{Meromorphic Continuation of the Dirichlet Series: Proof of
  Theorem~\ref{theorem:Dirichlet-series}}
\label{section:proof:Dirichlet-series}
For future use, we state an estimate for the binomial
coefficient. Unsurprisingly, it is a consequence of Stirling's formula.
\begin{lemma}\label{lemma:binomial-coefficient-asymptotics}
  Let $k\in\Z$, $k\ge 0$. Then
  \begin{equation}\label{eq:binomial-coefficient-estimate}
    \abs[\bigg]{\binom{-s}{k}}\sim \frac1{\Gamma(s)}k^{\Re s-1}
  \end{equation}
  uniformly for $s$ in a compact subset of $\C$ and $k\to\infty$.
\end{lemma}
\begin{proof}
  By \cite[(5.14)]{Graham-Knuth-Patashnik:1994} (``negating the upper index''), we rewrite the binomial coefficient as
  \begin{equation*}
    \binom{-s}{k}=(-1)^{k}\binom{s+k-1}{k}=\frac{(-1)^k}{\Gamma(s)}\frac{\Gamma(s+k)}{\Gamma(k+1)}.
  \end{equation*}
  Thus~\eqref{eq:binomial-coefficient-estimate} follows by \DLMF{5.11}{12}
  (which is an easy consequence of Stirling's formula for the Gamma function).
\end{proof}

\begin{proof}[Proof of Lemma~\ref{lemma:shifted-Dirichlet}]
  We have
  \begin{equation}\label{eq:Dirichlet-shifted:Sigma-as-diff}
    \f{\Sigma}{s, \beta, \calD}
    = \sum_{n\ge n_0} \bigl((n+\beta)^{-s}-n^{-s}\bigr) d(n)
  \end{equation}
  for $\Re s>\log_q R'+ 1$.
  We note that
  \begin{equation*}
    (n+\beta)^{-s} - n^{-s}
    = n^{-s}\Bigl(\Bigl(1+\frac{\beta}{n}\Bigr)^{-s} - 1 \Bigr)
    = \Oh[\big]{\abs{s}n^{-\Re s-1}}.
  \end{equation*}
  Therefore,
  \begin{equation*}
    \f{\Sigma}{s, \beta, \calD} = \Oh[\bigg]{\abs{s}\sum_{n\ge n_0}n^{\log_q R'-\Re s-1}},
  \end{equation*}
  and the series
  converges for $\Re s>\log_q R'$. As this holds for all $R'>\rho$, we obtain
  $\f{\Sigma}{s, \beta, \calD}=\Oh{\abs{\Im s}}$ as $\abs{\Im s}\to\infty$
  uniformly for $\log_q \rho + \delta \le \Re s \le \log_q \rho+\delta+1$.
  In the language of \cite[\S~III.3]{Hardy-Riesz:1915},
  $\f{\Sigma}{s, \beta, \calD}$ has order at most $1$ for $\log_q \rho + \delta \le \Re s \le \log_q
  \rho+\delta+1$. As $\log_q \rho+\delta+1$ is larger
  than the abscissa of absolute convergence of $\f{\Sigma}{s, \beta, \calD}$, it is clear that
  $\f{\Sigma}{s, \beta, \calD}=\Oh{1}$ for $\Re s=\log_q \rho+\delta+1$, i.e., $\f{\Sigma}{s, \beta, \calD}$ has order at most $0$ for
  $\Re s=\log_q \rho+\delta+1$. By Lindelöf's theorem
  (see \cite[Theorem~14]{Hardy-Riesz:1915}), we conclude that
  $\f{\Sigma}{s, \beta, \calD}=\Oh[\big]{\abs{\Im s}^{\mu_\delta(\Re s)}}$ for $\log_q \rho + \delta\le \Re s\le
  \log_q \rho +\delta+1$.

  For $\Re s > \log_q R' + 1$, we may
  rewrite~\eqref{eq:Dirichlet-shifted:Sigma-as-diff}
  using the binomial series as
  \begin{align}\label{eq:shifted-Dirichlet:diff:inner-sum}
    \f{\Sigma}{s, \beta, \calD} &=
    \sum_{n\ge n_0}{n^{-s}}\sum_{k\ge 1}\binom{-s}{k}\frac{\beta^k}{n^k} d(n)\notag\\
    &= \sum_{k\ge 1}
    \binom{-s}{k} \beta^k \sum_{n\ge n_0} n^{-(s+k)} d(n).
  \end{align}
  Switching the order of summation was legitimate because
  \begin{align*}
    \norm[\bigg]{\sum_{n\ge n_0} n^{-(s+k)} d(n)}
    &\le \sum_{n\ge n_0} n^{-(\Re s+k)} \norm{d(n)}\\
    &= \sum_{n\ge n_0} \Oh[\big]{n^{\log_q R'-\Re s -k}}
    = \Oh[\big]{n_0^{\log_q R'-\Re s-k+1}}
  \end{align*}
  for $\Re s+k>\log_q R'+1$ and Lemma~\ref{lemma:binomial-coefficient-asymptotics} imply absolute and
  uniform convergence for $s$ in a compact set.
  Noting that the previous arguments hold again for all $R'>\rho$ and that
  the inner sum in \eqref{eq:shifted-Dirichlet:diff:inner-sum}
  is $\calD(s+k)$ completes the proof.
\end{proof}

\begin{proof}[Proof of Theorem~\ref{theorem:Dirichlet-series}]
  As $f(n)=\Oh{R^{\log_q n}}=\Oh{n^{\log_q R}}$ by \eqref{eq:f-as-product} and
  \eqref{eq:bound-prod}, the Dirichlet series
  $\calF_{n_0}(s) = \sum_{n \ge n_0} n^{-s} f(n)$
  (see Appendix~\ref{sec:definitions-notations})
  converges absolutely and uniformly on compact sets for $\Re s>\log_q R+1$.
  As this holds for all $R>\rho$, i.e., does not depend on
  our particular (cf.\@ Appendix~\ref{sec:definitions-notations}) choice
  of $R>\rho$, this convergence result holds
  for $\Re s>\log_q \rho+1$.

  We use~\eqref{eq:regular-matrix-sequence} and
  Lemma~\ref{lemma:shifted-Dirichlet} (including its notation)
  to rewrite $\calF_{n_0}$ as
  \begin{align*}
    \calF_{n_0}(s) &= \sum_{n=n_0}^{qn_0-1}n^{-s}f(n) + \sum_{r=0}^{q-1 }\sum_{n\ge n_0} (qn+r)^{-s} f(qn+r)\\
    &= \sum_{n=n_0}^{qn_0-1} n^{-s}f(n) + q^{-s} \sum_{r=0}^{q-1} A_r
    \sum_{n\ge n_0} \Bigl(n+\frac{r}{q}\Bigr)^{-s} f(n)\\
    &= \sum_{n=n_0}^{qn_0-1} n^{-s}f(n) + q^{-s}
    \sum_{r=0}^{q-1} A_r \calF_{n_0}^{(r/q)}(s)\\
    &= \sum_{n=n_0}^{qn_0-1} n^{-s}f(n) + q^{-s}C\calF_{n_0}(s) + \calH_{n_0}(s)
  \end{align*}
  with
  \begin{equation*}
    \calH_{n_0}(s)\coloneqq q^{-s} \sum_{r=0}^{q-1}A_r \f[\big]{\Sigma}{s, \tfrac{r}{q}, \calF_{n_0}}
  \end{equation*}
  for $\Re s>\log_q R+ 1$.
  Thus
  \begin{equation}\label{eq:functional-equation-H}
    (I-q^{-s}C)\calF_{n_0}(s) = \sum_{n=n_0}^{qn_0-1}n^{-s}f(n)+\calH_{n_0}(s)
  \end{equation}
  for $\Re s>\log_q R+ 1$.
  By Lemma~\ref{lemma:shifted-Dirichlet} we have
  $\calH_{n_0}(s)=\Oh[\big]{\abs{\Im s}^{\mu_\delta(\Re s)}}$
  for $\log_q \rho + \delta\le \Re s\le
  \log_q \rho +\delta+1$.
  Rewriting the expression for $\calH_{n_0}(s)$ using the binomial series
  (see Lemma~\ref{lemma:shifted-Dirichlet} again) yields
  \begin{equation*}
    \calH_{n_0}(s)
    = q^{-s}\sum_{r=0}^{q-1} A_r \sum_{k\ge
      1}\binom{-s}{k}\Bigl(\frac{r}{q}\Bigr)^k \calF_{n_0}(s+k)
  \end{equation*}
  Combining this with~\eqref{eq:functional-equation-H}
  yields the expression~\eqref{eq:Dirichlet-recursion} for $\calG_{n_0}$.

  Solving \eqref{eq:analytic-continuation} for $\calF_{n_0}(s)$ yields the
  meromorphic continuation of $\calF_{n_0}(s)$ to $\Re s>\log_q R$ (and thus to $\Re
  s>\log_q \rho$) with possible poles where
  $q^s$ is an eigenvalue of $C$. As long as $q^s$ keeps a fixed positive
  distance $\delta$ from the eigenvalues, the bound for $\calG_{n_0}$
  (coming from the bound for $\calH_{n_0}$) carries over
  to a bound for $\calF_{n_0}$, i.e., \eqref{eq:order-F}.

  To estimate the order of the poles, let $w$ be generalised left eigenvector
  of rank $m$ of $C$ corresponding to an eigenvalue $\lambda$ with $\abs{\lambda}>R$. We claim that
  $w\calF_{n_0}(s)$ has a pole of order at most $m$ at $\log_q \lambda+\chi_k$ and no other
  poles for $\Re s>\log_q R$. We prove this by induction on $m$.

  Set $v\coloneqq w(C-\lambda I)$. By definition, $v=0$ or $v$ is a generalised
  eigenvector of rank $m-1$ of $C$. By induction hypothesis, $v\calF_{n_0}(s)$ has a
  pole of order at most $m-1$ at $\log_q \lambda+\chi_k$ for $k\in\Z$ and no other
  poles for $\Re s>\log_q R$.

  Multiplying \eqref{eq:analytic-continuation} by $w$,
  inserting the definition of~$v$ and reordering the summands yields
  \begin{equation*}
    (1 - q^{-s}\lambda)w\calF_{n_0}(s) = q^{-s}v \calF_{n_0}(s) + w\calG_{n_0}(s).
  \end{equation*}
  The right side has a pole of order at most $m-1$ at $\log_q \lambda+\chi_k$ for
  $k\in\Z$ and $1-q^{-s}\lambda$ has a simple zero at the same places. This
  proves the claim.
\end{proof}

\section{Fourier Coefficients: Proof of
  Theorem~\ref{theorem:use-Mellin--Perron}}
\label{section:proof:use-Mellin--Perron}
In contrast to the rest of this paper, this section does not directly relate to
a regular sequence but gives a general method to derive Fourier coefficients of
fluctuations.

\subsection{Pseudo-Tauberian Theorem}
In this section, we generalise a pseudo-Tau\-be\-rian argument by Flajolet, Grabner,
Kirschenhofer, Prodinger and
Tichy~\cite[Proposition~6.4]{Flajolet-Grabner-Kirschenhofer-Prodinger:1994:mellin}. In
contrast to their version, we allow for an additional logarithmic factor and
quantify the error in the case of a Hölder continuous function.
\begin{proposition}\label{proposition:pseudo-Tauber}
  Let $\kappa\in\C$ with $\Re \kappa > -1$, $q>1$ be a real number,  $k$ be a
  non-negative integer, $\Phi$ be a $1$-periodic Hölder continuous
  function with exponent $\alpha>0$. Then there exist continuously differentiable functions
  $\Psi_0$, \ldots, $\Psi_k$, periodic with period $1$, such that
  \begin{multline}
    \frac1{N^{\kappa+1}}\sum_{0\le n<N}n^\kappa (\log_q n)^k\Phi(\log_q n)
    =\sum_{j=0}^k \binom{k}{j}(\log_q N)^{k-j}\Psi_j(\log_q N)\\
    + \Oh[\big]{N^{-\alpha}(\log N)^k
      + N^{-(\Re\kappa+1)}(\log N)^{k+\iverson{\alpha=\Re\kappa+1}}}
    \label{eq:pseudo-Tauber-relation}
  \end{multline}
  for integers $N\to\infty$ and
  \begin{equation}
    \int_{0}^1 \Psi_j(u)\,\dd u = \frac{(-1)^{j}j!}{ (\log
      q)^{j}(1+\kappa)^{j+1}} \int_{0}^1 \Phi(u)\,\dd u
    \label{eq:pseudo-Tauber-Fourier}
  \end{equation}
  for $0\le j\le k$.
\end{proposition}
\begin{proof}\proofparagraph{Initial Simplification}
  Writing
  \begin{equation*}
    (\log_q n)^k = \sum_{j=0}^k \binom{k}{j}(\log_q N)^{k-j}(\log_q n - \log_qN)^j
  \end{equation*}
  shows that it is sufficient to consider sums
  \begin{equation*}
    S_j(N)\coloneqq  \frac1{N^{\kappa+1}}\sum_{0\le n<N}n^\kappa (\log_q n-\log_q N)^j\Phi(\log_q n).
  \end{equation*}
  We use the abbreviations $\Lambda\coloneqq \floor{\log_q N}$, $\nu\coloneqq
  \fractional{\log_q N}$ and $Q\coloneqq q^{\kappa+1}$, i.e.,
  $N=q^{\Lambda+\nu}$. Note that the assumptions imply that $\abs{Q}>1$.

  \proofparagraph{Construction of $\Psi_j$}
  Splitting the range of summation with respect to powers of $q$ yields
  \begin{align*}
    S_j(N) = \phantom{+\;}&\frac{1}{Q^{\Lambda+\nu}}
    \sum_{0\le p<\Lambda}\sum_{q^p\le n<q^{p+1}}n^\kappa
    (\log_q n-\log_q N)^j\Phi(\log_q n) \\
    +\; &\frac{1}{Q^{\Lambda+\nu}}
    \sum_{q^\Lambda\le n< q^{\Lambda+\nu}}n^{\kappa}(\log_q n-\log_q N)^j\Phi(\log_q n).
  \end{align*}
  We write $n=q^px$ (or $n=q^\Lambda x$ for the second sum), use the periodicity of $\Phi$ and get
  \begin{align*}
    S_j(N) = \phantom{+\;}&\frac{1}{Q^{\Lambda+\nu}}
    \sum_{0\le p<\Lambda}\sum_{\substack{x\in q^{-p}\Z\\ 1\le x < q}}
    \frac{Q^px^{\kappa}}{q^p} (p + \log_q x- \Lambda - \nu)^j\Phi(\log_q x) \\
    +\; &\frac{1}{Q^{\Lambda+\nu}}
    \sum_{\substack{x\in q^{-\Lambda}\Z\\ 1\le x < q^{\nu}}}
    \frac{Q^\Lambda x^{\kappa}}{q^\Lambda} (\log_q x-\nu)^j\Phi(\log_q x).
  \end{align*}
  We expand $p+\log_q x -\Lambda - \nu$ by the multinomial theorem and obtain
  \begin{align*}
    S_j(N) = \phantom{+\;}&\frac{1}{Q^{\nu}}
    \sum_{a+b+c=j}\trinom{j}{a}{b}{c}(-\nu)^b\sum_{0\le p<\Lambda}
    Q^{-(\Lambda-p)}(p-\Lambda)^c \\
    &\hspace*{14.25em}\times
    \sum_{\substack{x\in q^{-p}\Z\\ 1\le x < q}}
    x^\kappa(\log_q x)^a \Phi(\log_q x)\frac1{q^p} \\
    +\; &\frac{1}{Q^{\nu}}
    \sum_{a+b=j}\bibinom{j}{a}{b}(-\nu)^b
    \sum_{\substack{x\in q^{-\Lambda}\Z\\ 1\le x < q^{\nu}}}
    x^\kappa(\log_q x)^a \Phi(\log_q x)\frac1{q^\Lambda}.
  \end{align*}
  We write a multinomial coefficient $\bibinom{j}{a}{b}$ instead of the
  equivalent binomial coefficient $\binom{j}{a}$ in those locations where we
  want to emphasise the symmetry and the analogy to the multinomial coefficient $\trinom{j}{a}{b}{c}$.
  The inner sums are Riemann sums converging
  to the corresponding integrals for $p\to\infty$.
  We set
  \begin{equation*}
    I_a(u)\coloneqq\int_{1}^{q^u}x^{\kappa}(\log_q x)^a \Phi(\log_q x)\,\dd x.
  \end{equation*}
  It will be convenient to change variables $x=q^z$ in $I_a(u)$ to get
  \begin{equation}\label{eq:I_a-changed-variables}
    I_a(u)=(\log q)\int_{0}^{u}Q^z z^a \Phi(z)\,\dd z.
  \end{equation}
  We define the error~$\varepsilon_a(p, u)$ by
  \begin{equation*}
    \sum_{\substack{x\in q^{-p}\Z\\
        1\le x < q^u}}x^\kappa(\log_q x)^a \Phi(\log_q x)\frac1{q^p}=I_a(u) +
    \varepsilon_a(p, u).
  \end{equation*}
  By bounding~$\varepsilon_a(p, u)$ by the difference of upper and lower
  Darboux sums (step size~$q^{-p}$)
  corresponding to the integral~$I_a(u)$,
  and by Hölder continuity, we have $\varepsilon_a(p, u)=\Oh{q^{-\alpha p}}$
  uniformly in $0\le u\le 1$.
  This
  results in
  \begin{align*}
    Q^\nu S_j(N)&=\sum_{a+b+c=j}\trinom{j}{a}{b}{c}(-\nu)^bI_a(1)\sum_{0\le
      p<\Lambda}Q^{-(\Lambda-p)}(p-\Lambda)^c \\
    &\phantom{=}\;+ \sum_{a+b+c=j}\trinom{j}{a}{b}{c}(-\nu)^b\sum_{0\le
      p<\Lambda}Q^{-(\Lambda-p)}(p-\Lambda)^c \varepsilon_a(p, 1)\\
    &\phantom{=}\;+ \sum_{a+b=j}\bibinom{j}{a}{b}(-\nu)^b(I_a(\nu) + \varepsilon_a(\Lambda, \nu)).
  \end{align*}

  As $\Phi$ is Hölder continuous with exponent $\alpha$, we have
  \begin{equation*}
    \sum_{0\le
      p<\Lambda}Q^{-(\Lambda-p)}(p-\Lambda)^c \varepsilon_a(p, 1) = 
    \Oh[\bigg]{\sum_{0\le p< \Lambda} \abs{Q}^{-(\Lambda-p)}(\Lambda-p)^c q^{-\alpha p}}.
  \end{equation*}
  Replacing $p$ by $\Lambda-p$ on the right hand side of this equation yields
  \begin{equation*}
    \sum_{0\le p<\Lambda}Q^{-(\Lambda-p)}(p-\Lambda)^c \varepsilon_a(p, 1) =
    \Oh[\bigg]{q^{-\alpha \Lambda}\sum_{0< p\le \Lambda}
      p^c \Bigl(\frac{q^{\alpha}}{\abs{Q}}\Bigr)^p}.
  \end{equation*}
  We estimate the O-term in dependence of $q^\alpha/\abs{Q}$. If
  $q^\alpha<\abs{Q}$, then the sum converges and contributes
  $\Oh{q^{-\alpha \Lambda}}$. If $q^\alpha=\abs{Q}$, we estimate it by
  $\Oh{q^{-\alpha \Lambda} \Lambda^{c+1}}$ and if $q^\alpha>\abs{Q}$, then we use
  $\Oh{\Lambda^c \abs{Q}^{-\Lambda}}$. In total we obtain
  \begin{align*}
    \sum_{0\le p<\Lambda}Q^{-(\Lambda-p)}(p-\Lambda)^c \varepsilon_a(p, 1)
    &= \Oh[\big]{q^{-\alpha \Lambda}
      + \Lambda^c \abs{Q}^{-\Lambda}
      + q^{-\alpha \Lambda} \Lambda^{c+1}\iverson{q^\alpha=\abs{Q}}} \\
    &= \Oh[\big]{R_j(N)}
  \end{align*}
  with $R_j(N)\coloneqq N^{-\alpha} + N^{-(\Re\kappa+1)}(\log N)^{j+\iverson{\alpha=\Re\kappa+1}}$.

  Thus
  \begin{multline*}
    S_j(N)=\frac1{Q^\nu}\sum_{a+b+c=j}\trinom{j}{a}{b}{c}(-\nu)^bI_a(1)\sum_{0\le
      p<\Lambda}Q^{-(\Lambda-p)}(p-\Lambda)^c \\
    + \frac1{Q^\nu}\sum_{a+b=j}\bibinom{j}{a}{b}(-\nu)^bI_a(\nu)
    + \Oh[\big]{R_j(N)}
  \end{multline*}
  for $N\to\infty$.
  Replacing $p$ by $\Lambda-p$ in the first sum and then extending
  the sum to an infinite sum yields
  \begin{equation*}
    \sum_{0\le
      p<\Lambda}Q^{-(\Lambda-p)}(p-\Lambda)^c = \sum_{0<p\le
      \Lambda}Q^{-p}(-p)^c=\sum_{0<p}Q^{-p}(-p)^c + \Oh{Q^{-\Lambda}\Lambda^c}
  \end{equation*}
  for $N\to\infty$.
  Inserting this yields $S_j(N)=\Psi_j(\fractional{\log_q N}) + \Oh{R_j(N)}$ for $N\to \infty$ where
  \begin{multline*}
    \Psi_j(u)\coloneqq\frac1{Q^u}\sum_{a+b+c=j}\trinom{j}{a}{b}{c}(-u)^bI_a(1)\sum_{p\ge
      1}Q^{-p}(-p)^c \\
    + \frac1{Q^u}\sum_{a+b=j}\bibinom{j}{a}{b}(-u)^bI_a(u)
  \end{multline*}
  for $0\le u\le 1$. It is clear that $\Psi_j$ is a continuously differentiable
  function.

  \proofparagraph{Periodicity of $\Psi_j$}
  By splitting the trinomial coefficient and using the binomial theorem,
  we have
  \begin{align*}
    \sum_{a+b+c=j} \trinom{j}{a}{b}{c} I_a(1) (-u)^{b} (-p)^c
    &= \sum_{a=0}^{j} \binom{j}{a} I_a(1)
    \sum_{b=0}^{j-a} \binom{j-a}{b} (-u)^{b} (-p)^{(j-a)-b} \\
    &= \sum_{a=0}^{j} \binom{j}{a} I_a(1) (-u-p)^{j-a} \\
    &= \sum_{a+b=j} \bibinom{j}{a}{b} I_a(1) (-u-p)^b.
  \end{align*}
  Thus we can simplify the expression for $\Psi_j$ to
  \begin{equation}\label{eq:Psi_j-simplified}
    \Psi_j(u)=\frac{1}{Q^u}\sum_{a+b=j}\bibinom{j}{a}{b}\biggl(I_a(1)\sum_{p\ge
      1}Q^{-p}(-u-p)^{b} + I_a(u)(-u)^{b}\biggr).
  \end{equation}
  From this, it is easily seen that $\Psi_j(1)=\Psi_j(0)$. We calculate the
  derivative using the identity $b\bibinom{j}{a}{b}=j\bibinom{j-1}{a}{b-1}$ for
  $a+b=j$ and get
  \begin{multline*}
    \frac{\dd\Psi_j(u)}{\dd u} = - (\log Q) \Psi_j(u) \\ - \frac{j}{Q^u}\sum_{a+b=j}\bibinom{j-1}{a}{b-1}\biggl(I_a(1)\sum_{p\ge
      1}Q^{-p}(-u-p)^{b-1} + I_a(u)(-u)^{b-1}\biggr) \\+ \frac1{Q^u}\sum_{a+b=j}\bibinom{j}{a}{b} \frac{\dd I_a(u)}{\dd u}(-u)^b.
  \end{multline*}
  The second summand is $-j\Psi_{j-1}(u)$ (with the convention that
  $\Psi_{j-1}(u)=0$). By \eqref{eq:I_a-changed-variables}, we have
  $\frac{\dd I_a(u)}{\dd u} = (\log q) Q^u u^a \Phi(u)$. By inserting this and using the binomial
  theorem once more, we get
  \begin{align*}
    \frac{\dd\Psi_j(u)}{\dd u} &= -(\log Q)\Psi_j(u) - j\Psi_{j-1}(u) + (\log
                 q)\Phi(u)\sum_{a+b=j}\bibinom{j}{a}{b}u^a(-u)^b\\
    &= -(\log Q)\Psi_j(u) - j\Psi_{j-1}(u) + (\log
                 q)\Phi(u) \iverson{j=0}.
  \end{align*}
  From $\Psi_j(0)=\Psi_j(1)$, $\Psi_{j-1}(0)=\Psi_{j-1}(1)$ and
  $\Phi(0)=\Phi(1)$, we see that
  $\frac{\dd\Psi_j(u)}{\dd u}\bigr\vert_{u=0}=\frac{\dd\Psi_j(u)}{\dd u}\bigr\vert_{u=1}$.
  Thus $\Psi_j$ can be extended to a continuously differentiable $1$-periodic
  function on~$\R$.

  \proofparagraph{Mean}
  We now determine $\int_{0}^1 \Psi_j(u)\,\dd u$. For this aim, we note that
  \begin{equation*}
    \int (u+p)^bQ^{-u}\,\dd u = - \sum_{\ell=0}^b \frac{b^{\underline{\ell}}
      (u+p)^{b-\ell}Q^{-u}}{(\log Q)^{\ell+1}}
  \end{equation*}
  for integers $b\ge 0$ and
  falling factorials~$b^{\underline{\ell}}\coloneqq b(b-1)\dotsm (b-\ell+1)$,
  and we use \eqref{eq:I_a-changed-variables} and \eqref{eq:Psi_j-simplified}.
  Thus
  \begin{align*}
    \int_0^1 \Psi_j(u)\,\dd u
    &= \sum_{a+b=j}\bibinom{j}{a}{b} \biggl(I_a(1)
    (-1)^{b-1} \sum_{\ell=0}^{b}\frac{b^{\underline{\ell}}}{(\log Q)^{\ell+1}} \\
    &\hspace*{11.725em}\times
    \sum_{p\ge 1}\bigl((p+1)^{b-\ell}Q^{-(p+1)} -
    p^{b-\ell}Q^{-p}\bigr)\\
    &\hspace*{8.7em}
    + (\log q)\int_{0\le z\le u\le 1}Q^{z-u}z^a\Phi(z)(-u)^{b}\,\dd z\,\dd u\biggr)\!.
  \end{align*}
  The innermost sum (over $p\geq1$)
  is a telescoping sum and reduces to $-Q^{-1}$. Therefore,
  we obtain
  \begin{align*}
    \int_0^1 \Psi_j(u)\,\dd u
    &=\sum_{a+b=j}\bibinom{j}{a}{b} (-1)^{b}\biggl(I_a(1)
    \sum_{\ell=0}^{b}\frac{b^{\underline{\ell}}Q^{-1}}{(\log Q)^{\ell+1}} \\
    &\hspace*{9.125em}
      + (\log q)\int_{z=0}^1 Q^zz^a\Phi(z)\int_{u=z}^1  u^{b}Q^{-u}\,\dd u \,\dd z\biggr)\\
    &=\sum_{a+b=j} \bibinom{j}{a}{b}(-1)^{b}\sum_{\ell=0}^{b} 
    \frac{b^{\underline{\ell}}}{(\log Q)^{\ell+1}} \\
    &\hspace*{4em}\times
    \biggl(\frac{I_a(1)}{Q} - (\log q)\int_{0}^1 Q^z z^a \Phi(z)\bigl(Q^{-1}-z^{b-\ell}Q^{-z}\bigr)\,\dd z
      \biggr)\\
    &=(\log q)\sum_{b=0}^j\binom{j}{b}(-1)^b \sum_{\ell=0}^b
      \frac{b^{\underline{\ell}}}{(\log Q)^{\ell+1}}\int_{0}^1
      z^{j-\ell}\Phi(z)\,\dd z\\
    &=(\log q) \sum_{\ell=0}^{j}\frac{j!\,(-1)^{\ell}}{(j-\ell)!\, (\log
      Q)^{\ell+1}} \int_{0}^1 z^{j-\ell}\Phi(z)\,\dd z \\
    &\hspace*{14.75em}\times
      \sum_{b=\ell}^{j}\frac{(j-\ell)!}{(b-\ell)!\,(j-b)!}(-1)^{b-\ell}.
  \end{align*}
  Replacing $b-\ell$ by $b$ and using the binomial theorem once more
  the inner sum yields
  \begin{equation*}
    \sum_{b=\ell}^{j}\frac{(j-\ell)!}{(b-\ell)!\,(j-b)!}(-1)^{b-\ell}
    = \sum_{b=0}^{j-\ell}\frac{(j-\ell)!}{b!\,(j-b-\ell)!}(-1)^{b}
    = (1-1)^{j-\ell} = \iverson{j = \ell}.
  \end{equation*}
  This implies
  \begin{align*}
    \int_0^1 \Psi_j(u)\,\dd u
    &= (\log q) \sum_{\ell=0}^{j}\iverson{j = \ell}
    \frac{j!\,(-1)^{\ell}}{(j-\ell)!\, (\log Q)^{\ell+1}}
    \int_{0}^1 z^{j-\ell}\Phi(z)\,\dd z \\
    &= (\log q) \frac{j!(-1)^{j}}{ (\log
      Q)^{j+1}} \int_{0}^1 \Phi(z)\,\dd z.
  \end{align*}
  Replacing $Q$ by its definition yields the result.
\end{proof}

A straightforward application of Proposition~\ref{proposition:pseudo-Tauber} to
linear combinations with different values of $k$ provides the mean of the
functions $\Psi_j$ in terms of the original functions. For our application, we
need to rewrite this to express the mean of the original functions by the means
of the $\Psi_j$. Additionally, we also extract arbitrary Fourier coefficients
(instead of the zeroth Fourier coefficient) and we prove a uniqueness result.

We start with an auxiliary lemma which will provide uniqueness.

\begin{lemma}\label{lemma:uniqueness-fluctuations}
  Let $m$ be a positive integer, $q>1$ be a real number, $\kappa\in\C$ such
  that $\kappa\notin \frac{2\pi i}{\log q}\Z$, $c\in\C$ and $\Psi_0$, \ldots, $\Psi_{m-1}$ and $\Xi_0$,
  \ldots, $\Xi_{m-1}$ be $1$-periodic continuous functions such that
  \begin{equation}\label{eq:Fourier:function-comparison}
    \sum_{0\le k<m}(\log_qN)^k\Psi_k(\log_q N) = \sum_{0\le k<m}(\log_q N)^k
    \Xi_k(\log_q N) + c N^{-\kappa} + \oh{1}
  \end{equation}
  for integers $N\to\infty$. Then $\Psi_k=\Xi_k$ for $0\le k<m$.
\end{lemma}
\begin{proof}If $\Re \kappa <0$ and $c\neq 0$, then
  \eqref{eq:Fourier:function-comparison} is impossible as the growth of the
  right-hand side of the equation is larger than that on the left-hand side.
  So we can exclude this
  case from further consideration.
  We proceed indirectly and choose $k$ maximally such that $\Xi_k\neq\Psi_k$.
  Dividing \eqref{eq:Fourier:function-comparison} by $(\log_q N)^k$ yields
  \begin{equation}\label{eq:comparison}
    (\Xi_k-\Psi_k)(\log_q N) = cN^{-\kappa}\iverson{k=0} + \oh{1}
  \end{equation}
  for $N\to\infty$. Let
  $0< u<1$ and set $N_j=\floor{q^{j+u}}$. We
  clearly have $\lim_{j\to\infty} N_j=\infty$. Then
  \begin{equation*}
    j+u + \log_q(1-q^{-j-u}) = \log_q(q^{j+u}-1)\le \log_q N_j \le j+u.
  \end{equation*}
  We define $\nu_j\coloneqq \log_q N_j-j-u$ and see that $\nu_j=\Oh{q^{-j}}$ for
  $j\to\infty$, i.e., $\lim_{j\to\infty} \nu_j = 0$.
  This implies that $\lim_{j\to\infty}\fractional{\log_q N_j}=u$ and therefore
  \begin{equation*}
    \lim_{j\to\infty}(\Xi_k-\Psi_k)(\log_q N_j)=\lim_{j\to\infty}(\Xi_k-\Psi_k)(\fractional{\log_q N_j})=\Xi_k(u)-\Psi_k(u).
  \end{equation*}
  Setting $N=N_j$ in \eqref{eq:comparison} and letting $j\to \infty$ shows that
  \begin{equation}\label{eq:comparison-limit}
    \Xi_k(u)-\Psi_k(u) = \lim_{j\to\infty}cN_j^{-\kappa}\iverson{k=0}.
  \end{equation}
  If $k\neq 0$ or $\Re \kappa>0$, we immediately conclude that
  $\Xi_k(u)-\Psi_k(u)=0$. If $\Re \kappa<0$ we have
  $c=0$, which again implies that $\Xi_k(u)-\Psi_k(u)=0$.

  Now we assume that $\Re \kappa=0$ and $k=0$. We set
  $\beta\coloneqq -\frac{\log q}{2\pi i}\kappa$, which implies that
  $N^{-\kappa}=\exp(2\pi i \beta\log_q N)$. We choose sequences
  $(r_\ell)_{\ell\ge 1}$ and $(s_\ell)_{\ell\ge 1}$ such that
  $\lim_{\ell\to\infty }s_\ell=\infty$ and $\lim_{\ell\to\infty}\abs{s_\ell
    \beta - r_\ell}=0$: for rational $\beta=r/s$, we simply take $r_\ell=\ell
  r$ and $s_\ell=\ell s$; for irrational $\beta$, we consider the sequence of
  convergents $(r_\ell/s_\ell)_{\ell\ge 1}$ of the continued fraction of
  $\beta$ and the required properties follow from the theory of continued
  fractions, see for example \cite[Theorems~155
  and~164]{Hardy-Wright:1975}. By using $\log_q N_j = j+u+\nu_j$, we get
  \begin{align*}
    \lim_{\ell\to\infty}N_{s_\ell}^{-\kappa} &= \lim_{\ell\to\infty}\exp(2\pi i
    (r_\ell + \beta u + (s_\ell \beta - r_\ell) + \beta \nu_{s_\ell})=\exp(2\pi i \beta u),\\
    \lim_{\ell\to\infty}N_{s_\ell+1}^{-\kappa} &= \lim_{\ell\to\infty}\exp(2\pi i
    (r_\ell + \beta + \beta u + (s_\ell \beta - r_\ell)+\beta \nu_{s_\ell+1})=\exp(2\pi i \beta (1+u)).
  \end{align*}
  These two limits are distinct as $\beta\notin\Z$ by assumption.
  Thus $\lim_{j\to\infty}N_j^{-\kappa}$ does not exist. Therefore,
  \eqref{eq:comparison-limit} implies that $c=0$ and therefore $\Xi_k(u)-\Psi_k(u)=0$.

  We proved that $\Xi_k(u)=\Psi_k(u)$ for $u\notin\Z$. By continuity, this
  also follows for all $u \in \R$; contradiction.
\end{proof}

\begin{proposition}\label{proposition:pseudo-Tauber:combination}
  Let $\kappa\in\C$ with $\Re \kappa>-1$, $q>1$ be a real number, $m$ be a
  positive integer and $\Phi_0$, \ldots, $\Phi_{m-1}$ be $1$-periodic
  Hölder continuous functions with exponent $\alpha>0$. Then there exist uniquely defined $1$-periodic
  continuously differentiable functions $\Psi_0$, \ldots, $\Psi_{m-1}$ such that
  \begin{multline}\label{eq:pseudo-Tauber-relation:combination}
    \frac1{N^{\kappa+1}}\sum_{0\le n<N}n^\kappa \sum_{0\le k<m}(\log_q n)^k\Phi_k(\log_q n)
    =\sum_{0\le k<m} (\log_q N)^k\Psi_k(\log_q N) \\
    + \Oh[\big]{N^{-\alpha}(\log N)^{m-1}
      + N^{-(\Re\kappa+1)}(\log N)^{m-1+\iverson{\alpha=\Re\kappa+1}}}
  \end{multline}
  for integers $N\to\infty$.

  Denoting the Fourier coefficients of $\Phi_k$ and $\Psi_k$ by
  \begin{equation*}
    \varphi_{k\ell}\coloneqq \int_{0}^1 \Phi_k(u)\exp(-2\ell \pi i
    u)\,\dd u,\qquad
    \psi_{k\ell}\coloneqq \int_{0}^1 \Psi_k(u)\exp(-2\ell \pi i
    u)\,\dd u,
  \end{equation*}
  respectively, for $0\le k<m$ and $\ell\in\Z$ and setting $\psi_{m\ell}=0$ for $\ell\in\Z$, we have
  \begin{equation}\label{eq:Fourier-recurrence}
    \varphi_{k\ell}= (\kappa+1+\chi_\ell)\psi_{k\ell}
    + \frac{(k+1)}{\log q} \psi_{(k+1)\ell}
  \end{equation}
  for $0\le k<m$ and $\ell\in\Z$.
\end{proposition}
\begin{proof}\proofparagraph{Uniqueness}Uniqueness is a direct consequence of
  Lemma~\ref{lemma:uniqueness-fluctuations}.

  \proofparagraph{Existence}Existence and continuous differentiability of
  $\Psi_0$, \ldots, $\Psi_{m-1}$ are an immediate consequence of
  Proposition~\ref{proposition:pseudo-Tauber}. It also follows from setting
  $\ell=0$ in \eqref{eq:Fourier:second:second-version} below.

  \proofparagraph{Relation for Fourier Coefficients} For computing the Fourier coefficients,
  we do not apply Proposition~\ref{proposition:pseudo-Tauber} directly, but we
  apply it to a shifted version. Let $\ell\in\Z$ be fixed throughout the rest
  of the proof and recall that $\chi_\ell=(2\ell\pi i)/(\log q)$. Set
  \begin{equation*}
    S(N)\coloneqq \frac1{N^{\kappa+1}}\sum_{0\le n<N}n^\kappa \sum_{0\le k<m}(\log_q n)^k\Phi_k(\log_q n).
  \end{equation*}
  Proposition~\ref{proposition:pseudo-Tauber} with substitutions $\kappa\leftarrow
  \kappa+\chi_\ell$ and $\Phi(u)\leftarrow \Phi_k(u)\exp(-2\ell\pi i u)$ and
  summation over $0\le k<m$ shows that
  there are $1$-periodic continuous functions $\Psi_{kj}$ for $0\le j\le k<m$ such that
  \begin{align}
    \frac1{N^{\chi_\ell}}S(N)&=
    \frac{1}{N^{\kappa+1+\chi_\ell}}\sum_{0\le n<N}n^{\kappa+\chi_\ell}
    \sum_{0\le k<m}(\log_q n)^k\Phi_k(\log_q n)\exp(-2\ell\pi i \log_q n) \notag\\
                             &=
    \sum_{0\le k<m}\sum_{j=0}^k \binom{k}{j}
    (\log_q N)^{k-j}\Psi_{kj}({\log_q
                               N})+\oh{1}\label{eq:Fourier:second:first-version}
  \end{align}
  with
  \begin{equation}\label{eq:Fourier:second:integral}
    \int_{0}^1 \Psi_{kj}(u)\,\dd u = \frac{(-1)^jj!}{(\log q)^j
                                  (\kappa+1+\chi_\ell)^{j+1}}\int_{0}^1
                                  \Phi_k(u)\exp(-2\ell\pi i u)\,\dd u,
  \end{equation}
  where the integral on the right-hand side equals the Fourier
  coefficient~$\varphi_{kl}$ by definition.
  Replacing $j$ by $k-j$ in \eqref{eq:Fourier:second:first-version} and
  switching the order of summation yields
  \begin{equation}\label{eq:Fourier:second:second-version}
    \frac{1}{N^{\chi_\ell}}S(N) =
    \sum_{0\le j<m}(\log_q N)^{j}\Xi_j(\log_q N)+\oh{1}
  \end{equation}
  with
  \begin{equation}\label{eq:Xi-definition}
    \Xi_j(u) \coloneqq \sum_{j\le k <m} \binom{k}{j}
    \Psi_{k(k-j)}(u)
  \end{equation}
  for $0\le j<m$ and $u\in\R$. By construction, $\Xi_j$ is $1$-periodic and
  continuously differentiable. Multiplying~\eqref{eq:Fourier:second:second-version} by $N^{\chi_\ell} = \exp(2\ell\pi i \log_q N)$ and the uniqueness obtained in
  Lemma~\ref{lemma:uniqueness-fluctuations} show that
  \begin{equation*}
    \Psi_j(u) = \Xi_j(u)\exp(2\ell\pi i u)
  \end{equation*}
  for $0\le j<m$ and $u\in\R$.

  By integration and using \eqref{eq:Xi-definition} and \eqref{eq:Fourier:second:integral}, we get
  \begin{equation}\label{eq:psi-phi-equation}
    \begin{aligned}
    \psi_{j\ell} &=\int_{0}^1 \Psi_j(u)\exp(-2\ell \pi i u)\,\dd u
    = \int_{0}^1 \Xi_j(u)\,\dd u \\
    &= \sum_{j\le k <m}\binom{k}{j}\int_{0}^1\Psi_{k(k-j)}(u)\,\dd u\\
    &= \sum_{j\le k<m}\binom{k}{j}\frac{(-1)^{k-j}(k-j)!}{(\log
      q)^{k-j}(\kappa+1+\chi_\ell)^{k-j+1}}\varphi_{k\ell}
    \end{aligned}
  \end{equation}
  for $0\le j<m$.

  \proofparagraph{Solving the Linear System} We have to solve the system \eqref{eq:psi-phi-equation} for
  $\varphi_{k\ell}$. To do so, we multiply \eqref{eq:psi-phi-equation} by
  $(-1)^jj!(\log q)^{-j}(\kappa+1+\chi_\ell)^{-j+1}$ and get
  \begin{equation*}
    \frac{(-1)^j j!}{(\log q)^j (\kappa+1+\chi_\ell)^{j-1}} \psi_{j\ell} = \sum_{j\le
      k<m}\frac{(-1)^{k}k!}{(\log q)^k (\kappa+1+\chi_\ell)^{k}}
    \varphi_{k\ell}
  \end{equation*}
  for $0\le j<m$. The equation remains valid for $j=m$. Replacing $j$ by $j+1$ and taking the difference yields
  \begin{multline*}
    \frac{(-1)^j j!}{(\log q)^j (\kappa+1+\chi_\ell)^{j-1}}\psi_{j\ell} +
    \frac{(-1)^j (j+1)!}{(\log q)^{j+1}(\kappa+1+\chi_\ell)^j}\psi_{(j+1)\ell}
    \\
    = \frac{(-1)^{j}j!}{(\log q)^j (\kappa+1+\chi_\ell)^j}
    \varphi_{j\ell}
  \end{multline*}
  for $0\le j<m$.
  Cancelling the common factors and replacing $j$ by $k$ yields \eqref{eq:Fourier-recurrence}.
\end{proof}

\subsection{Proof of Theorem~\ref{theorem:use-Mellin--Perron}}
The idea of the proof is to compute the repeated summatory function of $F$
twice: On the one hand, we use the pseudo-Tauberian Proposition~\ref{proposition:pseudo-Tauber:combination} to rewrite the
right hand side of \eqref{eq:F-N-periodic} in terms of periodic
functions~$\Psi_{jk}$. On the other hand, we compute it using a higher
order Mellin--Perron summation
formula, relating it to the singularities of $\calF$.
More specifically, the expansions at the singularities of $\calF$ give
the Fourier coefficients of $\Psi_{jk}$. The Fourier coefficients of the functions
$\Psi_{jk}$ are related to those of the functions $\Phi_k$ via~\eqref{eq:Fourier-recurrence}.
\begin{proof}[Proof of Theorem~\ref{theorem:use-Mellin--Perron}]
  \proofparagraph{Initial observations and notations}
  As $\Phi_k$ is Hölder-continuous, its Fourier series converges by Dini's
  criterion, see e.g. \cite[p.~52]{Zygmund:2002:trigon}.

  We define $\varphi_{k\ell}\coloneqq \int_{0}^1
  \Phi_k(u)\exp(-2\ell\pi i u)\,\dd u$ to be the $\ell$th Fourier coefficient of
  $\Phi_k$ for $\ell\in\Z$
  and we will prove~\eqref{eq:Fourier-coefficient}. For any sequence $(g(n))_{n\ge
  1}$, we set $(\calS g)(N)\coloneqq \sum_{1\le n<N}g(n)$.
  We set $J=1 + \max\set{\floor{\eta}, 0}$. In particular, $J$ is a positive
  integer with $J>\eta$.

  \proofparagraph{Asymptotic Summation}
  For each integer $j$ with $0\le j\le J$, and by \eqref{eq:F-N-periodic} and
  Proposition~\ref{proposition:pseudo-Tauber:combination},
  a simple induction shows that
  there exist
  $1$-periodic continuous functions $\Psi_{jk}$ for $j\ge 0$ and $0\le k<m$ such that
  \begin{equation}\label{eq:S-j+1-f-asymptotic}
    (\calS^{j+1} f)(N) = N^{\kappa+j}\sum_{0\le k<m} (\log_q N)^k
    \Psi_{jk}(\fractional{\log_q N}) + \Oh{N^{\kappa_0+j}}
  \end{equation}
  for integers $N\to\infty$. In fact, $\Psi_{0k}=\Phi_k$ for
  $0\le k<m$. For $j\ge 1$ and $0\le k<m$, $\Psi_{jk}$ is continuously differentiable.

  We denote the corresponding Fourier coefficients by
  \begin{equation*}
    \psi_{jk\ell}\coloneqq \int_{0}^1 \Psi_{jk}(u)\exp(-2\ell\pi i u)\,\dd u
  \end{equation*}
  for $0\le j\le J$, $0\le k<m$, $\ell\in\Z$. We also set $\psi_{jm\ell}=0$ for $0\le
  j\le J$ and $\ell\in\Z$. By
  \eqref{eq:Fourier-recurrence}, the Fourier coefficients satisfy the
  recurrence relation
  \begin{equation}\label{eq:Fourier:Fourier-coefficient-recursion}
    \psi_{jk\ell} = (\kappa+j+1+\chi_\ell)\psi_{(j+1)k\ell} +
    \frac{(k+1)}{\log q} \psi_{(j+1)(k+1)\ell}
  \end{equation}
  for $0\le j<J$, $0\le k<m$ and $\ell\in\Z$.

  \proofparagraph{Explicit Summation}
  Explicitly, we have
  \begin{equation}\label{eq:S-j+1-explicit}
    (\calS^{j+1}f)(N) = \sum_{1\le n_1<n_2<\cdots<n_{j+1}<N}f(n_1)
    = \sum_{1\le n<N}f(n)\sum_{n<n_2<\cdots<n_{j+1}<N}1
  \end{equation}
  for $0\le j \le J$. Note that we formally write the outer sum over the range
  $1\le n<N$ although the inner sum is empty (i.e., equals~$0$) for $n\ge N-j$; this will be useful
  later on. The inner sum counts the number of selections of $j$ elements out
  of $\set{n+1,\ldots, N-1}$, thus we have
  \begin{equation}\label{eq:Fourier:explicit-summation}
    (\calS^{j+1}f)(N) = \sum_{1\le n< N}\binom{N-n-1}{j}f(n)=\sum_{1\le n< N}\frac1{j!}(N-n-1)^{\underline{j}}f(n)
  \end{equation}
  for $0\le j\le J$ and falling factorials
  $z^{\underline{j}}\coloneqq z(z-1)\dotsm (z-j+1)$.

  The polynomials $\frac1{j!}(X-1)^{\underline j}$, $0\le j\le J$, are clearly a basis of
  the space of polynomials in $X$ of degree at most $J$. Thus there exist
  rational numbers $b_0$, \ldots, $b_J$ such that
  \begin{equation*}
    \frac{X^J}{J!}=\sum_{j=0}^J \frac{b_j}{j!} (X-1)^{\underline{j}}.
  \end{equation*}
  Comparing the coefficients of  $X^J$ shows that $b_J=1$. Substituting
  $X\leftarrow N-n$, multiplication by $f(n)$ and summation over $1\le n<N$
  yields
  \begin{equation*}
    \frac1{J!}\sum_{1\le n<N}(N-n)^J f(n) = \sum_{j=0}^J b_j (\calS^{j+1}f)(N)
  \end{equation*}
  by~\eqref{eq:Fourier:explicit-summation}. When inserting the asymptotic
  expressions from \eqref{eq:S-j+1-f-asymptotic}, the summands for $0\le j<
  J$ are absorbed by the error term~$\Oh{N^{\kappa_0+J}}$
  of the summand for $j=J$ because $\Re\kappa - \kappa_0 < 1$. Thus
  \begin{equation}\label{eq:Mellin-Perron-sum}
    \frac1{J!}\sum_{1\le n<N}(N-n)^J f(n) =  N^{\kappa+J}\sum_{0\le k<m} (\log_q N)^k
    \Psi_{Jk}(\fractional{\log_q N}) + \Oh{N^{\kappa_0+J}}
  \end{equation}
  for $N\to\infty$.

  \proofparagraph{Mellin--Perron summation}
  By the $J$th order Mellin--Perron summation
  formula (see \cite[Theorem~2.1]{Flajolet-Grabner-Kirschenhofer-Prodinger:1994:mellin}), we have
  \begin{equation*}
    \frac1{J!}\sum_{1\le n<N}(N-n)^J f(n)  = \frac1{2\pi
      i}\int_{\sigma_a+1-i\infty}^{\sigma_a+1+i\infty} \frac{\calF(s) N^{s+J}}{s(s+1)\dotsm(s+J)}\,\dd s
  \end{equation*}
  with the arbitrary choice $\sigma_a+1>\sigma_a$ for the real part of
  the line of integration.
  The growth condition~\eqref{eq:Dirichlet-order} allows us to shift the
  line of integration to the left such that
  \begin{align*}
    \frac1{J!}\sum_{1\le n<N}(N-n)^J f(n) &=
    \sum_{\ell\in\Z}
    \Res[\Big]{\frac{\calF(s)N^{s+J}}{s(s+1)\dotsm (s+J)}}%
    {s=\kappa+\chi_\ell}\\
    &\phantom{=}\hspace*{0.65em}+ \frac{\calF(0)}{J!}N^J\iverson{\kappa_0<0}\iverson[\Big]{\kappa\notin \frac{2\pi i}{\log q}\Z}\\
    &\phantom{=}\hspace*{0.65em}+ \frac1{2\pi
      i}\int_{\kappa_0-i\infty}^{\kappa_0+i\infty} \frac{\calF(s) N^{s+J}}{s(s+1)\dotsm (s+J)}\,\dd s.
  \end{align*}
  The second term corresponds to a possible pole at $s=0$ which is not taken care of in the first sum; note that $\calF(s)$ is analytic at $s=0$
  by assumption because of~$\kappa_0<0$.
  We now compute the residue at $s=\kappa+\chi_\ell\eqqcolon s_\ell$. We use
  \begin{equation*}
    N^{s+J} = \sum_{k\ge 0}\frac1{k!}(\log N)^k N^{s_\ell+J} (s-s_\ell)^k
  \end{equation*}
  to split up the residue as
  \begin{equation*}
    \Res[\Big]{\frac{\calF(s)N^{s+J}}{s(s+1)\dotsm(s+J)}}{s=s_\ell} =
    \sum_{0\le k\le m}(\log_q N)^k N^{s_\ell+J} \xi_{k\ell}
  \end{equation*}
  for
  \begin{equation}\label{eq:Fourier:xi-as-residue}
    \xi_{k\ell} = \frac{(\log q)^k}{k!}
    \Res[\Big]{\frac{\calF(s)(s-s_\ell)^k}{s(s+1)\dotsm(s+J)}}{s=s_\ell}.
  \end{equation}
  Note that we allow $k=m$ for the case of $\kappa\in\frac{2\pi i}{\log q}\Z$
  in which case $\calF(s)/s$ might have a pole of order $m+1$ at $s=0$.
  Using the growth condition~\eqref{eq:Dirichlet-order} and the
  choice of~$J$ yields
  \begin{equation}\label{eq:Fourier:growth-frac}
    \frac{\calF(s)}{s(s+1)\dotsm(s+J)}
    = \Oh[\big]{\abs{\Im s}^{-1-J+\eta}} = \oh[\big]{\abs{\Im s}^{-1}}
  \end{equation}
  for $\abs{\Im s}\to\infty$ and $s$ which are at least a distance~$\delta$
  away from the poles~$s_\ell$.
  By writing the residue in~\eqref{eq:Fourier:xi-as-residue}
  in terms of an integral over a rectangle around
  $s=s_\ell$ (distance again at least~$\delta$ away from $s_\ell$),
  we see that \eqref{eq:Fourier:growth-frac} implies
  \begin{equation}\label{eq:Fourier:psi-growth}
    \xi_{k\ell} = \Oh[\big]{\abs{\ell}^{-1-J+\eta}} = \oh[\big]{\abs{\ell}^{-1}}
  \end{equation}
  for $\abs{\ell}\to\infty$, as well. Moreover,
  by~\eqref{eq:Fourier:growth-frac}, we see that
  \begin{equation*}
    \frac1{2\pi i} \int_{\kappa_0-i\infty}^{\kappa_0+i\infty}
    \frac{\calF(s) N^{s+J}}{s(s+1)\dotsm(s+J)}\,\dd s
    = \Oh{N^{\kappa_0+J}}.
  \end{equation*}

  Thus we proved that
  \begin{multline}\label{eq:calculate-Fourier-first}
    \frac1{J!}\sum_{1\le n<N}(N-n)^J f(n) = N^{\kappa+J}\sum_{0\le k\le m} (\log_q N)^k
    \Xi_k(\log_q N) \\+ \frac{\calF(0)}{J!}N^J\iverson{\kappa_0<0}\iverson[\Big]{\kappa\notin \frac{2\pi i}{\log q}\Z}+ \Oh{N^{\kappa_0+J}}
  \end{multline}
  for
  \begin{equation}\label{eq:Psi-tilde-k-definition}
    \Xi_k(u) =\sum_{\ell\in\Z}\xi_{k\ell} \exp(2\ell\pi i u)
  \end{equation}
  where the $\xi_{k\ell}$ are given in \eqref{eq:Fourier:xi-as-residue}.
  By \eqref{eq:Fourier:psi-growth}, the Fourier series
  \eqref{eq:Psi-tilde-k-definition} converges uniformly and absolutely. This
  implies that $\Xi_k$  is a $1$-periodic continuous function.

  \proofparagraph{Fourier Coefficients}
  By
  \eqref{eq:Mellin-Perron-sum}, \eqref{eq:calculate-Fourier-first} and Lemma~\ref{lemma:uniqueness-fluctuations}, we see
  that $\Xi_k=\Psi_{Jk}$. This immediately implies that $\calF(0)=0$ if
  $\kappa_0<0$ and $\kappa\notin\frac{2\pi i}{\log q}\Z$.
  We claim that
  \begin{equation}\label{eq:Fourier-coefficients-as-residue}
    \psi_{jk\ell} = \frac{(\log q)^k}{k!}
    \Res[\Big]{\frac{\calF(s)(s-s_\ell)^k}{s(s+1)\dotsm (s+j)}}{s=s_\ell}
  \end{equation}
  holds for $0\le j\le J$, $0\le k<m$, $\ell\in\Z$. We prove \eqref{eq:Fourier-coefficients-as-residue} by backwards
  induction on~$j$. For $j=J$, Equation~\eqref{eq:Fourier-coefficients-as-residue} is a
  restatement of \eqref{eq:Fourier:xi-as-residue} since
  $\psi_{Jk\ell}=\xi_{k\ell}$ because of $\Psi_{Jk}=\Xi_k$. Assume that
  \eqref{eq:Fourier-coefficients-as-residue} holds for some $j+1$. Then
  \eqref{eq:Fourier:Fourier-coefficient-recursion},
  \eqref{eq:Fourier-coefficients-as-residue} for $j+1$ and linearity of the residue imply that
  \begin{equation*}
    \psi_{jk\ell}=\frac{(\log
      q)^k}{k!}\Res[\Big]{\frac{\calF(s)(s-s_\ell)^k\bigl((s_\ell+j+1) +
      (s-s_\ell)\bigr)}{s(s+1)\dotsm (s+j)(s+j+1)}}{s=s_\ell}
  \end{equation*}
  and \eqref{eq:Fourier-coefficients-as-residue} follows for $j$.
  As for $0\le k<m$, $\Phi_{k}=\Psi_{0k}$ implies $\varphi_{k\ell}=\psi_{0k\ell}$
  for all $\ell\in\Z$, \eqref{eq:Fourier-coefficient} follows.
\end{proof}

\section{Proof of Theorem~\ref{theorem:simple}}\label{section:proof-theorem-simple}
\begin{proof}[Proof of Theorem~\ref{theorem:simple}]
  By Remark~\ref{remark:regular-sequence-as-a-matrix-product}, we have
  $x(n)=e_1 f(n)v(0)$. If $v(0)=0$, there is nothing to show.
  Otherwise, as observed in Appendix~\ref{section:q-regular-matrix-product},
  $v(0)$ is a right eigenvector of $A_0$ associated to the eigenvalue $1$.
  As a consequence, $Kv(0)$, $\vartheta_m v(0)$ and $\vartheta v(0)$ all vanish.
  Therefore, \eqref{eq:formula-X-n} follows from Corollary~\ref{corollary:main}
  by multiplication by $e_1$ and $v(0)$ from left and right, respectively.

  The functional equation \eqref{eq:functional-equation-V} follows from
  Theorem~\ref{theorem:Dirichlet-series} for $n_0=1$ by multiplication from right
  by $v(0)$.

  For computing the Fourier coefficients, we denote the rows of $T$ by $w_1$,
  \ldots, $w_d$. Thus $w_j$ is a generalised left eigenvector of $C$ of some
  order $m_j$ associated to some eigenvalue $\lambda_j$ of $C$. We can write
  $e_1=\sum_{j=1}^d c_j w_j$ for some suitable constants $c_1$, \ldots, $c_d$.
  For $1\le j\le d$, we consider the sequence  $(h_j(n))_{n\ge 1}$
  with
  \begin{equation*}
    h_j(n)=w_j\bigl(v(n)+v(0)\iverson{n=1}\bigr).
  \end{equation*}
  The reason for incorporating
  $v(0)$ into the value for $n=1$ is that the corresponding Dirichlet series $\calH^{(j)}(s)\coloneqq \sum_{n\ge
    1}n^{-s}h_j(n)$ only takes values at $n\ge 1$ into account. By definition, we
  have  $\calH^{(j)}(s)=w_jv(0) + w_j\calV(s)$. Taking the linear combination yields
  $\sum_{j=1}^dc_j\calH^{(j)}(s)=x(0) + \calX(s)$. We choose
  $\kappa_0>\max\set{-1, \log_q R}$ such that there are no eigenvalues
  $\lambda\in\sigma(C)$ with $\max\set{-1, \log_q R}<\log_q\lambda\le \kappa_0$ and such
  that $\kappa_0\neq 0$.

  By
  Theorem~\ref{theorem:contribution-of-eigenspace}, we have
  \begin{equation}\label{eq:simple:sum-lambda_j}
    \sum_{1\le n<N}h_j(n) = N^{\log_q \lambda_j}\sum_{0\le k<m_j}(\log_q N)^k
    \Psi_{jk}(\fractional{\log_q N}) + \Oh{N^{\kappa_0}}
  \end{equation}
  for $N\to\infty$ for suitable 1-periodic Hölder-continuous functions $\Psi_{jk}$
  (which vanish if $\abs{\lambda_j}\le R$). By
  Theorem~\ref{theorem:Dirichlet-series}, the Dirichlet
  series $\calH^{(j)}(s)$ is meromorphic for $\Re s>\kappa_0$ with possible
  poles at $s=\log_q \lambda_j + \chi_\ell$.

  We claim that
  \begin{equation}\label{claim-H-j=0}
    \calH^{(j)}(0) = 0
  \end{equation}
  for $1\le j\le d$ if $\kappa_0<0$ and $\lambda_j\neq 1$.

  We first prove the claim for the case that $\abs{\lambda_j}>1/q$. In that
  case,
  the sequence $(h_j(n))_{n\ge 1}$ satisfies
  the prerequisites of Theorem~\ref{theorem:use-Mellin--Perron}, either with
  $\kappa=\log_q \lambda_j$ if $\Re \log_q \lambda_j>\kappa_0$ or with
  arbitrary real $\kappa>\kappa_0$ and $\Phi_k=0$. The theorem then implies \eqref{claim-H-j=0}.

  We now turn to the proof of \eqref{claim-H-j=0} for the case that
  $\abs{\lambda_j} \le 1/q$. We have
  $\sum_{1\le n<N}h_j(n)=o(1)$
  as $N\to\infty$ by~\eqref{eq:simple:sum-lambda_j} and therefore
  $\calH^{(j)}(0)=0$ because $\calH^{(j)}(s)$ is analytic for $s=0$. This
  concludes the proof of \eqref{claim-H-j=0}.

  If $\abs{\lambda_j}>\max\set{R, 1/q}$, we apply
  Theorem~\ref{theorem:use-Mellin--Perron} on the sequence $(h_j(n))_{n\ge 1}$
  and obtain
  \begin{equation*}
    \Psi_{jk}(u)=\sum_{\ell\in\Z}\psi_{jk\ell}\exp(2\pi i\ell u)
  \end{equation*}
  for
  \begin{equation*}
    \psi_{jk\ell} = \frac{(\log q)^k}{k!}
    \Res[\Big]{\frac{\calH^{(j)}(s)\bigl(s-\log_q\lambda_j-\chi_\ell\bigr)^k}{s}}%
    {s=\log_q \lambda_j + \chi_\ell}.
  \end{equation*}

  We now have to relate the results obtained for the sequences $h_j$ with the
  results claimed for the original sequence $f(n)$.

  For $\lambda\in\sigma(C)$ with $\abs{\lambda}>\max\set{R, 1/q}$, we have
  \begin{align}
    \Phi_{\lambda k}(u)&=\sum_{\substack{1\le j\le d\\\lambda_j=\lambda}}c_j\Psi_{jk}(u)\notag\\
    \intertext{and}
    \varphi_{\lambda k\ell}
    &= \sum_{\substack{1\le j\le d\\\lambda_j=\lambda}} c_j \psi_{jk\ell} \notag\\
    &=\frac{(\log q)^k}{k!}
    \Res[\bigg]{\sum_{\substack{1\le j\le d\\\lambda_j=\lambda}} \frac{c_j\calH_j(s)
      \bigl(s-\log_q\lambda_j-\chi_\ell\bigr)^k}{s}}%
    {s=\log_q \lambda + \chi_\ell}.\label{eq:residue-with-condition}
  \end{align}

  In order to prove \eqref{eq:Fourier-coefficient:simple}, we need to remove the
  condition $\lambda_j=\lambda$ in \eqref{eq:residue-with-condition}. This is
  trivial for $\lambda\neq 1$ because in that case, $\calH_j(s)/s$ is analytic
  in $s=\log_q\lambda+\chi_\ell$ for $\lambda_j\neq\lambda$. If $\lambda=1$, the assumptions on $\lambda$
  imply that $R<1$ and therefore $\kappa_0<0$. In this case, we use
  \eqref{claim-H-j=0} to see that $\calH_j(s)/s$ is analytic for $\lambda_j\neq
  1$, too.
\end{proof}

It might seem to be somewhat artificial that
Theorem~\ref{theorem:use-Mellin--Perron} is used to prove that
$\calH^{(j)}(0)=0$ in some of the cases above. In fact, this can also be shown
directly using the linear representation.

\begin{remark}
  With the notations of the proof of Theorem~\ref{theorem:simple},
  $\calH^{(j)}(0)=0$ if $\lambda_j\neq 1$ and $R<1$ can also be shown using the
  functional equation~\eqref{eq:functional-equation-V}.
\end{remark}
\begin{proof}
  We prove this by induction on $m_j$. By definition of $T$, we have
  $w_j(C-\lambda_j I)=\iverson{m_j>1}w_{j+1}$. (We have $m_d=1$ thus $w_{d+1}$
  does not actually occur.)
  If $m_j>1$, then $\calH^{(j+1)}(0)=0$ by induction hypothesis.

  We add $(I-q^{-s}C)v(0)$ to \eqref{eq:functional-equation-V} and get
  \begin{align*}
    (I-q^{-s}C)(v(0)+\calV(s)) = (I-q^{-s}C)v(0)
    &+ \sum_{n=1}^{q-1}n^{-s}v(n) \\
    &+ q^{-s}\sum_{r=0}^{q-1}A_r
    \sum_{k\ge 1}\binom{-s}{k}\Bigl(\frac r q\Bigr)^k \calV(s+k).
  \end{align*}
  Multiplication by $w_j$ from the left yields
  \begin{align*}
    (1-q^{-s}\lambda)\calH^{(j)}(s) &= \iverson{m_j>1}\,q^{-s}\calH^{(j+1)}(s) \\
    &\phantom{{}={}}+ w_j (I - q^{-s}C)v(0)
    + w_j\sum_{n=1}^{q-1}n^{-s}v(n) \\
    &\phantom{{}={}}+ w_jq^{-s}\sum_{r=0}^{q-1}A_r
    \sum_{k\ge 1}\binom{-s}{k}\Bigl(\frac r q\Bigr)^k \calV(s+k).
  \end{align*}
  As $R<1$ and $\lambda_j\neq 1$, the Dirichlet series $\calH^{(j)}(s)$ is
  analytic in $s=0$ by Theorem~\ref{theorem:Dirichlet-series}. It is therefore
  legitimate to set $s=0$ in the above equation. We use the induction hypothesis that
  $\calH^{(j+1)}(0)=0$ as well as the fact that $v(n)=A_nv(0)$
  (note that $v(0)$ is a right eigenvector of $A_0$ to the eigenvalue~$1$;
  see Appendix~\ref{section:q-regular-matrix-product})
  for $0\le n<q$ to get
  \begin{equation*}
    (1-\lambda)\calH^{(j)}(0)=w_j\sum_{n=0}^{q-1}A_n v(0) -w_jCv(0) = 0
  \end{equation*}
  because all binomial coefficients $\binom{0}{k}$ vanish.
\end{proof}

\section{Appendix to Sequences Defined by Transducer Automata}
\label{appendix:transducers}
\movepfa
\movepfb
\begin{proof}
  As usual, the condensation of $D$ is
  the graph resulting from contracting each component of the original digraph
  to a single new vertex. By construction, the condensation is acyclic.

  We choose a refinement of the partial order of the components given by the successor relation in
  the condensation to a linear order in such a way that the final components
  come last. Note that this implies that if there is an arc from one component to another, the
  former component comes before the latter component in our linear order. We
  then denote the components by $\calC_1$, \ldots, $\calC_k$, $\calC_{k+1}$,
  \ldots, $\calC_{k+\ell}$ where the the first $k$ components are non-final and
  the last $\ell$ are final.
  W.l.o.g., we assume that the vertices of the original digraph $D$ are labeled such that
  vertices within a component get successive labels and such that the linear
  order of the components established above is respected.

  Therefore, the adjacency matrix $M$ is an upper block triagonal matrix of the
  shape
  \begin{equation*}
    M=\left(
    \begin{array}{c|c|c|c|c|c}
      M_1&\star&\star&\star&\star&\star\\\hline
      0&\ddots&\star&\star&\star&\star\\\hline
      0&0&M_k&\star&\star&\star\\\hline
      0&0&0&M_{k+1}&0&0\\\hline
      0&0&0&0&\ddots&0\\\hline
      0&0&0&0&0&M_{k+\ell}
    \end{array}
    \right)
  \end{equation*}
  where $M_j$ is the adjacency matrix of the component $\calC_j$.

  Each row of the non-negative square matrix $M$ has sum $q$ by construction.
  Thus $\inftynorm{M}=q$ and therefore the spectral radius of $M$ is bounded
  from above by $q$. As the all ones vector is obviously a right eigenvector
  associated with the eigenvalue $q$ of $M$, the spectral radius of $M$ equals
  $q$. The same argument applies to $M_{k+1}$, \ldots,
  $M_{k+\ell}$.

  By construction, the
  matrices $M_{k+1}$, \ldots, $M_{k+\ell}$ are irreducible.
  For $1\le j\le \ell$ all eigenvalues $\lambda$ of $M_{k+j}$ of modulus $q$ have
  algebraic and geometric multiplicities $1$ by Perron--Frobenius theory
  and $\lambda = q \zeta$ for some $p_{k+j}$th root of unity $\zeta$ where $p_{k+j}$ is
  the period of $\calC_{k+j}$.

  By construction, the vertices of the components $\calC_j$ for $1\le j\le k$
  have out-degree at most $q$. We add loops to these vertices to increase
  their out-degree to $q$, resulting in $\tildecalC_j$. The corresponding
  adjacency matrices are denoted by $\tildeM_j$. By the above argument,
  $\tildeM_j$ has spectral radius $q$ for $1\le j\le k$. As $M_j\le \tildeM_j$
  and $M_j\neq \tildeM_j$ by construction, the spectral radius of $M_j$ is strictly less than
  $q$ by \cite[Theorem~8.8.1]{Godsil-Royle:2001:alggraphtheory}.

  A left eigenvector $v_j$ of $M_{k+j}$ for $1\le j\le \ell$ can easily be
  extended to a left eigenvector $(0, \ldots, 0, v_j, 0, \ldots, 0)$ of
  $M$. This observation shows that the geometric multiplicity of any eigenvalue
  of $M$ of modulus $q$ is at least its algebraic multiplicity. This concludes
  the proof.
\end{proof}

\begin{proof}[Proof of Corollary~\ref{corollary:transducer} (Continued)]
\movetransducer
\end{proof}

\section{Appendix to Pascal's Rhombus}
\label{appendix:pascal}

\subsection{Recurrence Relations}
\label{appendix:pascal:recurrence}
\movepascalrecurrences
We obtain \eqref{eq:rec-pascal-rhombus:main}.

\moveproofpascalmeromorphic

\movesectionpascalfourier

\movesectionpascalexplicitbounds

\end{document}